\documentclass{amsart}

\usepackage{amssymb}
\usepackage{amsmath}
\usepackage{amsthm}
\usepackage{amsbsy}
\usepackage{bm}
\usepackage{hyperref}
\date{\today}
\usepackage{cite}
\usepackage{array}
\usepackage{xcolor}
\usepackage{tikz}
\usepackage{graphics}
\usepackage{subcaption}
\makeatletter
\newcommand{\distas}[1]{\mathbin{\overset{#1}{\kern\z@\sim}}}

\newcommand{\distras}[1]{
  \mathbin{\overset{#1}{\kern\z@\resizebox{\wd\mybox}{\ht\mysim}{$\sim$}}}
}

\theoremstyle{theorem}
    \newtheorem{theorem}{Theorem}
    \newtheorem{lemma}[theorem]{Lemma}
    \newtheorem{proposition}[theorem]{Proposition}
    \newtheorem{corollary}[theorem]{Corollary}

\theoremstyle{definition} % For roman text in the body
    \newtheorem{definition}[theorem]{Definition}
    
    \newtheorem{result}[theorem]{Result}
    \newtheorem{remark}[theorem]{Remark}
    \newtheorem{example}[theorem]{Example}
    \newtheorem{exercise}[theorem]{Exercise}

\def\suchthat{\; : \;}

\def\sig{\sigma}

\def\tends{\rightarrow}

\def\l{\left}
\def\r{\right}
\def\<{\langle}
\def\>{\rangle}

\newcommand{\E}{\mbox{E}}

\def\P{\mbox{P}}
\newcommand\Tr{{\mbox{Tr}}}

\newcommand\mnote[1]{} %off
\newcommand\be{\begin{equation*}}

\newcommand\ee{\end{equation*}}

\newcommand\ben{\begin{equation}}
\newcommand\een{\end{equation}}
\newcommand\bes{\begin{eqnarray*}}
\newcommand\ees{\end{eqnarray*}}

\newcommand\bex{\begin{exercise}}
\newcommand\eex{\end{exercise}}
\newcommand\beg{\begin{example}}
\newcommand\eeg{\end{example}}
\newcommand\benu{\begin{enumerate}}
\newcommand\eenu{\end{enumerate}}
\newcommand\beit{\begin{itemize}}
\newcommand\eeit{\end{itemize}}
\newcommand\berk{\begin{remark}}
\newcommand\eerk{\end{remark}}
\newcommand\bdefn{\begin{defintion}}
\newcommand\edefn{\end{definition}}
\newcommand\bthm{\begin{theorem}}
\newcommand\ethm{\end{theorem}}
\newcommand\bprf{\begin{proof}}
\newcommand\eprf{\end{proof}}
\newcommand\blem{\begin{lemma}}
\newcommand\elem{\end{lemma}}

\newcommand{\Cov}{\mbox{\rm Cov}}
\newcommand{\sm}{{\raise0.3ex\hbox{$\scriptstyle \setminus$}}}

\def\l{\left}
\def\r{\right}
\def\sig{\sigma}

\def\tends{\rightarrow}

\def\CHI{\mathchoice%
{\raise2pt\hbox{$\chi$}}%
{\raise2pt\hbox{$\chi$}}%
{\raise1.3pt\hbox{$\scriptstyle\chi$}}%
{\raise0.8pt\hbox{$\scriptscriptstyle\chi$}}}
\def\smalloplus{\raise1pt\hbox{$\,\scriptstyle \oplus\;$}}

\title[Fluctuations of reverse circulant and symmetric circulant matrices]{Time dependent fluctuations of  linear eigenvalue statistics of some patterned matrices
}

\author{Arup Bose}
\address{Statistics and Mathematics Unit\\
Indian Statistical Institute \\
203 B. T. Road, Kolkata 700108, India}
\email{bosearu@gmail.com }

\author{Shambhu Nath Maurya}
\address{Department of Mathematics\\
	Indian Institute of Technology Bombay\\
	Powai, Mumbai, Maharashtra 400076, India}

\email{snmaurya [at] math.iitb.ac.in}

\author{Koushik Saha}
\address{Department of Mathematics\\
        Indian Institute of Technology Bombay\\
         Powai, Mumbai, Maharashtra 400076, India}

\email{koushik.saha [at] iitb.ac.in}

\date{\today}
\thanks{The work of Shambhu Nath Maurya is partially supported by UGC Doctoral Fellowship, India and the work of Koushik Saha is partially supported by MATRICS grant of SERB, Department of Science and Technology, Government of India. The work of Arup Bose is supported by J.C. Bose National Fellowship, Government of India.}
\begin{document}

\begin{abstract}
Consider the $n \times n$  reverse circulant $RC_n(t)$  and symmetric circulant $SC_n(t)$ matrices with independent Brownian motion entries. We discuss the process convergence of the time dependent fluctuations of linear eigenvalue statistics of these matrices 
%$RC_n(t)$ and $SC_n(t)$ 
as $n \tends \infty$, when the test functions of the statistics are polynomials. The proofs are mainly combinatorial, based on the trace formula,
% for $RC_n(t)$ and $SC_n(t)$, 
method of moments and some results on process convergence.
\end{abstract}

\maketitle

\section{ introduction and main results}
Let $M_n$ be an $n\times n$ random matrix with real or complex entries and $\phi$ is some  fixed function. 
%The function $\phi$ is known as a test function. 
%The  fluctuation of linear statistics of eigenvalues is one of the main areas of research in random matrix theory. 
%It describes the asymptotic distribution of eigenvalues $\{  \lambda_i\}_{i=1}^n$ of any random matrix of large size $n$. 
The linear statistics of eigenvalues $\lambda_1,\lambda_2,\ldots, \lambda_n$ of $M_n$ is then defined as 
\begin{equation} \label{eqn:0}
\tau_n(\phi) = \frac{1}{n}\sum_{k=1}^{n}\phi(\lambda_k).
\end{equation}
 The fluctuation of $\tau_n(\phi)$ was first studied by Arharov \cite{arharov} in 1971 for sample covariance matrices. Since then, this has been studied for various random matrices. For results on Wigner and sample covariance matrices,  see  \cite{johansson1998}, \cite{bai2004clt}, \cite{lytova2009central}, \cite{shcherbina2011central} and the references therein.  For Toeplitz and related matrices, see  \cite{chatterjee2009fluctuations}, \cite{liu2012}, \cite{li_sun_2015}, \cite{maurya2019process} and \cite{m&s_toeplitz_2020}, and for band and sparse  random matrices, see  \cite{anderson2006clt}, \cite{li2013central}, \cite{jana2014} and \cite{shcherbina2015}. 
%and for recent results, see [adhikari, jana and saha 2020] and \cite{m&s_toeplitz_2020}.

There are different techniques for studying the fluctuations of $\tau_n(\phi)$. One  obvious way is by using the trace formula of the matrices. Note that if the test function $\phi$ is a polynomial, then 
\begin{equation*} \label{eqn:1}
\tau_n(\phi) = \frac{1}{n}\sum_{k=1}^{n} \Tr[\phi(M_n)].
\end{equation*}
So, if a closed form of $\Tr[\phi(M_n)]$
%, trace of $\phi(M_n)$ 
is known then one can exploit it to study the limiting behaviour of $\tau_n(\phi)$.
%the linear eigenvalue statistics.
% is easily tractable. 
%This is also a reason that why trace formula of polynomials of random matrices are more important.
 In articles \cite{soshnikov1998tracecentral}, \cite{liu2012}, \cite{li_sun_2015}, \cite{adhikari_saha2017}, \cite{adhikari_saha2018}, \cite{m&s_toeplitz_2020}, the authors have studied the fluctuations of $\tau_n(\phi)$ by using the trace formula of the matrices.

%The study of  random matrices with time dependent entries is quite interesting. 
%So, one can also be interested in fluctuation of linear eigenvalues statistics of time dependent random matrices. 
Suppose $M_n(t)$ is a time dependent random matrix with $t \in T$, where $T$ is an index set. Then the linear statistics of eigenvalues of $M_n(t)$ will also be time dependent, say $\tau_n(\phi, t)$.
So, the obvious questions that arise are about the joint fluctuations and/or tightness of the processes $\{\tau_n(\phi, t); t\in T\}_{n\geq 1}$. 

%The fluctuation problem in which we are interested  is inspired by \cite{li_sun_2015}, \cite{maurya2019process} and \cite{m&s_toeplitz_2020}.
In \cite{li_sun_2015}, Li and Sun have studied the joint fluctuations of the processes $\{\tau_n(\phi, t); t\geq 0\}_{n\geq 1}$ 
%with proper centring and scaling 
for band Toeplitz matrices but have not addressed  the tightness issues. Tightness was subsequently addressed 
%  of the processes $\{\tau_n(\phi, t); t\geq 0\}$ was established 
in \cite{m&s_toeplitz_2020}. 

The reverse circulant and symmetric circulant matrices are defined as follow:
\vskip5pt
\noindent \textbf{Reverse circulant matrix:} An $n\times n$ matrix with (i,j)-th element as $x_{(i+j-2) \bmod n}$ is called reverse circulant matrix. The time dependent reverse circulant matrix is defined as
\begin{equation*}\label{def:RC_n(t)}
RC_n(t)=\frac{1}{\sqrt n}\left(\begin{array}{cccccc}
b_0(t) & b_1(t) & b_2(t) & \cdots & b_{n-2}(t) & b_{n-1}(t) \\
b_{1}(t) & b_2(t) & b_3(t) & \cdots & b_{n-1}(t) & b_{0}(t)\\
\vdots & \vdots & {\vdots} & \vdots & {\vdots} & \vdots \\
b_{n-1}(t) & b_0(t) & b_1(t) & \cdots & b_{n-3}(t) & b_{n-2}(t)
\end{array}\right).
\end{equation*}
Note that for $j=1,2,\ldots, n-1$, its $(j+1)$-th row is obtained by giving its $j$-th row a left circular shift by one positions. 
%Note that the matrix is symmetric and the (i,\;j)-th element of the matrix is $x_{(i+j-1) \bmod  n}$.
% we write $x_n$ instead of $x_0$.
\vskip5pt
\noindent\textbf{Symmetric circulant matrix:} An $n\times n$ matrix with (i,j)-th element as $x_{\frac{n}{2}-|\frac{n}{2}-|i-j||}$ is called symmetric circulant matrix. The time dependent symmetric circulant matrix is defined as
\begin{equation*}\label{def:SC_n(t)}
SC_n(t)=\frac{1}{\sqrt n}\left(\begin{array}{cccccc}
b_0(t) & b_1(t) & b_2(t) & \cdots & b_{2}(t) & b_{1}(t) \\
b_{1}(t) & b_0(t) & b_1(t) & \cdots & b_{3}(t) & b_{2}(t)\\
\vdots & \vdots & {\vdots} & \vdots & {\vdots} & \vdots \\
b_1(t) & b_2(t) & b_3(t) & \cdots & b_{1}(t) & b_{0}(t)
\end{array}\right).
\end{equation*}
For $j=1,2,\ldots, n-1$, its $(j+1)$-th row is obtained by giving its $j$-th row a right circular shift by one positions. 
%Also note that the symmetric circulant matrix is a Toeplitz matrix with the restriction that $x_{n-j}=x_j$.

In this article we study, for these two matrices,  the joint fluctuation and tightness of $\{\tau_n(\phi, t); t\geq 0\}_{n\geq 1}$
%  for time dependent versions of these two matrices  with 
with polynomial test functions. We assume that $\{b_n(t);t\geq 0\}_{n\geq 0}$ is a sequence of independent standard Brownian motions. 
%The entries  $\{\frac{b_0(t)}{\sqrt n},\frac{b_1(t)}{\sqrt n},\ldots,\frac{b_{n-1}(t)}{\sqrt n}\}$ are assumed to be independent standard Brownian motions. 

For $RC_n(t)$ with $\phi(x)=x^{2p}$, $p\geq 1$, observe that
$$ \sum_{k=1}^{n} \phi (\lambda_k(t))= \sum_{k=1}^{n}(\lambda_k(t))^{2p}= \Tr(RC_n(t))^{2p},$$
where $\lambda_1(t),\lambda_2(t),\ldots,\lambda_n(t)$ are the eigenvalues of $RC_n(t)$. We scale and centre $\Tr(RC_n(t))^{2p}$, to study its fluctuation, and define 
\begin{equation}\label{eqn:w_p(t)_rp}
   w_{p}(t) := \frac{1}{\sqrt{n}} \bigl\{ \Tr(RC_n(t))^{2p} - \E[\Tr(RC_n(t))^{2p}]\bigr\}. 
   \end{equation}
Note that here we have considered only even degree monomials as test functions for $RC_n$. Odd degree monomials will be dealt separately. See Remark \ref{rem:RC_odd}.  

Similarly, for $SC_n(t)$ with $\phi(x)=x^{p}$, $p\geq 1$, define
\begin{equation}\label{eqn:w_p(t)_sp}
   \eta_p(t) := \frac{1}{\sqrt{n}} \bigl\{ \Tr(SC_n(t))^p - \E[\Tr(SC_n(t))^p]\bigr\}. 
   \end{equation}
%For a given real polynomial  $$Q(x)=\displaystyle\sum_{k=1}^da_kx^{k}$$ 
%of degree $d$ where $d\geq 2$, we define 
%\begin{equation}\label{eqn:SCw_Q}
%\eta_Q := \frac{1}{\sqrt{n}} \bigl\{ \Tr(Q(SC_n)) - \E[\Tr(Q(SC_n))]\bigr\}. 
%\end{equation}
We have suppressed the  dependence of $w_p(t)$ and $\eta_p(t)$  on $n$ to keep the notation simple. 

In our first result, we calculate the covariance between $w_p(t_1)$ and $w_q(t_2)$ as $n \tends \infty$. 
\begin{theorem}\label{thm:revcovar_rp} 
%Suppose  $C_n(t)$ is a random circulant matrix as defined in \eqref{def:C_n(t)}. Then 
For $0<t_1\leq t_2$ and $p,q\geq 1$, 
\begin{equation}\label{eqn:cov_rp}
 \lim_{n\to \infty} \Cov\big(w_{p}(t_1),w_{q}(t_2)\big)  = \sum_ {r'=1}^{q} \binom{2q}{2r'} (t_1)^{p+r'} (t_2-t_1)^{q-r'} \Big\{\sum_{k=1}^{\min\{p,r'\}}  c_k g(k) - c_1\Big\} .
 \end{equation}
where
\begin{align}\label{eqn:c_k,g(k)_rp}
 c_k & =\l(\binom{p}{p-k}^2(p-k)! \binom{r'}{r'-k}^2(r'-k)!\binom{q}{q-r'}^2(q-r')!\r),\\
 g(k) & =\frac{1}{(2k-1)!}\sum_{s=-(k-1)}^{k-1}\sum_{j=0}^{k+s-1}(-1)^j\binom{2k}{j}(k+s-j)^{2k-1}(2-{\bf 1}_{\{s=0\}})k!k!. \nonumber
\end{align}
\end{theorem}
The following theorem describes the finite dimensional joint convergence of $\{w_{p}(t);p\geq 1\}$. 

\begin{theorem}\label{thm:revmulti_rp} 
%	As $n\to\infty$, $\{w_{p}(t);t\geq 0, p\geq 1\}$ jointly converge to a family of Gaussian processes $\{N_{p}(t);t\geq 0,  p\geq 1\}$ in the following sense. Suppose  $\{p_1, p_2, \ldots, p_r\}\subset \mathbb N$ and $ p_i\geq  1$ for $1\leq i\leq r$;  $0<t_1<t_2 \cdots <t_r$ and $\{a_1, a_2, \ldots, a_r \} \subset \mathbb{R}$. 
%	Then
%	\begin{align}\label{eqn:condition_rp}
%	\lim_{n\to \infty} & \P  \Big( w_{p_1}(t_1)  \leq a_1,       w_{p_2}(t_2)  \leq a_2,\ldots, w_{p_r}(t_r)  \leq a_r \Big) \nonumber \\
%	&= \P  \Big( N_{p_1}(t_1)  \leq a_1,                              N_{p_2}(t_2)  \leq a_2, \ldots, N_{p_r}(t_r)  \leq a_r \Big), 
%	\end{align}
	 Suppose  $\{p_1, p_2, \ldots, p_r\}\subset \mathbb N$ with $ p_i\geq  1$ for $1\leq i\leq r$ and  $0<t_1<t_2 \cdots <t_r$. Then 
	 %as $n \tends \infty$
	\begin{equation*}
	(w_{p_1}(t_1), w_{p_2}(t_2), \ldots , w_{p_r}(t_r)) \stackrel{\mathcal D}{\rightarrow} (N_{p_1}(t_1), N_{p_2}(t_2), \ldots , N_{p_r}(t_r)),\ \text{as}\ n \to \infty,
	\end{equation*}
	where 	$\{N_{p}(t);t\geq 0, p\geq 1\}$ are Gaussian processes with mean zero and covariance structure as described in the right side of \eqref{eqn:cov_rp}.
%The finite dimensional distributions of $\{N_{p}(t);t\geq 0, p\geq 1\}$ are Gaussian  with mean zero and covariance structure as described in the right side of \eqref{eqn:cov_rp}.
	%\begin{equation}\label{eqn:N_p(t):cov_rp}
	%\Cov \Big[ N_{p}(t_1), N_{q}(t_2) \Big]
	%=  \sum_ {r'=1}^{q} \binom{2q}{2r'} (t_1)^{p+r'} (t_2-t_1)^{q-r'} \Big\{\sum_{k=1}^{\min\{p,r'\}}  c_k g(k) - c_1\Big\} ,
	%\end{equation}
	%with $c_k$ and $g(k)$  are as in \eqref{eqn:c_k,g(k)_rp}.
\end{theorem}

%In the next result, we are interested in the process convergence of $\{ w_{p}(t) ; t \geq 0\}$ for reverse circulant matrices which
%The fluctuation problems, we are interested to consider for Band Toeplitz matrices, 
Inspired by the results in \cite{maurya2019process} and \cite{m&s_toeplitz_2020}, 
% where the authors have done the  process convergence for circulant matrices.
the following theorem describes the process convergence of $\{ w_{p}(t) ; t \geq 0\}$ for  $p \geq 1$.
 \begin{theorem}\label{thm:revprocess}
	Suppose $ p\geq 1 $. Then as $n\to\infty$
	\begin{equation*}
	\{ w_{p}(t) ; t \geq 0\} \stackrel{\mathcal D}{\rightarrow} \{N_{p}(t) ; t \geq 0\}, 
	\end{equation*}
	where $\{N_{p}(t);  t \geq 0\}$  is defined as in Theorem \ref{thm:revmulti_rp} and $\stackrel{\mathcal D}{\rightarrow}$ denotes the weak convergence of probability measures on $(C_\infty,\mathcal C_\infty)$, see \eqref{weak_convergence_C_infinity}.
	%  and $\stackrel{\mathcal D}{\rightarrow}$ denotes the weak convergence as in Definition \ref{def:processconvergence}.
\end{theorem}
\begin{remark} \label{rem:RC_odd}
In the above theorems for $RC_n(t)$,
% and Remark \ref{rem:poly_test_function_rp}, 
we have considered only even degree monomials. 
%\textbf{must be part of the statements of the three theorems} 
%and polynomial with even degree terms. 
The fluctuation of linear eigenvalue statistics of $RC_n(t)$ with odd degree monomial test function is quite  different from the even degree monomial test functions. For odd degree monomial the limit is not Gaussian except for the one degree monomial, $\phi(x)=x$. We have discussed the fluctuations for odd degree monomial and polynomial with odd degree terms  after the proof of the above theorems.  
\end{remark}
Now we state our main results for $SC_n(t)$. The following theorem describes the covariance structure of $\eta_p(t_1)$ and $\eta_q(t_2)$, as $n\to\infty$.
\begin{theorem}\label{thm:symcovar_sp} 
%Suppose  $C_n(t)$ is a random circulant matrix as defined in \eqref{def:C_n(t)}. Then 
For $0<t_1\leq t_2$ and $p,q\geq 2$, 
\begin{align} \label{eqn:cov_sp}
	\lim_{n\to\infty} & \Cov\big(\eta_p(t_1),\eta_q(t_2)\big) \nonumber \\ 
	& =  \left\{\begin{array}{ll} 	 
		 	\displaystyle \sum_ {r=2,4,}^{q} \binom{q}{r}  t_1^{ \frac{p+r}{2}} (t_2-t_1)^{ \frac{q-r}{2}} \Big\{  \frac{2a_1} {2^{\frac{p+q-4}{2}}  } + \\
		 	\qquad \displaystyle + \sum_{m=2}^{ \min\{ \frac{p}{2},\frac{r}{2} \} } \Big( \frac{a_m}{2^{\frac{p+r-4m}{2}}  } \sum_{s=0}^{2m}\binom{2m}{s}^2 s!(2m-s)! \ h_{2m}(s) \Big) \Big\} & \text{if}\ p, \mbox{} q \mbox{ both are even},\\\\
		 	\displaystyle \sum_ {r=1,3,}^{q} \binom{q}{r}  t_1^{ \frac{p+r}{2}} (t_2-t_1)^{ \frac{q-r}{2}}  \sum_{m=0}^{ \min\{ \frac{p-1}{2},\frac{r-1}{2} \} }  \frac{b_m}{2^{\frac{p+r-4m-2}{2}}  } \\
		 \qquad	\displaystyle \sum_{s=0}^{2m+1} \Big( \binom{2m+1}{s}^2 s!(2m+1-s)! \ h_{2m+1}(s) \Big)  \\
		  \displaystyle \qquad +  \frac{pq}{2^{(\frac{p+q}{2}-1)}} \sum_ {r=0,2,}^{q-1} \binom{q-1}{r}  t_1^{ \frac{p+1+r}{2}} (t_2-t_1)^{ \frac{q-1-r}{2}} d_r
		 & \text{if}\ p, \mbox{} q \mbox{ both are odd,} \\\\	 	
			0 & \text{otherwise}, 	 
		 	  \end{array}\right.	
	\end{align}	
where $a_m, b_m, h_m(s)$ and $d_{r}$ are appropriate constants, will be given in the proof of Theorem \ref{thm:symcovar_sp}.
\end{theorem}
The following theorem describes joint convergence of $\{\eta_p(t);p\geq 2\}$ as $n \to \infty$.
\begin{theorem}\label{thm:symmulti_sp} 
 Suppose  $\{p_1, p_2, \ldots, p_r\}\subset \mathbb N$ with $ p_i\geq  2$ for $1\leq i\leq r$ and  $0<t_1<t_2 \cdots <t_r$. Then, as $n \tends \infty$,
	\begin{equation*}
	(\eta_{p_1}(t_1), \eta_{p_2}(t_2), \ldots , \eta_{p_r}(t_r)) \stackrel{\mathcal D}{\rightarrow} (N_{p_1}(t_1), N_{p_2}(t_2), \ldots , N_{p_r}(t_r)),
	\end{equation*}	
	where $\{N_p(t);t\geq 0, p\geq 2\}$ are Gaussian processes with mean zero and  covariance structure as described in the right side of \eqref{eqn:cov_sp}.
%	following covariance structure:
%	\begin{align}\label{eqn:N_p(t):cov_sp}
%	& \Cov \Big[ N_p(t_1), N_{q}(t_2) \Big]  \\
%	& =  \left\{\begin{array}{ll} 	 
%		 	\displaystyle \sum_ {r=2,4,}^{q} \binom{q}{r}  t_1^{ \frac{p+r}{2}} (t_2-t_1)^{ \frac{q-r}{2}} \Big\{  \frac{2a_1} {2^{\frac{p+q-4}{2}}  } + \\
%		 	\qquad \displaystyle + \sum_{m=2}^{ \min\{ \frac{p}{2},\frac{r}{2} \} } \Big( \frac{a_m}{2^{\frac{p+r-4m}{2}}  } \sum_{s=0}^{2m}\binom{2m}{s}^2 s!(2m-s)! \ h_{2m}(s) \Big) \Big\} & \text{if}\ p, \mbox{} q \mbox{ both are even},\\\\
%		 	\displaystyle \sum_ {r=1,3,}^{q} \binom{q}{r}  t_1^{ \frac{p+r}{2}} (t_2-t_1)^{ \frac{q-r}{2}}  \sum_{m=0}^{ \min\{ \frac{p-1}{2},\frac{r-1}{2} \} }  \frac{b_m}{2^{\frac{p+r-4m-2}{2}}  } \\
%		 \qquad	\displaystyle \sum_{s=0}^{2m+1} \Big( \binom{2m+1}{s}^2 s!(2m+1-s)! \ h_{2m+1}(s) \Big)  \\
%		  \displaystyle \qquad +  \frac{pq}{2^{(\frac{p+q}{2}-1)}} \sum_ {r=0,2,}^{q-1} \binom{q-1}{r}  t_1^{ \frac{p+1+r}{2}} (t_2-t_1)^{ \frac{q-1-r}{2}} d_r
%		 & \text{if}\ p, \mbox{} q \mbox{ both are odd,} \\\\	 	
%			0 & \text{otherwise}, 	 
%		 	  \end{array}\right.	\nonumber
%	\end{align}
%	with $a_m, b_m, h_m(s)$ and $d_{r}$  are as in (\ref{eqn:cov_sp}).
\end{theorem}
The following theorem describes the process convergence of $\{ \eta_p(t) ; t \geq 0\}$ for  $p \geq 2$.
\begin{theorem}\label{thm:symprocess}
	Suppose $ p\geq 2 $. Then as $n\to\infty$
	\begin{equation*}
	\{ \eta_p(t) ; t \geq 0\} \stackrel{\mathcal D}{\rightarrow} \{N_p(t) ; t \geq 0\}, 
	\end{equation*}
	where $\{N_p(t);  t \geq 0\}$  is defined as in Theorem \ref{thm:symmulti_sp}.
%	 and $\stackrel{\mathcal D}{\rightarrow}$ denotes the weak convergence as in Definition \ref{def:processconvergence}. 	
\end{theorem}

%\begin{remark}
In the above theorems for $SC_n(t)$, we have considered the fluctuation of $\eta_p(t)$ for  $p\geq2$. 
For $p=1$,
\begin{align*}
\eta_1(t) = \frac{1}{\sqrt{n}} \bigl\{ \Tr(SC_n(t)) - \E[\Tr(SC_n(t))]\bigr\} = \frac{1}{\sqrt{n}} \big[n \frac{b_0(t)}{\sqrt n} - \E(n\frac{b_0(t)}{\sqrt n})\big] = b_0(t),
\end{align*}
as $\E(b_0(t))=0$. So $\eta_1$ is distributed as $b_0(t)$ and its distribution does not depend on $n$. So we ignore the case for $p=1$.  
%\end{remark}	

Here is a brief outline of the rest of the manuscript.  In Section \ref{sec:cov_rp} we prove Theorem \ref{thm:revcovar_rp}. 
%We recall the trace formula of $RC_n(t)$ and state some results  to prove Theorem \ref{thm:revcovar_rp}. 
In Section \ref{sec:joint convergence_rp} we prove Theorem \ref{thm:revmulti_rp} by using method of moments, Cramer-Wold device and Theorem \ref{thm:revcovar_rp} and finally, in  Section \ref{sec:process convergence_rp} we prove Theorem \ref{thm:revprocess} using Theorem \ref{thm:revmulti_rp}.
% and some standard results on process convergence. 
Similarly, in Section \ref{sec:cov_sp} we prove Theorem \ref{thm:symcovar_sp}, 
% using the trace formula of $SC_n(t)$ and state some results  to prove Theorem \ref{thm:symcovar_sp}. In 
in Section \ref{sec:joint convergence_sp} we prove Theorem \ref{thm:symmulti_sp} and in Section \ref{sec:process convergence_sp}, we complete the proof of Theorem \ref{thm:symprocess}.
% using \ref{thm:symmulti_sp} and some standard results on process convergence.
%We use  method of moments and  Wick's formula to prove Theorem  \ref{thm:revcirpoly}. 
% In Section \ref{sec:cov_symmetric} we prove Theorem \ref{thm:symcircovar}, and derive the trace formula of $SC_n$  and state some necessary results  to prove Theorem \ref{thm:symcircovar}. In Section \ref{sec:poly_symmetric} we prove Theorem \ref{thm:symcirpoly}. 

%We use  method of moments and  Wick's formula to prove Theorem \ref{thm:revmulti_rp} and Theorem  \ref{thm:symmulti_sp}. 
 
\section{Proof of Theorem \ref{thm:revcovar_rp}}\label{sec:cov_rp}
%We first define some notation which will be used   in the proof of Theorem \ref{thm:revcirpoly}.
%\begin{align}\label{def:A_2p}
%A_{2p}&=\big\{(i_1,\ldots,i_{2p})\in \mathbb N^{2p}\suchthat \sum_{k=1}^{2p}(-1)^ki_k=0 \mbox{ (mod $n$)}, 1\le i_1,\ldots,i_{2p}\le n\big\},\\
%A_{2p}'&=\big\{(i_1,\ldots,i_{2p})\in \mathbb N^{2p}\suchthat \sum_{k=1}^{2p}(-1)^ki_k=0 \mbox{ (mod $n$)}, 1\le i_1\neq i_2\neq \cdots\neq i_{2p}\le n\big\}, \nonumber\\
%A_{2p,s}&=\big\{(i_1,\ldots,i_{2p})\in \mathbb N^{2p}\suchthat \sum_{k=1}^{2p}(-1)^ki_k=sn,1\le i_1,\ldots,i_{2p}\le n\big\}, \nonumber\\
%A_{2p,s}'&=\big\{(i_1,\ldots,i_{2p})\in \mathbb N^{2p}\suchthat \sum_{k=1}^{2p}(-1)^ki_k=sn,1\le i_1\neq i_2\neq \cdots\neq i_{2p}\le n\big\}. \nonumber
%\end{align}
We first recall the trace formula of reverse circulant matrices with entries $\{x_i;i\geq 1\}$ from \cite{adhikari_saha2017}.
Let $e_1,\ldots,e_n$ be the standard unit vectors in $\mathbb R^n$, i.e., $e_i=(0,\ldots,1,\ldots, 0)^t$ ($1$ in $i$-th place).  Then we have 
	\begin{align*}
	(RC_n)e_i=\mbox{$i$-th column}=\sum_{i_1=1}^nx_{i_1}e_{i_1-i+1 \mbox{ mod $n$}},
	\end{align*}
	for $i=1,\ldots, n$ (we write $e_0=e_n$). Repeating the procedure we get 
%	\begin{align*}
%	(RC_n)^2e_i=\sum_{i_1,i_2=1}^nx_{i_1}x_{i_2}e_{i_2-i_1+i \mbox{ mod $n$}},
%	\end{align*}
%	for $i=1,\ldots, n$. Therefore in general we get
	\begin{align*}
	(RC_n)^{2p}e_i&=\sum_{i_1,\ldots,i_{2p}=1}^nx_{i_1}\ldots x_{i_{2p}}e_{i_{2p}-i_{2p-1}\cdots -i_1+i \mbox{ mod $n$}},
	\\(RC_n)^{2p+1}e_i&=\sum_{i_1,\ldots,i_{2p+1}=1}^nx_{i_1}\ldots x_{i_{2p+1}}e_{i_{2p+1}-i_{2p}\cdots +i_1-i+1 \mbox{ mod $n$}},
	\end{align*}
	for $i=1,\ldots, n$.  Therefore the trace of $(RC_n)^{2p}$ can be written as 
	\begin{align}\label{trace formula RC_n}
	\Tr[(RC_n)^{2p}]=\sum_{i=1}^ne_i^t(RC_n)^{2p}e_i=n\sum_{A_{2p}}x_{i_1}\ldots x_{i_{2p}},
	\end{align}
where 
\begin{equation}\label{def:A_2p}
A_{2p}=\big\{(i_1,\ldots,i_{2p})\in \mathbb N^{2p}\suchthat \sum_{k=1}^{2p}(-1)^ki_k=0 \mbox{ (mod $n$)}, 1\le i_1,\ldots,i_{2p}\le n\big\}.
\end{equation}

The following lemma is an easy consequence of the trace formula.
\begin{lemma}\label{lem:w_p_rp}
Suppose $0<t_1 \leq t_2$ and $p, q\geq 1$.
Then
\begin{align}\label{eqn:covariance_rp}
\lim_{n\to \infty} \Cov\big(w_{p}(t_1),w_{q}(t_2)\big)&=\lim_{n\to \infty} \frac{1}{n^{p+q-1}}  \sum_ {r=0}^{2q} \binom{2q}{r} \sum_{A_{2p}, A_{2q}} \Big\{\E[U_{I_{2p}}U_{J_r}] \E[V_{J_{r+1,2q}} ]  \\
&  \qquad - \E[ U_{I_{2p}}] \E[U_{J_r}]\E[ V_{J_{r+1,2q}}  ] \Big\}, \nonumber
\end{align}
where $U_{I_{2p}} = u_{i_1}\cdots u_{i_{2p}}$, $U_{J_r} = u_{j_1}\cdots u_{j_r}$ and $V_{J_{r+1,2q}} = v_{j_{r+1}}\cdots v_{j_{2q}}$ with
$u_i=b_i(t_1)$ and $ v_i=b_i(t_2)- b_i(t_1)$.
\end{lemma}
\begin{proof}
Define
\begin{align*}
 U:=RC_n(t_1)=  \frac{1}{\sqrt{n}} \Big(u_{j-i \;{(\mbox{mod $n$})}}\Big)_{i,j=1} ^ {n} \mbox{ and }
 V:=RC_n(t_2)-RC_n(t_1) =  \frac{1}{\sqrt{n}} \Big(v_{j-i \;{(\mbox{mod $n$})}}\Big)_{i,j=1} ^ {n}.
 \end{align*}
 Since $\E(w_{p}(t_1))=\E(w_{q}(t_2))=0$, we have 
\begin{align*}
 \Cov\big(w_{p}(t_1),w_{q}(t_2)\big)
 %&=\E[w_p(t_1) w_q(t_2)]\\
 &= \frac{1}{n} \Big\{ \E[\Tr(RC_n(t_1))^{2p} \Tr(RC_n(t_2))^{2q} ]- \E[\Tr(RC_n(t_1))^{2p}]\E[\Tr(RC_n(t_2))^{2q}]   \Big\} \\
 &= \frac{1}{n} \Big\{ \E[\Tr(U)^{2p}\Tr(U+V)^{2q} ]- \E[\Tr(U)^{2p}]\E[\Tr(U+V)^{2q}]   \Big\}. 
 \end{align*}
Now on using the trace formula (\ref{trace formula RC_n}), we get
% \begin{align*}
% \E[\Tr(U^p)]&= \E\Big[n \sum_{A_{p}} \frac{u_{i_1}}{ \sqrt{n}} \cdots \frac{u_{i_p}}{ \sqrt{n}} \Big]
% = \frac{1}{ n^ {\frac{p}{2}-1}} \E\Big[ \sum_{A_{p}} u_{i_1} \cdots u_{i_p}\Big],\\
%  \E[\Tr(U+V)^q]&= \E\Big[n \sum_{A_{q}} \frac{(u_{j_1} + v_{j_1})}{ \sqrt{n}} \cdots \frac{(u_{j_q} + v_{j_q})}{ \sqrt{n}} \Big] \\
%  &= \frac{1}{ n^ {\frac{q}{2}-1}}\E\Big[ \sum_{A_{q}} (u_{j_1} + v_{j_1}) \cdots  (u_{j_q} + v_{j_q})\Big].
%\end{align*}
% Therefore
\begin{align*}
\lim_{n\to \infty} \Cov\big(w_{p}(t_1),w_{q}(t_2)\big)&=  \lim_{n\to \infty} \frac{1}{ n^ {p+q-1}} \Bigg[ \E\Big\{ \Big( \sum_{A_{2p}} u_{i_1} \cdots u_{i_{2p}} \Big) \Big(   \sum_{A_{2q}} (u_{j_1} + v_{j_1}) \cdots  (u_{j_{2q}} + v_{j_{2q}})   \Big)   \Big\} \\
 &\qquad  - \E \Big( \sum_{A_{2p}} u_{i_1} \cdots u_{i_{2p}} \Big) \Big(\E\sum_{A_{2q}} (u_{j_1} + v_{j_1}) \cdots  (u_{j_{2q}} + v_{j_{2q}}) \Big)     \Bigg] \\
%  &= \frac{1}{ n^ {(\frac{p+q}{2}-1)}} \Bigg[ \E\Big\{  \sum_{A_{p}, A_{q}} u_{i_1} \cdots u_{i_p} (u_{j_1} + v_{j_1}) \cdots  (u_{j_q} + v_{j_q})   \Big\} \\
%  &\qquad  -   \sum_{A_{p} A_{q}}  \Big(\E(u_{i_1} \cdots u_{i_p})\E\big[ (u_{j_1} + v_{j_1}) \cdots  (u_{j_q} + v_{j_q})\big]\Big)     \Bigg] \\
%
% &= \frac{1}{ n^ {(\frac{p+q}{2}-1)}} \sum_{A_{p}, A_{q}} \Big[ \E\Big\{  u_{i_1} \cdots u_{i_p} (u_{j_1} + v_{j_1}) \cdots  (u_{j_q} + v_{j_q})   \Big\} \\
%  &\qquad  -   \Big(\E(u_{i_1} \cdots u_{i_p})\E\big[ (u_{j_1} + v_{j_1}) \cdots  (u_{j_q} + v_{j_q})\big]\Big)     \Big]. \\
%\end{align*}
%Hence
%\begin{align*}
  %\Cov\big(w_p(t_1),w_q(t_2)\big) 
  &=\lim_{n\to \infty} \frac{1}{n^{p+q-1}}  \sum_ {r=0}^{2q} \binom{2q}{r} \sum_{A_{2p}, A_{2q}} \Big\{\E[u_{i_1} \cdots u_{i_{2p}} u_{j_1} \cdots u_{j_r}v_{j_{r+1}} \cdots v_{j_{2q}} ] \\
  &\qquad -\E[ u_{i_1} \cdots u_{i_{2p}} ] \E[u_{j_1}\cdots u_{j_r} v_{j_{r+1}} \cdots v_{j_{2q}}  ]        \Big\} \nonumber \\
 &= \lim_{n\to \infty} \frac{1}{n^{p+q-1}}  \sum_ {r=0}^{2q} \binom{2q}{r} \sum_{A_{2p}, A_{2q}} \Big\{\E[U_{I_{2p}}U_{J_r}] \E[V_{J_{r+1,2q}} ]  \\
&  \qquad - \E[ U_{I_{2p}}] \E[U_{J_r}]\E[ V_{J_{r+1,2q}}  ] \Big\},
\end{align*}
 where for $p,q\geq 1$ and $0 \leq r \leq q$,
 \begin{align*}
 I_{2p}&= (i_1, i_2, \ldots, i_{2p}),\quad J_{2q}=(j_1,j_2, \ldots, j_{2q}),	\\
 J_r&= (j_1, j_2, \ldots, j_r),\quad J_{r+1,2q}= (j_{r+1}, j_{r+2}, \ldots, j_{2q}),\\
U_{I_{2p}} &= u_{i_1}\cdots u_{i_{2p}},\quad  U_{J_r} = u_{j_1}\cdots u_{j_r} \ \mbox{ and } \ V_{J_{r+1,2q}} = v_{j_{r+1}}\cdots v_{j_{2q}}.
\end{align*} 	
Also note that, $U_{J_0}=1$ and $V_{J_{2q+1,2q}}=1$.
This completes the proof of the lemma.
\end{proof}

Now we state a result which will be used in the proof of  Theorem \ref{thm:revcovar_rp}. 
We refer to \cite[Lemma 14]{adhikari_saha2017} for its proof.

\begin{result} \label{result:cardinalityA}
	Suppose 
	\begin{equation*}
	A_{2p,s}=\big\{(i_1,\ldots,i_{2p})\in \mathbb N^{2p}\suchthat \sum_{k=1}^{2p}(-1)^ki_k=sn,1\le i_1,\ldots,i_{2p}\le n\big\}.
	\end{equation*}
%	$|A_{2p,s}|$ denotes   the cardinality of $A_{2p,s}$. 
	Then the cardinality of $A_{2p,s}$ is  
	$$
	|A_{2p,s}|=\sum_{k=0}^{p+s-1}(-1)^k\binom{2p}{k}\binom{(p+s-k)n+p-1}{2p-1}, \;\;\mbox{ for $s=-(p-1),\ldots, 0,\ldots, p-1$.}
	$$ 	
\end{result}
%We will use this result to prove Theorem \ref{thm:revcircovar}.
Now for a given vector $(i_1, i_2, \ldots, i_p)\in\mathbb N^p$, we define a term called {\it odd-even pair matched} elements of the vector. 
\begin{definition}\label{def:odd-even_for_reverse}
	Suppose $v=(i_1, i_2, \ldots, i_p)$ is a vector in $\mathbb N^p$. Two elements $i_k, i_\ell$ of $v$ are said to be {\it odd-even pair matched} if $i_k=i_\ell$, $i_k$ appears exactly twice in $v$, once each at an odd and an even position. For example, in $(1,1,3,4)$, $1$ is {\it odd-even pair matched} whereas it is not so in $(1,2,1,3)$. 
	
	We shall call $v$ {\it odd-even pair matched} if all its entries are {\it odd-even pair matched}. 
	%In other words, each entry of $(i_1, i_2, \ldots, i_p)$ appears same number of times in odd and even positions. Note that, for a vector $(i_1, i_2, \ldots, i_p)$ to be {\it odd-even pair matched}, 
	In that case, $p$ is necessarily even.
\end{definition}
Observe that, if $(i_1,i_2,\ldots,i_{2p})\in A_{2p}$  and each entry of $\{i_1,i_2,\ldots,i_{2p} \}$ has multiplicity greater than or equal to two, then the maximum number of free variables in $(i_1,i_2,\ldots,i_{2p})$ will be $p$, only when $(i_1,i_2,\ldots,i_{2p})$ is {\it odd-even pair matched}. By free variables we mean that these variables can be chosen freely from their range $\{1,2,\ldots,n\}$. We shall use this observation in the proof of Theorem \ref{thm:revcovar_rp}. 
 %for maximum contribution.

%Now we  prove Theorem \ref{thm:revcircovar}. 
%To prove Theorem \ref{thm:revcircovar}, we will use similar idea as [ saha and maurya, Theorem 1] and  \cite[Theorem 5]{adhikari_saha2017}. 
\begin{proof}[Proof of Theorem \ref{thm:revcovar_rp}]
 First note the following observations from (\ref{eqn:covariance_rp}) and Definition \ref{def:odd-even_for_reverse}, 
 \begin{enumerate}
 \item[(i)] If $r$ is odd then $\E[V_{J_{r+1,2q}} ]=0$, as for odd value of $r$, there exist at least one random variable with odd power in $V_{J_{r+1,2q}}$. Moreover $\E[U_{I_{2p}} U_{J_r}]\E[ V_{J_{r+1,2q}} ]    -\E[ U_{I_{2p}}] \E[U_{J_r}]\E[ V_{J_{r+1,2q}}  ]=0$.
 \item[(ii)] If $r=0$ or $\{i_1,i_2,\ldots,i_{2p}\}\cap \{j_1,j_2,\ldots,j_{r}\}=\emptyset$ then from the independence of $u_i$ and $v_i$,    $\E[U_{I_{2p}} U_{J_r}]\E[ V_{J_{r+1,2q}} ]    -\E[ U_{I_{2p}}] \E[U_{J_r}]\E[ V_{J_{r+1,2q}}  ]=0$.
 % even $\{i_1, i_2, \ldots, i_{2p}\} \cap \{j_{r+1}, j_{r+2}, \ldots, j_{2q} \} \neq \emptyset$.
  \item[(iii)] Since $u_i$ and $v_i$ are independent for each $i$, cross-matches among  $\{i_1, i_2, \ldots, i_{2p}\}$ and $\{j_{r+1}, j_{r+2}, \ldots, j_{2q} \}$ reduces the number of free variables in total contribution. For maximum contribution in (\ref{eqn:covariance_rp}), we discuss only  $\{i_1, i_2, \ldots, i_{2p}\} \cap \{j_{r+1}, j_{r+2}, \ldots, j_{2q} \} = \emptyset$ case. In fact we can show that, $\{i_1, i_2, \ldots, i_{2p}\} \cap \{j_{r+1}, j_{r+2}, \ldots, j_{2q} \} \neq \emptyset$  has a zero contribution in (\ref{eqn:covariance_rp}).
\end{enumerate}  
% 
%  and hence $\E[U_{I_{2p}} U_{J_r}]\E[ V_{J_{r+1,2q}} ]    -\E[ U_{I_{2p}}] \E[U_{J_r}]\E[ V_{J_{r+1,2q}}  ]=0.$ Also if $r=0$ or $\{i_1,i_2,\ldots,i_{2p}\}\cap \{j_1,j_2,\ldots,j_{r}\}=\emptyset$, even $\{i_1, i_2, \ldots, i_{2p}\} \cap \{j_{r+1}, j_{r+2}, \ldots, j_{2q} \} \neq \emptyset$ then from independence of $u_i$ and $v_i$
%	$$\E[U_{I_{2p}} U_{J_r}]\E[ V_{J_{r+1,2q}} ]    -\E[ U_{I_{2p}}] \E[U_{J_r}]\E[ V_{J_{r+1,2q}}  ]=0.$$

	%the contribution from the right hand side is zero. 
	Now, from the above observations,  we conclude that, non-zero contribution in (\ref{eqn:covariance_rp}) will come only when $r$ is even and there is at least one cross-match among $\{i_1,\ldots,i_{2p}\}$ and $\{j_1,\ldots,j_{r}\}$ with
%	, i.e., $\{i_1,i_2,\ldots,i_{2p}\}\cap \{j_1,j_2,\ldots,j_{2q}\}\neq\emptyset$. Since $u_i$ and $v_i$ are independent therefore for non-zero contribution in (\ref{eqn:Cov_U,V_RC}), we will come only when $\{i_1,i_2,\ldots,i_{2p}\}\cap \{j_1,j_2,\ldots, j_{r}\}\neq\emptyset$ with 
	$\{i_1, i_2,\ldots, i_{2p}\}\cap \{j_{r+1}, j_{r+2},\ldots, j_{2q}\} = \emptyset$.
	Similar to the proof of Theorem 1 (Case I) of \cite{RC_independent_2020}, we can show that (\ref{eqn:covariance_rp}) has zero contribution when there is an odd number of cross-matches among $\{i_1,\ldots,i_{2p}\}$ and $\{j_1,\ldots, j_{r}\}$. 
	So for non-zero contribution, we discuss only the case when $\{i_1,\ldots,i_{2p}\}\cap \{j_1,\ldots,j_{r}\}|=2k$ for $k=1,2,\ldots, \min\{p,r'\}$, where $|\{\cdot\}|$ denotes cardinality of the set  $\{\cdot\}$ and $r'=\frac{r}{2}$. Under the above observations, (\ref{eqn:covariance_rp}) will be
	\begin{align} \label{eqn:T_1+T_2_rp}
	\lim_{n\to \infty} \Cov\big(w_{p}(t_1),w_{q}(t_2)\big) & = \lim_{n\to \infty}  \frac{1}{n^{p+q-1}}  \sum_ {r'=1}^{q} \binom{2q}{2r'} \sum_{k=1}^{\min\{p,r'\}}\sum_{ \tilde{A}_{2k}} \Big\{\E[U_{I_{2p}}U_{J_{2r'}}] \E[V_{J_{2r'+1,2q}} ] \nonumber \\
	&   \ \ \  -\E[ U_{I_{2p}}] \E[U_{J_{2r'}}]\E[ V_{J_{2r'+1,2q}}  ]        \Big\} \nonumber \\
	&= \sum_ {r'=1}^{q} \binom{2q}{2r'} \sum_{k=1}^{\min\{p,r'\}} T_1 - T_2, \mbox{ say}, 
	\end{align}	
	where for each $k=1,2, \ldots, \min\{p,r'\}$, $\tilde{A}_{2k}$ is defined as
	\begin{equation}\label{def:I_k_rp}
	\tilde{A}_{2k}:=\{((i_1,\ldots,i_{2p}),(j_1,\ldots, j_{2r'}, j_{2r'+1}, \ldots, j_{2q}))\in A_{2p}\times A_{2q}\suchthat |\{i_1,\ldots,i_{2p}\}\cap \{j_1,\ldots,j_{2r'}\}|=2k\}.
	\end{equation}
Note that a typical element of $\tilde{A}_{2k}$ can be written as $$(( \ell_1, \ell_2 , \ldots, \ell_{2k}, i_{2k+1}, \ldots,i_{2p} ) ( \ell_1, \ell_2 , \ldots, \ell_{2k}, j_{2k+1}, \ldots,j_{2r'} j_{2r'+1}, \ldots, j_{2q})).$$
Now we first calculate the term $T_1$ of (\ref{eqn:T_1+T_2_rp}) for some fixed $r'$. Recall that, for $2k$ many cross-matches among $I_{2p}$ and $J_{2r'}$, $T_1$ looks like
$$T_1= \lim_{n\to \infty} \frac{1}{n^{p+q-1}} \sum_{\tilde{A}_{2k}}\E[ u_{\ell_1}u_{\ell_2} \cdots u_{\ell_{2k}}  u_{i_{2k+1}}  \cdots u_{i_{2p}} u_{\ell_1}u_{\ell_2} \cdots u_{\ell_{2k}} u_{j_{2k+1}}\cdots u_{j_{2r'}} v_{j_{2r'+1}} \cdots v_{j_{2q}}  ].$$
Note that, non-zero contribution will occur in $T_1$ when the following conditions hold:
\begin{enumerate}
	\item[(i)]  $(i_{2k+1}, i_{2k+2}, \ldots,i_{2p}), (j_{2k+1}, j_{2k+2}, \ldots,j_{2r'})$ and $(j_{2r'+1}, j_{2r'+2}, \ldots,j_{2q})$ are {\it odd-even pair matched},
	\item[(ii)] $\{\ell_1,\ell_2, \ldots, \ell_{2k}\} \cap \{i_{2k+1}, i_{2k+2}, \ldots,i_{2p}\} \cap \{ j_{2k+1}, j_{2k+2}, \ldots, j_{2r'}\} \cap \{j_{2r'+1}, j_{2r'+2}, \ldots,j_{2q} \} = \emptyset $,
	\item[(iii)] for $k=1$, $\{\ell_1,\ell_2\}$ is {\it odd-even pair matched}, 
	\item[(iv)] for $k\geq 2$, each entries of the set  $\{\ell_1,\ell_2, \ldots, \ell_{2k}\}$ are distinct.
	\end{enumerate}  
%Since for $k\geq 2$, (i) imply that the constraint of the entries of $I_{2p}$ and $J_{2p}$ will change into a new resultant constraint $\sum_{d=1}^{2k}(-1)^d \ell_d=0 \mbox{ (mod $n$) }$. Therefore 
For $k\geq 2$, the contribution will be of the order $O(n^{2k-1+p-k+ r'-k + q-r'})=O(n^{p+q-1})$, where $(-1)$ arises due to the effective constraint, $\sum_{d=1}^{2k}(-1)^d \ell_d=0 \mbox{ (mod $n$)}$. For $k=1$, the contribution will be of the order $O(n^{p+ r'-1 + q-r'})=O(n^{p+q-1}) $, where $(-1)$ arises because once we have counted the entries of $I_{2p}$, the entries $\ell_1$ and $\ell_2$ of $J_{2r'}$ will be fixed. In other situations we can show that the contribution will be of the order $o(n^{p+q-1})$. Therefore for each $k$, total contribution in $T_1$ is $O(n^{p+q-1})$. Hence for each fixed $r'$ and $k$, 
\begin{align} \label{T_1_rp}
T_1 &= \lim_{n\to \infty} \frac{1}{n^{p+q-1}} \sum_{\tilde{A}_{2k}}\E[ u_{i_1} \cdots u_{i_{2p}} u_{j_1} \cdots u_{j_{2r'}} v_{j_{2r'+1}} \cdots v_{j_{2q}}   ] \nonumber \\
&= \lim_{n\to \infty} \frac{1}{n^{p+q-1}} \sum_{A'_{2k}} \sum_{A'_{2k}} c_k n^{(p-k + r'-k+ q-r')} (t_1)^{p-k+r'-k} (t_2-t_1)^{q-r'} \E[u_{i_1} \cdots u_{i_{2k}} u_{j_1} \cdots u_{j_{2k}}  ]\nonumber \\
&=  {t_1}^{p+r'-2k} (t_2-t_1)^{q-r'} c_k \lim_{n\to \infty} \frac{1}{n^{2k-1}} \sum_{A'_{2k}} \sum_{A'_{2k}} \E[u_{i_1} \cdots u_{i_{2k}} u_{j_1} \cdots u_{j_{2k}}  ] \nonumber \\
&=  {t_1}^{p+r'-2k} (t_2-t_1)^{q-r'} t_1^{2k} c_k \sum_{s=-(k-1)}^{k-1}\lim_{n\to \infty}\frac{|A_{2k,s}'| (2-{\bf 1}_{\{s=0\}})k!k!}{n^{2k-1}} \nonumber \\
&=  {t_1}^{p+r'} (t_2-t_1)^{q-r'} c_k\sum_{s=-(k-1)}^{k-1} \sum_{j=0}^{k+s-1}(-1)^j\binom{2k}{j}(k+s-j)^{2k-1}(2-{\bf 1}_{\{s=0\}})k!k! \nonumber \\
 &= {t_1}^{p+r'} (t_2-t_1)^{q-r'}  c_k g(k), 
\end{align}
 where $ A_{2k}' =\big\{(i_1,\ldots,i_{2k})\in \mathbb N^{2k}\suchthat \sum_{d=1}^{2k}(-1)^di_d=0 \mbox{ (mod $n$)}, 1\le i_1\neq i_2\neq \cdots\neq i_{2k}\le n\big\}$
 and 
 \begin{align*}
c_k &=\l(\binom{p}{p-k}^2(p-k)! \binom{r'}{r'-k}^2(r'-k)!\binom{q}{q-r'}^2(q-r')!\r),\\
 g(k) & =\frac{1}{(2k-1)!}\sum_{s=-(k-1)}^{k-1}\sum_{j=0}^{k+s-1}(-1)^j\binom{2k}{j}(k+s-j)^{2k-1}(2-{\bf 1}_{\{s=0\}})k!k!.
 \end{align*}
For details about $c_k, g(k)$ and last three equalities above, we refer the reader to the proof of Theorem 1 of \cite{RC_independent_2020}.

 Now we calculate the second term $T_2$ of  \eqref{eqn:T_1+T_2_rp} for some fixed $r'$. First let $k=1$. Then 
 $$T_2= \lim_{n\to \infty} \frac{1}{n^{p+q-1}} \sum_{\tilde{A}_{2}}\E[ u_{\ell_1}u_{\ell_2}  u_{i_3}  \cdots u_{i_{2p}} ] \E[u_{\ell_1}u_{\ell_2} u_{j_3}\cdots u_{j_{2r'}}]\E[ v_{j_{2r'+1}} \cdots v_{j_{2q}}  ].$$
 Clearly non-zero contribution will occur when following conditions hold: 
\begin{enumerate}
	\item[(i)] $\{\ell_1, \ell_2\}$, $(i_3, i_4, \ldots,i_{2p}), (j_3, j_4, \ldots,j_{2r'})$ and $(j_{2r'+1}, j_{2r'+2}, \ldots,j_{2q})$ are {\it odd-even pair matched},
	\item[(ii)] $\{\ell_1,\ell_2\} \cap \{i_3, i_4, \ldots,i_{2p}\} \cap \{ j_3, j_4, \ldots,j_{2r'}\} \cap \{j_{2r'+1}, j_{2r'+2}, \ldots,j_{2q} \} = \emptyset $,
	\end{enumerate} 
% $\{\ell_1, \ell_2\}$, $(i_3, i_4, \ldots,i_{2p}), (j_3, j_4, \ldots,j_{2r'})$ and $(j_{2r'+1}, j_{2r'+2}, \ldots,j_{2q})$ are {\it odd-even pair matched}; $\{\ell_1,\ell_2\} \cap \{i_3, i_4, \ldots,i_{2p}\} \cap \{ j_3, j_4, \ldots,j_{2r'}\} \cap \{j_{2r'+1}, j_{2r'+2}, \ldots,j_{2q} \} = \emptyset $, 
and the contribution will be $O(n^{p+ r'-1 + q-r'})=O(n^{p+q-1})$. 
In other situations, we can show that the contribution will be $o(n^{p+q-1})$. Therefore
 \begin{align}\label{T_2_k=1_rp}
 T_2\mathbb I_{\{k=1\}} & = \lim_{n\to \infty} \frac{1}{n^{p+q-1}} \sum_{\tilde{A}_{2}}\E[ u_{\ell_1}u_{\ell_2}  u_{i_3}  \cdots u_{i_{2p}} ] \E[u_{\ell_1}u_{\ell_2} u_{j_3}\cdots u_{j_{2r'}}]\E[ v_{j_{2r'+1}} \cdots v_{j_{2q}}  ]\nonumber \\
 &= c_1  t_1^{p+r'} (t_2-t_1)^{q-r'},
 \end{align}
where $c_1 = \l(\binom{p}{p-1}^2(p-1)! \binom{r'}{r'-1}^2(r'-1)!\binom{q}{q-r'}^2(q-r')!\r).$ 

 Now we deal with $k\geq 2$.  Here a typical term of $T_2$ looks like 
 $$\E[ u_{\ell_1}u_{\ell_2} \cdots u_{\ell_{2k}}  u_{i_{2k+1}}  \cdots u_{i_{2p}} ] \E[u_{\ell_1}u_{\ell_2} \cdots u_{\ell_{2k}} u_{j_{2k+1}}\cdots u_{j_{2r'}}]\E[ v_{j_{2r'+1}} \cdots v_{j_{2q}}  ].$$
 Similar to $T_1$, non-zero contribution in $T_2$ is possible when the following conditions hold: 
\begin{enumerate}
	\item[(i)] $(i_{2k+1}, i_{2k+2}, \ldots,i_{2p}), (j_{2k+1}, j_{2k+2}, \ldots,j_{2r'})$ and $(j_{2r'+1}, j_{2r'+2}, \ldots,j_{2q})$ are {\it odd-even pair matched},
	\item[(ii)] $\{\ell_1,\ell_2, \ldots, \ell_{2k}\} \cap \{i_{2k+1}, i_{2k+2}, \ldots,i_{2p}\} \cap \{ j_{2k+1}, j_{2k+2}, \ldots, j_{2r'}\} \cap \{j_{2r'+1}, j_{2r'+2}, \ldots,j_{2q} \} = \emptyset $,
	\item[(iii)] each entries of the set  $\{\ell_1,\ell_2, \ldots, \ell_{2k}\}$ are distinct,
	\end{enumerate}  
% $\{\ell_1, \ell_2 , \ldots, \ell_{2k} \}$, $(i_{2k+1}, i_{2k+2}, \ldots,i_{2p}), (j_{2k+1}, j_{2k+2}, \ldots,j_{2r'})$ and $(j_{2r'+1}, j_{2r'+2}, \ldots,j_{2q})$ are {\it odd-even pair matched}; $\{\ell_1,\ell_2, \ldots, \ell_{2k}\} \cap \{i_{2k+1}, i_{2k+2}, \ldots,i_{2p}\} \cap \{ j_{2k+1}, j_{2k+2}, \ldots, j_{2r'}\} \cap \{j_{2r'+1}, j_{2r'+2}, \ldots,j_{2q} \} = \emptyset $, 
and the contribution will be $O(n^{p+ r'-k + q-r'})=O(n^{p+q-k}) $. Therefore for each $k\geq 2$ and $r' \in \{1, 2, \ldots, q\}$ 
 \begin{equation} \label{T_2_k>1_rp}
 T_2 \mathbb I_{\{k\geq 2\}}=0.
 \end{equation}
Now using (\ref{T_1_rp}) (\ref{T_2_k=1_rp}) and (\ref{T_2_k>1_rp}) in  (\ref{eqn:T_1+T_2_rp}), we get
\begin{align*}
\lim_{n\to \infty} \Cov\big(w_{p}(t_1),w_{q}(t_2)\big) & = \sum_ {r'=1}^{q} \binom{2q}{2r'} \Big\{\sum_{k=1}^{\min\{p,r'\}} {t_1}^{p+r'} (t_2-t_1)^{q-r'}  c_k g(k) -  {t_1}^{p+r'} (t_2-t_1)^{q-r'} c_1 \Big\} \nonumber \\
& = \sum_ {r'=1}^{q} \binom{2q}{2r'} {t_1}^{p+r'} (t_2-t_1)^{q-r'} \Big\{\sum_{k=1}^{\min\{p,r'\}}  c_k g(k) - c_1\Big\} .
\end{align*}

This completes the proof of Theorem \ref{thm:revcovar_rp}.
\end{proof}
The following result is a consequence of  Theorem \ref{thm:revcovar_rp}. It will be used in the proof of Theorem \ref{thm:revmulti_rp}.
\begin{corollary}\label{res:existence of Gaussian process_rp} 
For $p\geq 1$, there exists a centred Gaussian process $\{N_{p}(t);t\geq 0\}$  such that
$$\Cov(N_{p}(t_1),N_{p}(t_2))= \sum_ {r'=1}^{p} \binom{2p}{2r'} {t_1}^{p+r'} (t_2-t_1)^{p-r'} \big[ \sum_{k=1}^{r'}  \tilde{c}_k g(k) - \tilde{c}_1 \big],$$
where $g(k)$ as in \eqref{eqn:c_k,g(k)_rp} and for $k=1,2, \ldots, p$
 \begin{align*}
\tilde{c}_k &=\l(\binom{p}{p-k}^2(p-k)! \binom{r'}{r'-k}^2(r'-k)!\binom{p}{p-r'}^2(p-r')!\r).
 \end{align*} 
\end{corollary}
The existence of such a Gaussian process is due to  Kolmogorov consistency theorem.

\section{Proof of Theorem \ref{thm:revmulti_rp}}\label{sec:joint convergence_rp}
First we define some notation and give some definitions which will be used in the proof of Theorem \ref{thm:revmulti_rp} and again in Section \ref{sec:joint convergence_sp} to prove Theorem \ref{thm:symmulti_sp}. 

For a vector $J\in \mathbb N^p$, define a multi-set $S_J$ as
$$S_{J} := \{j_1, j_2, \ldots,j_{p}\}.$$
\begin{definition}\label{def:connected_rev}
	Two vectors $J =(j_1, j_2, \ldots, j_{p})$ and $J' = (j'_1, j'_2, \ldots, j'_{q})$, where $J \in \mathbb N^p$ and $J' \in \mathbb N^q$, are said to be \textit{connected} if
	$S_{J}\cap S_{J'} \neq \emptyset$.	
\end{definition}
%For $1 \leq i \leq \ell$, suppose $J_i \in \mathbb N^{p_i}$. Now, we define cross-matched and self-matched element in $\displaystyle{\cup_{i=1}^{\ell} S_{J_i} }$.
%\begin{definition}\label{def:cross matched}
%	An element in $\displaystyle{\cup_{i=1}^{\ell} S_{J_i} }$ is called \textit{cross-matched} if it appears at least in two distinct $S_{J_i}$. If it appears in $k$ many ${S_{J_i}}^,s$, then we say its \textit{cross-multiplicity} is $k$.
%\end{definition}
%\begin{definition}\label{def:self-matched}
%An element in $\displaystyle{\cup_{i=1}^{\ell} S_{J_i} }$ is called \textit{self-matched} if it appears more than once in one of $S_{J_i}$. If it appears $k$ many times in $S_{J_i}$, then we say its \textit{self-multiplicity} in $S_{J_i}$ is $k$.
%	
%An element in $\displaystyle{\cup_{i=1}^{\ell} S_{J_i} }$  can be both self-matched and cross-matched. If an element of $\displaystyle{\cup_{i=1}^{\ell} S_{J_i} }$ has \textit{cross-multiplicity} one, that means, it appears only in one of the ${S_{J_i}}$. 
%\end{definition}	

\begin{definition}\label{def:cluster_rev}
	Given a set of vectors $S= \{J_1, J_2, \ldots, J_\ell \}$, where $J_i \in \mathbb N^{p_i}$ for $1 \leq i \leq \ell$, a subset $T=\{J_{n_1}, J_{n_2}, \ldots, J_{n_k}\}$ of $S$ is called a \textit{cluster} if it satisfies the following two conditions: 
	\begin{enumerate}
		\item[(i)] For any pair $J_{n_i}, J_{n_j}$ from  $T$ one can find a chain of vectors from 
		$T$, which starts with $J_{n_i}$ and ends with $J_{n_j}$ such that any two neighbouring vectors in the chain are connected.
		\item[(ii)] The subset $\{J_{n_1}, J_{n_2}, \ldots, J_{n_k}\}$ can not  be enlarged to a subset which preserves condition (i).
	\end{enumerate}
\end{definition}

Now we define $B_{P_\ell} \subseteq   A_{2p_1} \times A_{2p_2} \times \cdots \times A_{2p_\ell}$, where $A_{2p_j}$, $j=1,2,\ldots,\ell$, is as defined in \eqref{def:A_2p} in Section \ref{sec:cov_rp},
$$A_{2p_j}=\big\{(i_1,\ldots,i_{2p_j})\in \mathbb N^{2p_j}\suchthat \sum_{k=1}^{2p_j}(-1)^ki_k=0 \mbox{ (mod $n$)}, 1\le i_1,\ldots,i_{2p_j}\le n\big\}.$$ 
\begin{definition}\label{def:B_{P_l}_rp}
	Let $\ell \geq 2$ and  $P_\ell = (2p_1,2p_2, \ldots, 2p_\ell ) $. $ B_{P_\ell}$ is a subset of $ A_{2p_1} \times A_{2p_2} \times \cdots \times A_{2p_\ell}$ such that  $ (J_1, J_2, \ldots, J_\ell) \in B_{P_\ell} $ if 
	\begin{enumerate} 
		\item[(i)] $\{J_1, J_2, \ldots, J_\ell\} $ form a cluster, 
		\item[(ii)] each element in  $\displaystyle{\cup_{i=1}^{\ell} S_{J_i} }$ has  multiplicity greater than or equal to two. 
	\end{enumerate}
\end{definition} 

  The following Result gives us the cardinality of $B_{P_\ell}$.
 \begin{result}\label{res:cluster_rp} (Lemma 13, \cite{RC_independent_2020}
For  $\ell \geq 3 $, 
	$$|B_{P_\ell }| = o \big(n^{p_1+p_2 + \cdots + p_\ell -\frac{\ell}{2} }\big).$$
 \end{result}

The next lemma is an easy consequence of Result \ref{res:cluster_rp}.
\begin{lemma}\label{lem:maincluster_rp}
Suppose $\{J_1, J_2, \ldots, J_\ell \} $ form a cluster where $J_i\in A_{{2p}_i}$ with $p_i\geq 1$ for $1\leq i\leq \ell$. Then for $\ell \geq 3,$
\begin{equation}\label{equation:maincluster_rp}
	\frac{1}{ n^{p_1+p_2+ \cdots + p_\ell - \frac{\ell}{2}}} \sum_{A_{2p_1}, A_{2p_2}, \ldots, A_{2p_\ell}} \E\Big[\prod_{k=1}^{\ell}\Big(b_{J_k}(t_k) - \E(b_{J_k}(t_k))\Big)\Big] = o(1),
\end{equation}
where $0<t_1 \leq t_2 \leq \cdots \leq t_\ell $ and  for $k \in \{1,2, \ldots, \ell \} $, 
$$J_k = (j_{k,1}, j_{k,2}, \ldots, j_{k,2p_k} ) \ \mbox{and} \ b_{J_k}(t_k) = b_{j_{k,1}}(t_k) b_{j_{k,2}}(t_k) \cdots b_{j_{k,2p_k}}(t_k).$$	
\end{lemma}
\begin{proof} First observe that   $\E\Big[\prod_{k=1}^{\ell}\Big(b_{J_k}(t_k) - \E(b_{J_k}(t_k))\Big)\Big]$ will be non-zero only if each $b_i(t)$ appears at least twice in the collection $\{b_{j_{k,1}}(t_k), b_{j_{k,2}}(t_k), \ldots ,b_{j_{k,2p_k}}(t_k); 1\leq k\leq \ell\}$, because $\E(b_i(t))=0$ for $t\geq 0$. Therefore
\begin{equation}\label{eqn:equality_reduction_rp}
\sum_{A_{{2p}_1}, \ldots, A_{{2p}_\ell}} \hspace{-3pt}\E\Big[\prod_{k=1}^{\ell}\Big(b_{J_k}(t_k) - \E(b_{J_k}(t_k))\Big)\Big]=\sum_{(J_1,\ldots,J_\ell)\in B_{P_\ell}} \hspace{-3pt} \E\Big[\prod_{k=1}^{\ell}\Big(b_{J_k}(t_k) - \E(b_{J_k}(t_k))\Big)\Big],
\end{equation} 
where $B_{P_\ell}$ as in Definition \ref{def:B_{P_l}_rp}. Also note that  for $0<t_1 \leq t_2 \leq \cdots \leq t_\ell $ and $m\in \mathbb N$,  
\begin{equation*}\label{eqn:higher moment finite_rp}
\sup_{1\leq j\leq \ell}\E|b_i(t_j)|^{2m}=\frac{(2m)!}{2^m m!}t_\ell^m,
\end{equation*}  
as $\{b_i(t);t\geq 0\}_{i\geq 0}$ are independent standard Brownian motions. Therefore for fixed $0<t_1 \leq t_2 \leq \cdots \leq t_\ell $ and $p_1,p_2,\ldots,p_\ell \geq 1$, there exists $\alpha_\ell>0$, which depends only on $t_\ell$ and $p_1,p_2,\ldots,p_\ell$, such that  
\begin{equation}\label{eqn:modulus finite_rp}
\Big|\E\big[\prod_{k=1}^{\ell}\big(b_{J_k}(t_k) - \E(b_{J_k}(t_k))\big)\big]\Big|\leq \alpha_\ell
\end{equation} 
for all $(J_1, J_2, \ldots, J_\ell)\in A_{2p_1}\times A_{2p_2}\times \cdots \times A_{2p_\ell}$.
 
Now using \eqref{eqn:equality_reduction_rp} and \eqref{eqn:modulus finite_rp}, we get
\begin{align*}
	 \sum_{A_{2p_1}, A_{2p_2}, \ldots, A_{2p_\ell}} \Big|\E\big[\prod_{k=1}^{\ell}\big(b_{J_k}(t_k) - \E(b_{J_k}(t_k))\big)\big]\Big|
 \leq  \sum_{(J_1,J_2,\ldots,J_\ell)\in B_{P_\ell}} \alpha_{\ell} 
\ = |B_{p_\ell}| \ \alpha_\ell.
\end{align*}
%where $B_{P_{\ell}}$ as in Definition \ref{def:B_{P_l}}   and $\alpha$ is a constant, which depends on $t_1,t_2,\ldots,t_\ell$ and $p_1,p_2,\ldots,p_\ell$. 
Therefore \eqref{equation:maincluster_rp} follows from the last inequality and Result \ref{res:cluster_rp}. This completes the proof of lemma. 
\end{proof}
Now we shall use the above lemma and result  to prove Theorem \ref{thm:revmulti_rp}.

\begin{proof}[Proof of Theorem \ref{thm:revmulti_rp}] 
%The proof of Theorem \ref{thm:revmulti_rp}  is similar to the proof of Theorem $2$ of \cite{maurya2019process}, where the authors have proved a similar result for circulant  matrices. The only difference is that  the set $A_{2p}$ here has different constraint than the circulant case in \cite{maurya2019process}. So, we shall point out  the main steps of the proof and skip the details.
 
We use Cram\'er-Wold theorem and method of moments to prove Theorem \ref{thm:revmulti_rp}. So it is enough to show that, for $0<t_1 \leq t_2 \leq \cdots \leq t_\ell $ and $p_1, p_2, \ldots , p_\ell \geq 1$,
$$\lim_{n\to\infty}\E[w_{{p}_1}(t_1) w_{{p}_2}(t_2) \cdots w_{{p}_\ell}(t_\ell)]=\E[N_{{p}_1}(t_1)N_{{p}_2}(t_2) \cdots N_{{p}_\ell}(t_\ell)].$$
By using trace formula (\ref{trace formula RC_n}), we get  
\begin{align}\label{eqn:expectation_thm2_rp}
&\quad \E[w_{{p}_1}(t_1) w_{{p}_2}(t_2) \cdots w_{{p}_\ell}(t_\ell)] \\
&= \frac{1}{n^{p_1 + p_2 + \cdots +p_\ell -\frac{\ell}{2}}} \sum_{A_{2p_1}, A_{2p_2}, \ldots, A_{2p_\ell}} \E\big[ (b_{J_1} - \E b_{J_1}) (b_{J_2} - \E b_{J_2})  \cdots (b_{J_\ell} - \E b_{J_\ell})\big],\nonumber
 \end{align}
 where for each $k= 1, 2, \ldots, \ell, \  J_k = (j_{k,1},j_{k,2},\ldots,j_{k, 2p_k})\in A_{2p_k}$ and $b_{J_k}= b_{j_{k,1}}(t_k)\cdots b_{j_{k, 2p_k}}(t_k)$.

First observe that, for a fixed $J_1,J_2,\ldots,J_\ell$, if there exists a $k\in\{1,2,\ldots,\ell\}$ such that $J_k$ is not connected with any $J_i$ for $i\neq k$, then 
$$\E\big[ (b_{J_1} - \E b_{J_1}) (b_{J_2} - \E b_{J_2})  \cdots (b_{J_\ell} - \E b_{J_\ell})\big]=0$$
due to the independence of Brownian motions $\{b_i(t)\}_{i\geq 0}$.

%Therefore $J_1,J_2,\ldots,J_\ell$ must form  clusters with each cluster length greater than or equal to two, that is, each cluster should contain at least two vectors. Suppose $C_1,C_2,\ldots,C_s$ are the clusters formed by  vectors $J_1,J_2,\ldots,J_\ell$ and   $|C_i|\geq 2$ for all $1\leq i \leq s$ where $|C_i|$ denotes the length of the cluster $C_i$. Observe that $\sum_{i=1}^s |C_i|=\ell$.   
%
%If there exists  a cluster $C_j$ among $C_1,C_2,\ldots,C_s$ such that $|C_j|\geq 3$, then from Theorem \ref{thm:circovar} and Lemma \ref{lem:maincluster}, we have  
%\begin{align*}
%\frac{1}{n^{\frac{p_1 + p_2 + \cdots +p_\ell -\ell}{2}}} \sum_{A_{p_1}, A_{p_2}, \ldots, A_{p_\ell}} \E\big[ (b_{J_1} - \E b_{J_1}) (b_{J_2} - \E b_{J_2})  \cdots (b_{J_\ell} - \E b_{J_\ell})\big]=o(1).
%\end{align*}  
Now if $\ell$ is odd then there will be a cluster of odd length in $\{ J_1, J_2, \ldots, J_\ell \}$ and hence from Result \ref{res:cluster_rp}, we get
$$ \lim_{n\tends \infty}  \E[w_{{p}_1}(t_1)  w_{{p}_2}(t_2) \cdots w_{{p}_\ell}(t_\ell)] = 0.$$ 
 
Similarly, if $\ell$ is even then the contribution due to $\{ J_1, J_2, \ldots, J_\ell \}$ to  $\E[w_{{p}_1}(t_1)  w_{{p}_2}(t_2) \cdots w_{{p}_\ell}(t_\ell)] $ is $O(1)$ only when $\{ J_1, J_2, \ldots, J_\ell\}$  decomposes into clusters of length 2. Therefore from \eqref{eqn:expectation_thm2_rp}, we get
\begin{align*} \label{eq:multisplit}
 & \quad \lim_{n\tends \infty}  \E[w_{{p}_1}(t_1)  w_{{p}_2}(t_2) \cdots w_{{p}_\ell}(t_\ell)] \\
% & =\lim_{n\to\infty} \frac{1}{n^{\frac{p_1 + p_2 + \cdots +p_\ell -\ell}{2}}} \sum_{A_{p_1}, A_{p_2}, \ldots, A_{p_\ell}} \E\big[ (b_{J_1} - \E b_{J_1}) (b_{J_2} - \E b_{J_2})  \cdots (b_{J_\ell} - \E b_{J_\ell})\big]\\
 & =\lim_{n\to\infty} \frac{1}{n^{p_1 + p_2 + \cdots +p_\ell -\frac{\ell}{2}}} \sum_{\pi \in \mathcal P_2(\ell)} \prod_{i=1}^{\frac{\ell}{2}}  \sum_{A_{p_{y(i)}},\ A_{p_{z(i)}}} \E\big[ (b_{J_{y(i)}} - \E b_{J_{y(i)}}) (b_{J_{z(i)}} - \E b_{J_{z(i)}})\big],
 \end{align*}
where  $\pi = \big\{ \{y(1), z(1) \}, \ldots , \{y(\frac{\ell}{2}), z(\frac{\ell}{2})  \} \big\}\in \mathcal P_2(\ell)$ and $\mathcal P_2(\ell)$ is the set of all pair-partitions of $ \{1, 2, \ldots, \ell\} $. Using Theorem \ref{thm:revcovar_rp}, from the last equation, we get
\begin{equation}\label{eqn:product of expectation_rp}
 \lim_{n\tends \infty}  \E[w_{{p}_1}(t_1)  w_{{p}_2}(t_2) \cdots w_{{p}_\ell}(t_\ell)]
  =\sum_{\pi \in P_2(\ell)} \prod_{i=1}^{\frac{\ell}{2}} \lim_{n\tends \infty} \E[w_{p_{y(i)}} (t_{y(i)}) w_{p_{z(i)}} (t_{z(i)})].
  \end{equation}
%&= \lim_{n\tends \infty} \sum_{\pi \in P_2(\ell)} \prod_{i=1}^{\frac{\ell}{2}} \E[w_{p_{y(i)}} (t_{y(i)}) w_{p_{z(i)}} (t_{z(i)})] \nonumber \\
Now from Corollary \ref{res:existence of Gaussian process_rp}, there exists a centred Gaussian process $\{N_{p}(t);t\geq 0\}_{p\geq 1}$ such that 
$$\E(N_{p}(t_1)N_{q}(t_2))=\lim_{n\to\infty}\E(w_{p}(t_1)w_{q}(t_2)).$$
Therefore using Wick's formula,  from \eqref{eqn:product of expectation_rp} we get 
\begin{align*}
\lim_{n\tends \infty}  \E[w_{{p}_1}(t_1) w_{{p}_2}(t_2) \cdots w_{{p}_\ell}(t_\ell)]
%  &=\sum_{\pi \in \mathcal P_2(\ell)} \prod_{i=1}^{\frac{\ell}{2}} \lim_{n\tends \infty} \E[w_{p_{y(i)}} (t_{y(i)}) w_{p_{z(i)}} (t_{z(i)})]\\
% & =\sum_{\pi \in \mathcal P_2(\ell)} \prod_{i=1}^{\frac{\ell}{2}} \E[N_{p_{y(i)}} (t_{y(i)}) N_{p_{z(i)}}  (t_{z(i)})] \\
 &=\E[ N_{{p}_1}(t_1)N_{{p}_2}(t_2) \cdots N_{{p}_\ell}(t_\ell) ].
\end{align*}
This completes the proof of Theorem \ref{thm:revmulti_rp}. 
\end{proof}
%The following remark gives us a generalization of Theorem \ref{thm:revmulti_rp}.

%In the following remark we discuss the case when the test functions are odd degree monomial or polynoma ial with odd degree terms. 
The following remark is a continuation of Remark \ref{rem:RC_odd}. 
%In this remark, we discuss fluctuation of linear eigenvalue value statistics of $RC_n(t)$ when the test functions are odd degree monomial or polynomial with odd degree terms. 
\begin{remark}
Suppose $RC_n(t)$ is the reverse circulant matrix with entries $\{\frac{b_i(t)}{\sqrt n}; 1\le i\le n\}$, where $\{b_i(t)\}_{i\geq 1}$ is a sequence of independent standard Brownian motion.
From Chapter 1 of  \cite{bosesaha_circulantbook}, the eigenvalues $\lambda_0,\lambda_1,\ldots,\lambda_{n-1}$ of $RC_n(t)$ are given by 
\begin{equation} \label{def:eigen,RC}
 \left\{\begin{array}{ll} 	 
            \lambda_0 = \frac{1}{\sqrt n}\displaystyle\sum_{j=1}^{n} b_j(t), & \\
		 	\lambda_{\frac{n}{2}} = \frac{1}{\sqrt n}\displaystyle \sum_{j=1}^{n} (-1)^j b_j(t), &   \text{if n} \mbox{ is even},\\
				\lambda_k = - \lambda_{n-k},  & \mbox{for } 1 \leq k \leq \lfloor\frac{n-1}{2}\rfloor. 	 	 
		 	  \end{array}\right.	
 \end{equation}
%where $w_k = \frac{2 \pi k}{n}$ and  $I_n(w_k)= \displaystyle \frac{1}{n} | \sum_{t=0}^{n-1} \frac{X_t}{\sqrt n} \exp^{-itw_k} |^2.$
If we consider linear eigenvalue statistics with test function $\phi(x)=x^{2p+1}$, $p\ge 1,$ then  
$$ \sum_{k=0}^{n-1} \phi (\lambda_k)= \sum_{k=1}^{n}(\lambda_k)^{2p+1}= \left\{\begin{array}{ll}
\lambda_0^{2p+1} & \mbox{ if } n \mbox{ is odd},\\\\
\lambda_0^{2p+1}+\lambda_{\frac{n}{2}}^{2p+1} & \mbox{ if } n \mbox{ is even}.
\end{array}
\right.$$
Therefore, by using CLT and continuous mapping theorem, we get 
$$\sum_{k=0}^{n-1} \phi (\lambda_k)\stackrel{d}{\longrightarrow} (N(0,t))^{2p+1}, \mbox{ if }n \mbox{ is odd} \mbox{ and }  n\to \infty.$$
%where $N$ is a standard normal random variable.
%Also, by central limit theorem and continuous mapping theorem, 
It is also easy to see that 
$$\sum_{k=0}^{n-1} \phi (\lambda_k)\stackrel{d}{\longrightarrow} N_1^{2p+1}+N_2^{2p+1}, \mbox{ if }n \mbox{ is even} \mbox{ and }  n\to \infty,$$
where $N_1$ and $N_2$ are independent $N(0,t)$ random variables.

%A similar conclusion can be drawn if the test function is a polynomial with odd degree terms only.
Similarly, for $\phi(x)=\sum_{k=0}^p a_kx^{2k+1}$, $p\ge 1$, 
\begin{align*}
\sum_{k=0}^{n-1} \phi (\lambda_k)& \stackrel{d}{\longrightarrow} \sum_{k=0}^p a_k (N(0,t))^{2k+1}\mbox{ if }n \mbox{ is odd} \mbox{ and }  n\to \infty,\\
\sum_{k=0}^{n-1} \phi (\lambda_k)& \stackrel{d}{\longrightarrow} \sum_{k=0}^p  a_k (N_1^{2k+1}+N_2^{2k+1}) \mbox{ if }n \mbox{ is even} \mbox{ and }  n\to \infty. 
\end{align*}
Therefore, for any odd degree monomial or a polynomial with only odd degree terms,  the behaviour of the linear eigenvalue statistics of $RC_n(t)$ depends on $n$.  
\end{remark}
\section{Proof of Theorem \ref{thm:revprocess}}\label{sec:process convergence_rp} 
We first state some standard notation and results on process convergence which we shall need.
Suppose $C_{\infty}:= C[0, \infty)$ is the space of all real-valued continuous functions on $[0, \infty)$. Then $(C_{\infty}, \rho )$ is a metric space with the following metric.
$$ \rho(X,Y) = \sum_{k=1}^{\infty} \frac{1}{2^k} \frac{\rho_k(X,Y)}{1+ \rho_k(X,Y)} \ \ \forall \ X, Y \in C_{\infty},$$  
where $\rho_k(X,Y) = \sup_{0 \leq t \leq k} |X(t)- Y(t)|$ is the usual metric on $C_k := C[0,k]$, the space of all continuous real-valued functions on $[0, k]$. Suppose $\mathcal{C_{\infty}}$ and $\mathcal{C}_k$ are the $\sig$-field generated by the open sets of $C_{\infty}$ and $C_k$, respectively. 
\begin{definition}\label{def:processconvergence}
Let $\{ \P_n \}$ and $\P$ be probability measures on $( C_{\infty},\mathcal{C_{\infty}} )$. If 
\begin{equation}\label{weak_convergence_C_infinity}
 \P_n f :=  \int_{C_{\infty}} f d \P_n \longrightarrow \P f := \int_{C_{\infty}} f d \P \ \mbox{ as }n\to\infty,
 \end{equation}
 for every bounded, continuous real-valued function $f$ on $C_{\infty}$, then we say $\P_n$ converge to $\P$ \textit{weakly} or in \textit{distribution}. We denote it by $\P_n \stackrel{\mathcal D}{\rightarrow} \P $.
 
 Let $\{\boldsymbol{X_n}\}_{n\geq 1}$ = $\{ X_n(t); t \geq 0 \}_{n\geq 1}$ be a sequence of real-valued continuous processes and $\boldsymbol{X}=\{X(t);t\ge 0\}$ be a real-valued continuous process.
 % defined on same probability space. 
 We say $\boldsymbol{X_n}$ converge to $\boldsymbol{X}$ \textit{weakly} or in \textit{distribution} if $\P_n \stackrel{\mathcal D}{\rightarrow} \P $, where $\P_n$ and $\P$ are the probability measures on $( C_{\infty},\mathcal{C_{\infty}} )$ induced by $\boldsymbol{X_n}$ and $\boldsymbol X$, respectively. We denote it by $\boldsymbol{X_n} \stackrel{\mathcal D}{\rightarrow} \boldsymbol{X}$.
 \end{definition}
 The following result gives a necessary and sufficient  condition for the weak convergence of $\P_n$ to $\P$.
\begin{result} \label{result:process convergance}(Theorem 3, \cite{Ward1970})  
%\textbf{(Theorem 3, \cite{Ward1970})}  
	Suppose $\{ \P_n \}$ and $\P$ are probability measures on  $( C_{\infty},\mathcal{C_{\infty}} )$. Then  $\P_n \stackrel{\mathcal D}{\rightarrow} \P$ if and only if:
	\vskip3pt
	\noindent \textbf{(i)} the finite-dimensional distributions of $\P_n$ converge weakly to those of $\P$, that is,
	\begin{equation} \label{eqn:P_nGamma convergence}
	\P_n \pi^{-1}_{t_1\ldots t_r} \stackrel{\mathcal D}{\rightarrow} \P \pi^{-1}_{t_1\ldots t_r} \ \ \forall \ r \in \mathbb{N} \ \mbox{and} \ \forall \ r\mbox{-tuples} \ t_1, \ldots, t_r, \nonumber
	\end{equation}
	where $\pi_{t_1\ldots t_r}$ is the natural projection from $C_{\infty}$ to $\mathbb{R}^r$ for all $r \in \mathbb{N}$, and 
	\vskip3pt
	\noindent \textbf{(ii)} the sequence $\{ \P_n \} $ is tight.
\end{result}
The following result provides a sufficient condition for the tightness of the probability measures $\{ \P_n\}$.

\begin{result} (Theorem I.4.3, \cite{Ikeda1981})\label{result:tight2}
	Suppose $\{\boldsymbol{X_n}\}_{n\geq 1}$ = $\{ X_n(t); t \geq 0 \}_{n\geq 1}$ is a sequence of continuous processes satisfying the following two conditions:
	\vskip3pt
    \noindent	\textbf{(i)} there exists positive constants $M$ and $\gamma$ such that
	$$\E{|X_n(0)|^ \gamma } \leq M \ \ \ \forall \ n \in \mathbb{N},$$ 	 
    \noindent	\textbf{(ii)} there exists positive constants $\alpha, \ \beta$ and $M_T$, $T = 1, 2, \ldots,$ such that
	$$\E{|X_n(t)- X_n(s)|^ \alpha } \leq M_T |t-s|^ {1+ \beta} \ \ \ \forall \ n \in \mathbb{N} \ \mbox{and } t,s \in[0,\ T],(T= 1, 2, \ldots,).$$
	Then the sequence $\{ \boldsymbol{X_n} \} $ is tight and the sequence of probability measures $\{ \P_n \}$ on $( C_{\infty},\mathcal{C_{\infty}} )$, induced by $\{ \boldsymbol{X_n} \} $ is also tight.	
	\end{result}
Note that in our situation $X_n(t)$ is $w_{p}(t)$, where $w_{p}(t)$ is as defined in (\ref{eqn:w_p(t)_rp}). Now from Result 
%\textbf{(Theorem 3, \cite{Ward1970})}, \textbf{(Theorem I.4.3, \cite{Ikeda1981})}
\ref{result:process convergance} and Result \ref{result:tight2}, to complete the proof of Theorem \ref{thm:revprocess}, it is sufficient to prove the following two propositions: 
\begin{proposition} \label{pro:finite convergence_rp}
	For each $ p\geq  1$, suppose $0<t_1<t_2 \cdots <t_r$. Then as $n \tends \infty$
	\begin{equation*}
	(w_{p}(t_1), w_{p}(t_2), \ldots , w_{p}(t_r)) \stackrel{\mathcal D}{\rightarrow} (N_{p}(t_1), N_{p}(t_2), \ldots , N_{p}(t_r)).
	\end{equation*}	
\end{proposition}

\begin{proposition} \label{pro:tight_rp}
	For each $p \geq1$, there exists positive constants $M$ and $\gamma$ such that
	\begin{equation*} \label{eq:tight3}
	\E{|w_{p}(0)|^ \gamma } \leq M \ \ \ \forall \ n \in \mathbb{N}, 	 
	\end{equation*}
	there exists positive constants $\alpha, \ \beta $ and $M_T$, $T= 1, 2, \ldots,$ such that
	\begin{equation} \label{eq:tight4_rp}
	\E{|w_{p}(t)- w_{p}(s)|^ \alpha } \leq M_T |t-s|^ {1+ \beta} \ \ \ \forall \ n \in \mathbb{N} \ \mbox{and } t,s \in[0,\ T],(T= 1, 2, \ldots,).
	\end{equation}
	Then $\{ w_{p}(t) ; t \geq 0 \}$ is tight.
\end{proposition} 
    Proposition \ref{pro:finite convergence_rp} yields the finite dimensional convergence of $\{w_{p}(t);t \geq 0\}$ and Proposition \ref{pro:tight_rp} provides the tightness of the process $\{w_{p}(t);t \geq 0\}$.
%\begin{proof}[Proof of Proposition \ref{pro:finite convergence}]
Proposition \ref{pro:finite convergence_rp} is a particular case of Theorem \ref{thm:revmulti_rp}, once we take $p_i = p$.

\begin{proof}[Proof of Proposition \ref{pro:tight_rp}] Note that $w_{p}(0)= 0$ for all $p \geq 1$. 
	So to complete the proof, we shall prove (\ref{eq:tight4_rp}) of Proposition \ref{pro:tight_rp} for $\alpha = 4$ and $\beta = 1$. 
	First recall $w_{p}(t)$ from (\ref{eqn:w_p(t)_rp}),
$$ w_{p}(t) = \frac{1}{\sqrt{n}} \bigl\{ \Tr(RC_n(t))^{2p} - \E[\Tr(RC_n(t))^{2p}]\bigr\}.$$
Suppose $p \geq 1$ is fixed and $t,s \in[0,T] $, for some $T \in \mathbb{N}$. Then
\begin{align} \label{eqn:w_p(t-s)_rp}
 w_{p}(t) - w_{p}(s) &= \frac{1}{\sqrt{n}} \big[ \Tr(RC_n(t))^{2p} - \Tr(RC_n(s))^{2p} - \E[\Tr(RC_n(t))^{2p} - \Tr(RC_n(s))^{2p}]\big].
\end{align}
Since $0< s <t $, we have
\begin{align*}
RC_n(t) &= RC_n(t) - RC_n(s) + RC_n(s)
 = RC'_n(t-s) + RC_n(s)
\end{align*}
where $RC'_n(t-s)$ is a reverse circulant matrix with entries $\{ b_n(t) -b_n(s)\}_{n \geq 0}$. 
Now
 \begin{equation*}
 (RC_n(t))^{2p}  = [RC'_n(t-s)+ RC_n(s)]^{2p}.
 \end{equation*} 
 Using binomial expansion on right hand side of the above equation, we get 
\begin{align} \label{eqn:C_n^p(t)- C_n^p(s)}
(RC_n(t))^{2p} - (RC_n(s))^{2p} & = (RC'_n(t-s))^{2p} + \sum_{d=1}^{2p-1} \binom{2p}{d} (RC'_n(t-s))^d (RC_n(s))^{2p-d}, \nonumber 
\end{align} 
and hence
\begin{equation} \label{eqn:trace t-s_rp}
 \Tr (RC_n(t))^{2p} - \Tr (RC_n(s))^{2p}   
 = \Tr[ (RC'_n(t-s) )^{2p}] + \sum_{d=1}^{2p-1} \binom{2p}{d} \Tr [(RC'_n(t-s))^d (RC_n(s))^{2p-d}].
\end{equation}
%Recall the trace formula from (\ref{trace formula RC_n}) for $(RC'_n(t-s) )^{2p}$ and $(RC'_n(t-s))^d (RC_n(s))^{2p-d}$
Now from the trace formula (\ref{trace formula RC_n}), we get
\begin{align} \label{eqn:tracemultply_rp}
	\Tr[ (RC'_n(t-s) )^{2p}] & =\frac{1}{n^{p-1}}\sum_{A_{2p}} b'_{I_{2p}}(t-s),  \\
	\Tr [(RC'_n(t-s))^d (RC_n(s))^{2p-d}] & = \frac{1}{n^{p-1}}\sum_{A_{2p}}  b'_{I_{d}}(t-s) b_{I_{2p-d}}(s), \nonumber
\end{align}
where 
\begin{align*}
b'_{I_{2p}}(t-s) & = (b_{i_1}(t)-b_{i_1}(s)) (b_{i_2}(t) - b_{i_2}(s)) \cdots (b_{i_{2p}}(t) - b_{i_{2p}}(s)), \\
b'_{I_{d}}(t-s) & = (b_{i_1}(t)-b_{i_1}(s)) (b_{i_2}(t) - b_{i_2}(s)) \cdots (b_{i_{d}}(t) - b_{i_{d}}(s)), \\
b_{I_{2p-d}}(s) & = b_{i_{d+1}}(s) b_{i_{d+2}}(s)\cdots b_{i_{2p}}(s).
\end{align*}
Using the fact that $b_n(t) -b_n(s)$ has same distribution as $b_n(t-s)$ and  $b_n(t)$ has same distribution as $N(0, \sqrt{t})$, we have
\begin{align}  \label{eqn:same distrib_rp}
b'_{I_{2p}}(t-s) & \distas{D} x_{i_1} x_{i_2} \cdots x_{i_{2p}} \distas{D} (t-s)^{p} \frac{x_{i_1}}{\sqrt{(t-s)}} \frac{x_{i_2}}{\sqrt{(t-s)}} \cdots  \frac{x_{i_{2p}}} {\sqrt{(t-s)}} = (t-s)^{p} X_{I_{2p}}, \mbox{ say } \\
b'_{I_{d}}(t-s)  & \distas{D} (t-s)^{\frac{d}{2}} \frac{x_{i_1}}{\sqrt{(t-s)}} \frac{x_{i_2}}{\sqrt{(t-s)}} \cdots  \frac{x_{i_{d}}} {\sqrt{(t-s)}} = (t-s)^{\frac{d}{2}} X_{I_{d}}, \mbox{ say} \nonumber \\
b'_{I_{2p-d}}(s) & \distas{D} y_{i_{d+1}} y_{i_{d+2}} \cdots y_{i_{2p}} \distas{D} {s}^{\frac{2p-d}{2}} \frac{y_{i_{d+1}}}{\sqrt{s}} \frac{y_{i_{d+2}}}{\sqrt{s}} \cdots  \frac{y_{i_{2p}}} {\sqrt{s}} = {s}^{\frac{2p-d}{2}} Y_{I_{2p-d}}, \mbox{ say } \nonumber
%b'_{I_{d}}(t-s)  & \distas{D} (t-s)^{\frac{d}{2}} X_{I_{d}}, \ \  \
%b_{I_{2p-d}}(s)  & \distas{D} (s)^{\frac{2p-d}{2}} Y_{I_{2p-d}}
\end{align}
where $x_{i_r}$'s and $y_{i_r}$'s are independent normal random variable with mean zero, variance $(t-s)$ and $s$, respectively. We denotes $x_1 \distas{D} x_2$ if $x_1$ and $x_2$ have same distribution. 

Now, using (\ref{eqn:tracemultply_rp}) and (\ref{eqn:same distrib_rp}) in (\ref{eqn:trace t-s_rp}), we get
\begin{align} \label{eqn:Z_p_rp}
 \Tr (RC_n(t))^{2p}- \Tr (RC_n(s))^{2p}  & \distas{D}  \frac{1}{n^{p-1}}\Big(\sum_{A_{2p}} (t-s)^{p} X_{I_{2p}} + \sum_{d=1}^{p-1} \binom{p}{d}  (t-s)^{\frac{d}{2}} {s}^{\frac{2p-d}{2}} X_{I_{d}}  Y_{I_{2p-d}}\Big) \nonumber \\
 & = \frac{\sqrt{t-s}}{n^{p-1}}\Big(\sum_{A_{2p}} (t-s)^{\frac{2p-1}{2}} X_{I_{2p}} + \sum_{d=1}^{p-1} \binom{p}{d}  (t-s)^{\frac{d-1}{2}} {s}^{\frac{2p-d}{2}} X_{I_{d}}  Y_{I_{2p-d}}\Big) \nonumber \\
 & = \frac{\sqrt{t-s}}{n^{p-1}}\sum_{A_{2p}} Z_{I_{2p}}, \mbox{ say}.
\end{align} 
On using (\ref{eqn:Z_p_rp}) in (\ref{eqn:w_p(t-s)_rp}), we get
\begin{align*} 
 w_{p}(t) - w_{p}(s) \distas{D} \frac{\sqrt{(t-s)}}{n^{\frac{2p-1}{2}}}\sum_{A_{2p}} (Z_{I_{2p}} - \E[Z_{I_{2p}}])
\end{align*}
and therefore 
\begin{align} \label{eqn:w_p^4_rp}
\E[w_{p}(t) - w_{p}(s)]^4 & = \frac{(t-s)^2}{n^{4p-2}}\sum_{A_{2p}, A_{2p}, A_{2p}, A_{2p}} \E \big[ (Z_{I_{2p}} - \E Z_{I_{2p}}) (Z_{J_{2p}} - \E Z_{J_{2p}}) \\
& \qquad \qquad \qquad \qquad \qquad (Z_{K_{2p}} - \E Z_{K_{2p}}) (Z_{L_{2p}} - \E Z_{L_{2p}}) \big], \nonumber
\end{align}
where $I_{2p}, J_{2p}, K_{2p}, L_{2p}$ are vectors in $A_{2p}$. Depending on connectedness between $I_{2p}, J_{2p}, K_{2p}$ and $L_{2p}$, the following three cases arise: \\

\noindent \textbf{Case I.} \textbf{At least one of $I_{2p}, J_{2p}, K_{2p}, L_{2p}$ is not connected with remaining ones:}  Without loss of generality suppose $I_{2p}$ is not connected with $J_{2p}, K_{2p}$ or $L_{2p}$, then due to independence of entries
$$\E \big[ (Z_{I_{2p}} - \E Z_{I_{2p}}) (Z_{J_{2p}} - \E Z_{J_{2p}}) (Z_{K_{2p}} - \E Z_{K_{2p}}) (Z_{L_{2p}} - \E Z_{L_{2p}})\big] = 0.$$
If there are more than one vectors which are not connected with remaining ones, then again by the independence of entries
$$\E \big[ (Z_{I_{2p}} - \E Z_{I_{2p}}) (Z_{J_{2p}} - \E Z_{J_{2p}}) (Z_{K_{2p}} - \E Z_{K_{2p}}) (Z_{L_{2p}} - \E Z_{L_{2p}})\big] = 0.$$ 
%Hence in this case $\E[w_{2p}(t) - w_{2p}(s)]^4 = 0$ which shows that (\ref{eq:tight4}) is true. \\
 
 \noindent \textbf{Case II.} \textbf{$I_{2p}$ is connected with one of $J_{2p}, K_{2p}, L_{2p}$ only and the remaining two of $J_{2p}, K_{2p}, L_{2p}$ are also connected with themselves only:} Without loss of generality suppose $I_{2p}$ is connected with $J_{2p}$ only and $K_{2p}$ is connected with $L_{2p}$ only. %that is, $(S_{I_{2p}} \cup S_{J_{2p}}) \cap (S_{K_{2p}} \cup S_{L_{2p}}) = \emptyset$ and 
 With this conditions, the right side of (\ref{eqn:w_p^4_rp}) can be written as
\begin{align} \label{eqn:case II_rp} 
%\E[w_{2p}(t) - w_{2p}(s)]^4 & = 
\frac{(t-s)^2}{n^{4p-2}} \sum_{ A_{2p}, A_{2p}} \E \big[ (Z_{I_{2p}} - \E Z_{I_{2p}}) (Z_{J_{2p}} - \E Z_{J_{2p}}) \big] 
 \sum_{A_{2p}, A_{2p}} \E \big[ (Z_{K_{2p}} - \E Z_{K_{2p}}) (Z_{L_{2p}} - \E Z_{L_{2p}}) \big].
\end{align}
We denote the above expression by $S_1$. From the definition of $Z_{J_{2p}}$, given in (\ref{eqn:Z_p_rp}), it is clear that $\E\big[(Z_{I_{2p}} - \E Z_{I_{2p}}) (Z_{J_{2p}} - \E Z_{J_{2p}})\big]$ will be non-zero only when each $x_{i_r}$ and $y_{i_r}$ appears at least twice, because $\E(x_i)= \E(y_i) = 0$. Therefore
\begin{equation} \label{eqn:B_p_2_rp}
\sum_{ A_{2p}, A_{2p}} \E \big[ (Z_{I_{2p}} - \E Z_{I_{2p}}) (Z_{J_{2p}} - \E Z_{J_{2p}}) \big] = \sum_{( I_{2p}, J_{2p}) \in B_{P_2}} \E \big[ (Z_{I_{2p}} - \E Z_{I_{2p}}) (Z_{J_{2p}} - \E Z_{J_{2p}}) \big],
\end{equation}
where $B_{P_2}$ as in Definition \ref{def:B_{P_l}_rp} for $\ell=2$. 
%Similarly  
%\begin{equation}
%\sum_{A_{2p}, A_{2p}} \E \big[ (Z_{K_{2p}} - \E Z_{K_{2p}}) (Z_{L_{2p}} - \E Z_{L_{2p}}) \big] \leq \sum_{( K_{2p}, L_{2p}) \in B_{P_2}} \E \big[ (Z_{K_{2p}} - \E Z_{K_{2p}}) (Z_{L_{2p}} - \E Z_{L_{2p}}) \big].
%\end{equation}
Since $x_i \distas{D} N(0, t-s)$, $y_i \distas{D} N(0, s)$ and $t, s \in [0,T]$,   there exists a common upper bound, say $\alpha$, for each summand in \eqref{eqn:B_p_2_rp}. That is, for any $( I_{2p}, J_{2p}) \in B_{P_2}$, 
$$\big|\E [(Z_{I_{2p}} - \E Z_{I_{2p}}) (Z_{J_{2p}} - \E Z_{J_{2p}}) ]\big| \leq \alpha.$$
Now using the above inequality in (\ref{eqn:B_p_2_rp}), we get
\begin{align} \label{eqn:case2, B_2_rp}
\Big|\sum_{ A_{2p}, A_{2p}} \E \big[ (Z_{I_{2p}} - \E Z_{I_{2p}}) (Z_{J_{2p}} - \E Z_{J_{2p}}) \big] \Big| \leq \sum_{( I_{2p}, J_{2p}) \in B_{P_2}} \alpha = |B_{P_2}| \alpha.
\end{align} 
Now we calculate cardinality of $B_{P_2}$. First recall $B_{P_2}$ from Definition \ref{def:B_{P_l}_rp}, 
\begin{align*}B_{P_2}=\{( I_{2p}, J_{2p}) \in  A_{2p} \times A_{2p} &: S_{I_{2p}} \cap S_{J_{2p}} \neq \emptyset \mbox{ and each entries of } S_{I_{2p}} \cup S_{J_{2p}}\\
&\quad \mbox{ has multiplicity greater than or equal to two}\}.
\end{align*}
It is easy to see that
\begin{align*}
 |B_{P_2}|  & = \left\{\begin{array}{ccc} 	 
		 	 O({n}^{2p-1}) & \text{if}&  |S_{I_{2p}} \cap S_{J_{2p}}| \mbox{ is even}, \\
			o({n}^{2p-1}) & \text{otherwise}.& 	 	 
		 	  \end{array}\right.
\end{align*}
Hence $|B_{P_2}|= O({n}^{2p-1})$ and therefore from (\ref{eqn:case2, B_2_rp}) we get
\begin{align} \label{eqn:Z_I,J_rp}
\Big|\sum_{ A_{2p}, A_{2p}} \E \big[ (Z_{I_{2p}} - \E Z_{I_{2p}}) (Z_{J_{2p}} - \E Z_{J_{2p}}) \big] \Big|\leq O(n^{2p-1}).
\end{align} 
%Similarly, we have 
%\begin{align} \label{eqn:Z_K,L}
%\sum_{ A_{p}, A_{p}} \Big|\E \big[ (Z_{J^3_{p}} - \E Z_{J^3_{p}}) (Z_{J^4_{p}} - \E Z_{J^4_{p}}) \big] I(i_3, J^3_p) I(i_4, J^4_p)\Big| \leq \alpha O({b_n}^{p-1}).
%\end{align}
By using (\ref{eqn:Z_I,J_rp}) 
%and (\ref{eqn:Z_K,L}) 
in (\ref{eqn:case II_rp}), we get
\begin{align*}
%\E[w_{p}(t) -w_{p}(s)]^4 
|S_1|\leq (t-s)^2 \alpha^2 O(1).
\end{align*}
Since $t, s \in [0,T]$,  there exist $M_2 >0$, depending only on $p, T$ such that 
\begin{align} \label{case II,(t-s)_rp}
%\E[w_{p}(t) -w_{p}(s)]^4 
|S_1|\leq M_2 (t-s)^2 \ \mbox{ for all }n\geq 1.
\end{align} \\
\noindent \textbf{Case III.} \textbf{$I_{2p}, J_{2p}, K_{2p}, L_{2p}$ all are connected to each other, that is, $I_{2p}, J_{2p}, K_{2p}, L_{2p}$ forms a cluster:}  
In this situation, we get
\begin{align} \label{eqn:B_p_2,2_rp}
&\sum_{ A_{2p}, A_{2p}, A_{2p}, A_{2p}}  \E \Big[ (Z_{I_{2p}} - \E Z_{I_{2p}}) (Z_{J_{2p}} - \E Z_{J_{2p}}) (Z_{K_{2p}} - \E Z_{K_{2p}}) (Z_{L_{2p}} - \E Z_{L_{2p}}) \Big]  \\ 
& = \sum_{( I_{2p}, J_{2p}, K_{2p}, L_{2p}) \in B_{P_4}} \E \Big[ (Z_{I_{2p}} - \E Z_{I_{2p}}) (Z_{J_{2p}} - \E Z_{J_{2p}}) (Z_{K_{2p}} - \E Z_{K_{2p}}) (Z_{L_{2p}} - \E Z_{L_{2p}}) \Big],\nonumber
\end{align}
 where $B_{P_4}$ as in Definition \ref{def:B_{P_l}_rp}. By the similar arguments as given in Case II, there exist a common upper bound, say $\beta$,  such that for all $(I_{2p},J_{2p},K_{2p},L_{2p})\in B_{P_4}$
$$\big|\E [(Z_{I_{2p}} - \E Z_{I_{2p}}) (Z_{J_{2p}} - \E Z_{J_{2p}}) ] (Z_{K_{2p}} - \E Z_{K_{2p}}) (Z_{L_{2p}} - \E Z_{L_{2p}}) \big| \leq \beta. $$
Now using the above inequality in (\ref{eqn:B_p_2,2_rp}), we get 
\begin{align*}
&\sum_{ A_{2p}, A_{2p}, A_{2p}, A_{2p}} \Big| \E \big[ (Z_{I_{2p}} - \E Z_{I_{2p}}) (Z_{J_{2p}} - \E Z_{J_{2p}}) (Z_{K_{2p}} - \E Z_{K_{2p}}) (Z_{L_{2p}} - \E Z_{L_{2p}}) \big]\Big|\\
 &\qquad \leq \sum_{( I_{2p}, J_{2p}, K_{2p}, L_{2p}) \in B_{P_4}} \beta = |B_{P_4}| \beta.\nonumber 
\end{align*} 
Since from Result \ref{res:cluster_rp}, $|B_{P_4}| = o(n^{4p-2})$ and  therefore 
% \begin{align} \label{eqn:case III,Z}
% \sum_{ A_{2p}, A_{2p}, A_{2p}, A_{2p}} \Big|\E \big[ (Z_{I_{2p}} - \E Z_{I_{2p}}) (Z_{J_{2p}} - \E Z_{J_{2p}}) (Z_{K_{2p}} - \E Z_{K_{2p}}) (Z_{L_{2p}} - \E Z_{L_{2p}}) \big] \Big|= o(n^{4p-2}).
% \end{align}
% Now using (\ref{eqn:case III,Z}) in (\ref{eqn:w_p^4}), we get
\begin{align*}
\frac{(t-s)^2}{n^{4p-2}}\sum_{A_{2p}, A_{2p}, A_{2p}, A_{2p}} \Big| \E \big[ (Z_{I_{2p}} - \E Z_{I_{2p}}) (Z_{J_{2p}} - \E Z_{J_{2p}})(Z_{K_{2p}} - \E Z_{K_{2p}}) (Z_{L_{2p}} - \E Z_{L_{2p}}) \big]\Big|\leq (t-s)^2 o(1).
\end{align*}
Hence there exist $M_2 >0$, depending only on $p$ and  $T$, such that for all $n$
\begin{equation} \label{case III,(t-s)_rp}
%\E[w_{2p}(t) - w_{2p}(s)]^4 \leq M_2 (t-s)^2 .
\frac{(t-s)^2}{n^{4p-2}}\sum_{A_{2p}, A_{2p}, A_{2p}, A_{2p}} \Big| \E \big[ (Z_{I_{2p}} - \E Z_{I_{2p}}) (Z_{J_{2p}} - \E Z_{J_{2p}})(Z_{K_{2p}} - \E Z_{K_{2p}}) (Z_{L_{2p}} - \E Z_{L_{2p}}) \big]\Big|\leq M_2(t-s)^2.
\end{equation}

Finally on combining \eqref{case II,(t-s)_rp} and \eqref{case III,(t-s)_rp}, it is clear that there exist a positive constant $M_T$, depending only on $p, T$ such that
\begin{align} \label{eqn:M_T_rp}
%\E[w_p(t) - w_p(s)]^4 & \leq M^T_7 (t-s)^2 + M^T_{10} (t-s)^2 \nonumber \\
\E[w_{p}(t) - w_{p}(s)]^4 & \leq M_T (t-s)^2 \ \ \ \forall \ n \in \mathbb{N} \ \mbox{and } t,s \in[0,\ T].
\end{align}
This completes the proof of Proposition \ref{pro:tight_rp}.
\end{proof}
\begin{remark}
Note that the constant $M_T$ of (\ref{eqn:M_T_rp}) depends on $p$. So there is a possibility that it tends to infinity as $p \tends \infty$. Therefore, from Theorem \ref{thm:revprocess} we can not conclude anything about the process convergence of $\{ w_{p}(t) ; t \geq 0, p \geq 1\}$. But from the proof of Theorem \ref{thm:revprocess}, it follows that 
	%we can conclude that 
	as $n \tends \infty$
	\begin{equation*}
	\{ w_{p}(t) ; t \geq 0, 1 \leq p \leq N\} \stackrel{\mathcal D}{\rightarrow} \{N_{p}(t) ; t \geq 0 , 1 \leq p \leq N\}, 
	\end{equation*}  
	for any fixed $N \in \mathbb{N}$.
\end{remark}

\begin{remark} \label{rem:poly_test_function_rp}
For a given real polynomial  $Q(x)=\sum_{k=1}^da_kx^{2k}$
with degree $2d$, $d\geq 1$ and  even degree terms. If we define
$$w_Q(t) := \frac{1}{\sqrt{n}} \bigl\{ \Tr(Q(RC_n(t))) - \E[\Tr(Q(RC_n(t)))]\bigr\}.$$
Then we can extend our results, Theorem \ref{thm:revcovar_rp}, \ref{thm:revmulti_rp} and \ref{thm:revprocess} for $w_Q(t)$. The proof will go similar to the proof of Theorem \ref{thm:revcovar_rp}, \ref{thm:revmulti_rp} and \ref{thm:revprocess}.
\end{remark}

\section{Proof of Theorem \ref{thm:symcovar_sp}}\label{sec:cov_sp}
First we recall a convenient formula for trace of $(SC_n)^p$ from \cite{SC_independent_2020},
\begin{equation}\label{trace,formula_sp}
\Tr(SC_n)^p
   = \left\{\begin{array}{ccc} 	 
		 	\displaystyle n\sum_{k=0}^{p}\binom{p}{\ell}X_0^{p-k}\sum_{J_{k} \in A_{k}} X_{J_{k}}  & \text{if}& \mbox{ $n$ is odd}\\\\
			\displaystyle \frac{n}{2} \sum_{k=0}^{p}\binom{p}{k} \Big[ Y_k  \sum_{J_{k} \in A_k} X_{J_{k}} + \tilde{Y}_k \sum_{J_{k} \in \tilde{A}_k}  X_{J_{k}} \Big] & \text{if}& \mbox{ $n$ is even}, 	 	 
		 	  \end{array}\right.	
 \end{equation}

%\begin{align}\label{trace formula_SC_odd}
%\Tr(SC_n)^p=n\sum_{k=0}^{p}\binom{p}{\ell}X_0^{p-k}\sum_{A_{k}} X_{j_1}X_{j_2}\ldots X_{j_{k}} \ \ \ \ \mbox{if  $n$ is odd.}
%\end{align}
%where  for each  $k=0, 1, 2,\ldots, p$, $A_{k}$ is as defined in (\ref{def:A_p_SC}).
%\begin{align}\label{trace formula_SC_even}
%\Tr(SC_n)^p = \frac{n}{2} \sum_{k=0}^{p}\binom{p}{k} \Big[ Y_k  \sum_{A_k} X_{J_{k}} + \tilde{Y}_k \sum_{ \tilde{A}_k}  X_{J_{k}} \Big] \ \ \ \ \mbox{if  $n$ is even.}
%\end{align}
where  for each  $k=0, 1, 2,\ldots, p$, $A_{k}$ and ${ \tilde{A}_k}$ are defined as
\begin{align} \label{def:A_p_sp}
A_{k}&=\{(j_1,\ldots,j_{k})\in\mathbb N^k \suchthat \sum_{i=1}^{k}\epsilon_i j_i=0\; \mbox{(mod n)}, \epsilon_i\in\{+1,-1\}, 1\le j_1,\ldots,j_{k}\le \frac{n}{2}\}, \\
\tilde{A}_{k} &=\Big\{(j_1,\ldots,j_{k})\in\mathbb N^k\suchthat \sum_{i=1}^{k}\epsilon_i j_i=0\; \mbox{(mod } \frac{n}{2}) \mbox{ and } \sum_{i=1}^{k}\epsilon_i j_i \neq 0\; \mbox{(mod }n), \nonumber \\
& \qquad \ \  \epsilon_i\in\{+1,-1\}, 1\le j_1,\ldots,j_{k}< \frac{n}{2}\Big\}\nonumber,
\end{align}
 and 
\begin{align} \label{eqn:Y_k}
Y_k & = {(X_0 +  X_{\frac{n}{2}})}^{p-k} + {(X_0 -  X_{\frac{n}{2}})}^{p-k}, \ \tilde{Y}_k = {(X_0 +  X_{\frac{n}{2}})}^{p-k} - {(X_0 -  X_{\frac{n}{2}})}^{p-k}, \
X_{J_{k}}  = X_{j_1}X_{j_2}\ldots X_{j_{k}}. \nonumber 
\end{align}
Here note that $A_0$ and$\tilde{A}_0$ are empty set with the understanding that the contribution from the sum corresponding to $A_0$ and $\tilde{A}_0$ are $1$, and in $A_k$ and $\tilde{A}_k$, $(j_1,\ldots,j_{k})$ are collected according to their multiplicity. 

If $SC_n$ and $SC_n'$ are symmetric circulant matrices with input entries $\{ X_i\}$ and $\{X_i'\}$, respectively, then similar to (\ref{trace,formula_sp}), we can also derive

%\begin{align} \label{eqn:trace,product}
%	\Tr [(SC_n)^p (SC'_n)^{q}] & = \frac{1}{n^{\frac{p+q}{2}-1}} \sum_{m=0}^p \sum_{r=0}^{q} \binom{p}{m} \binom{q}{r} (X_{0})^{p-m} (X_{0}')^{q-r} \sum_{A_{m+r}} X_{J_{m}} X'_{J_{r}}(s), \nonumber 
%\end{align}
\begin{equation}\label{trace,product_sp}
\Tr [(SC_n)^p (SC'_n)^{q}] 
   = \left\{\begin{array}{ccc} 	 
		 	\displaystyle n\sum_{m=0}^p \sum_{r=0}^{q} \binom{p}{m} \binom{q}{r} (X_{0})^{p-m} (X_{0}')^{q-r} \sum_{ J_{m+r} \in A_{m+r}} X_{J_{m}} X'_{J_{r}}  & \text{if}& \mbox{ $n$ is odd}\\\\
			\displaystyle \frac{n}{2} \sum_{m=0}^p \sum_{r=0}^{q} \binom{p}{m} \binom{q}{r} \Big[ Q_k  \sum_{J_{m+r} \in A_{m+r}} X_{J_{m}}X'_{J_{r}} \\
			  \displaystyle +  \tilde{Q}_k \sum_{J_{m+r} \in \tilde{A}_{m+r}}  X_{J_{m}}X'_{J_{r}} \Big] & \text{if}& \mbox{ $n$ is even}, 	 	 
		 	  \end{array}\right.	
 \end{equation}
 where  for each  $m=0, 1, 2,\ldots, p$, $r=0, 1, 2,\ldots, q$,
$$J_{m+r}= (j_1, j_2, \ldots, j_m, j_{m+1}, j_{m+2}, \ldots, j_{m+r}),$$ 
   $A_{m+r}$, ${ \tilde{A}_{m+r}}$ are as defined in (\ref{def:A_p_sp}) and
\begin{align*} 
Q_k & = {(X_0 +  X_{\frac{n}{2}})}^{p-k} {(X'_0 +  X'_{\frac{n}{2}})}^{q-r}+ {(X_0 -  X_{\frac{n}{2}})}^{p-k} {(X'_0 -  X'_{\frac{n}{2}})}^{q-r}, \\
\tilde{Q}_k &= {(X_0 +  X_{\frac{n}{2}})}^{p-k}{(X'_0 +  X'_{\frac{n}{2}})}^{q-r} - {(X_0 -  X_{\frac{n}{2}})}^{p-k} {(X'_0 -  X'_{\frac{n}{2}})}^{q-r}, \\
X_{J_{m}}X'_{J_r} & = X_{j_1}X_{j_2}\ldots X_{j_{m}}X'_{j_{m+1}}X'_{j_{m+2}}\ldots X'_{j_{m+r}}. \nonumber 
\end{align*}

Here observe from the definition of $A_k$ and $\tilde{A}_k$, as in (\ref{def:A_p_sp}), that $|A_k|= O(n^{k-1})$, because the entries of $A_k$ has one constraint, whereas $|\tilde{A}_k|= O(n^{k-2})$, because the entries of $\tilde{A}_k$ has two constraints. Therefore
\begin{equation} \label{eqn:A,tildeA}
|\tilde{A}_k| < |A_k|.
\end{equation}
Now we state a result from \cite{adhikari_saha2017} which will be used in the proof of  Theorem \ref{thm:symcovar_sp}. 
For its proof see \cite[Lemma 15]{adhikari_saha2017}. 
%For the proof the result, we refer the readers to \cite[Lemma 14]{adhikari_saha2017}.
%number of elements of $B_{2p,s}$, which will be used in the prove of the Theorem \ref{thm:reverseciculant}.
\begin{result} \label{result:def_h}
	Suppose
	$$A_p^{(k)}=\Big\{(j_1,\ldots,j_p)\in\mathbb N^p\suchthat j_1+\cdots+j_k - j_{k+1}-\cdots -j_p=0 \mbox{ (mod $n$)},1\leq j_1,\ldots,j_p\leq \frac{n}{2}\Big\}.$$ 
%	$$A_{p,s}^{(k)}=\Big\{(j_1,\ldots,j_p)\suchthat j_1+\cdots+j_k - j_{k+1}-\cdots -j_p=sn, 1\leq j_1,\ldots,j_p\leq \frac{n}{2}\Big\}.$$
%	$|A_{p,s}|$ denotes   the cardinality of $A_{p,s}$. 
	Then 
	$$
	h_p(k):=
\lim_{n\to \infty}\frac{|A_p^{(k)}|}{n^{p-1}}= \frac{1}{(p-1)!}\sum_{s=-\lceil\frac{p-k}{2}\rceil}^{\lfloor\frac{k}{2}\rfloor}\sum_{q=0}^{2s+p-k}(-1)^q\binom{p}{q}\l(\frac{2s+p-k-q}{2}\r)^{p-1}, 
$$
where $\lceil x\rceil$ denotes the smallest integer not less than $x$.	
\end{result}
Now for a given vector $(j_1, j_2, \ldots, j_p)\in \mathbb N^p$, we define {\it opposite sign pair matched} elements of the vector. 
\begin{definition}\label{def:opposite_sign}
Suppose $(j_1, j_2, \ldots, j_p) \in \mathbb N^p$ and $(\epsilon_1,\epsilon_2,\ldots,\epsilon_p)\in \{+1,-1\}^p$. We say $j_k, j_\ell$ are {\it opposite sign pair matched} given $(\epsilon_1,\epsilon_2,\ldots,\epsilon_p)$, if $\epsilon_k$ and $\epsilon_\ell$  are of opposite sign $j_k= j_\ell$ and $j_k$ appear exactly twice in $(j_1, j_2, \ldots, j_p).$
%Suppose $(j_1, j_2, \ldots, j_p) \in A_p$. 
% We say $j_k, j_\ell$ is {\it opposite sign pair matched}, if $\epsilon_k$ and $\epsilon_\ell$ corresponding to $j_k$ and $j_\ell$, respectively, are of opposite sign and $j_k= j_\ell$, where $\epsilon_k$ and $\epsilon_\ell$  corresponds to  (\ref{eqn:A_k_symmetric}).
%% one of the elements of the pair have positive sign and other one have negative sign. 

 For example, in $(2,3,5,2)$, $2$ is {\it opposite sign pair matched}, if $\epsilon_1=1$ and $\epsilon_4=-1$ or $\epsilon_1=-1$ and $\epsilon_4=1$ whereas if $\epsilon_1$ = $\epsilon_4= 1$ or $\epsilon_1$ = $\epsilon_4= -1$, then $2$ is not  {\it opposite sign pair matched}. 
 
% Similarly, we  define {\it opposite sign pair matched} vectors given a vector $(\epsilon_1,\epsilon_2,\ldots,\epsilon_p)\in\{+1,-1\}^p$. We 
We say that a vector $(j_1, j_2, \ldots, j_p)$ is {\it opposite sign pair matched} for a given $(\epsilon_1,\epsilon_2,\ldots,\epsilon_p)$, if all its entries are {\it opposite sign pair matched}.
\end{definition}
 Observe that if $(j_1, j_2, \ldots, j_p)\in A_{p}$ and each entry of $\{j_1, j_2, \ldots, j_p\}$ has multiplicity greater than or equal to two, then the maximum number of free variable in $(j_1, j_2, \ldots, j_p)$ will be $\frac{p}{2}$ only when $p$ is even and $(j_1, j_2, \ldots, j_p)$ is {\it opposite sign pair matched}. We shall use this observation in the proof of Theorem \ref{thm:symcovar_sp}.

 We first start with the following lemma which will be proved by using the trace formula of $SC_n$. 
\begin{lemma}\label{lem:w_p_sp}
Suppose $0<t_1 \leq t_2$ and $p, q\geq 2$.
Then
\begin{align}\label{eqn:cov_e,o_sp}
\Cov & \big(\eta_p(t_1),\eta_q(t_2)\big) \nonumber \\
   &= \left\{\begin{array}{ccc} 	 
		\displaystyle \frac{1}{n^{\frac{p+q}{2}-1}} \sum_{k, \ell =0}^{p, q}\binom{p}{k} \binom{q}{\ell} \sum_ {r=0}^{\ell} \binom{\ell}{r}  \sum_{A_{k}, A_{\ell}} \Big\{\E[u_0^{p-k} (u_0+v_0)^{q-\ell} U_{I_k}U_{J_r}V_{J_{r+1,\ell}} ]     \\
   -\E[ u_0^{p-k} U_{I_k}] \E[(u_0+v_0)^{q-\ell}U_{J_r} V_{J_{r+1,\ell}}  ]        \Big\}  & \text{if}& \mbox{ $n$ is odd}\\\\
			\displaystyle  \frac{1}{4n^{\frac{p+q}{2}-1}} \sum_{k, \ell =0}^{p, q}\binom{p}{k} \binom{q}{\ell}  \Big\{ \E \Big( \big[ Y_{k} \sum_{A_{k}}  U_{I_k} +\tilde{Y}_{k} \sum_{\tilde{A}_{k}} U_{I_{k}} \big]  \big[Z_{\ell} \sum_{A_{\ell}}  (U+V)_{J_\ell}   \\
   +\displaystyle \tilde{Z}_{\ell} \sum_{\tilde{A}_{\ell}} (U+V)_{J_{\ell}} \big] \Big) -\E \big[ Y_{k} \sum_{A_{k}}  U_{I_k} +\tilde{Y}_{k} \sum_{\tilde{A}_{k}} U_{I_{k}} \big] \\
 \displaystyle  \E \big[Z_{\ell} \sum_{A_{\ell}}  (U+V)_{J_\ell} +\tilde{Z}_{\ell} \sum_{\tilde{A}_{\ell}} (U+V)_{J_{\ell}} \big]  \Big\} & \text{if}& \mbox{ $n$ is even}, 	 	 
		 	  \end{array}\right.	
 \end{align}
where
$u_0=b_0(t_1), v_0=b_0(t_2)- b_0(t_1)$ and $U_{I_{k}},  (U+V)_{J_\ell},  Y_{k}, \tilde{Y}_{k}, Z_{\ell}, \tilde{Z}_{\ell}$ are given in (\ref{eqn:U,Y_sp}). 
\end{lemma}
\begin{proof}
Define
\begin{align*}
 U&:=SC_n(t_1)=  \frac{1}{\sqrt{n}} \Big(u_{j-i \;{(\mbox{mod $n$})}}\Big)_{i,j=1} ^ {n} \mbox{ and }\\
 V&:=SC_n(t_2)-SC_n(t_1) =  \frac{1}{\sqrt{n}} \Big(v_{j-i \;{(\mbox{mod $n$})}}\Big)_{i,j=1} ^ {n}.
 \end{align*}
As $\E(\eta_p(t_1))=\E(\eta_q(t_2))=0$, we have 
\begin{align} \label{eqn:cov_trace_sp}
 \Cov\big(\eta_p(t_1),\eta_q(t_2)\big)&=\E[\eta_p(t_1) \eta_q(t_2)] \nonumber \\
 &= \frac{1}{n} \Big\{ \E[\Tr(SC_n(t_1))^p\Tr(SC_n(t_2))^q ]- \E[\Tr(SC_n(t_1))^p]\E[\Tr(SC_n(t_2))^q]   \Big\}  \nonumber \\
 &= \frac{1}{n} \Big\{ \E[\Tr(U)^p\Tr(U+V)^q ]- \E[\Tr(U)^p]\E[\Tr(U+V)^q]   \Big\}. 
 \end{align}
 Using the trace formula (\ref{trace,formula_sp}) for odd $n$, we get
 \begin{align*}
	\E[\Tr(U)^{p}]&= \E\Big[n\sum_{k=0}^{p}\binom{p}{k} (\frac{u_{0}}{ \sqrt{n}}) ^{p-k}\sum_{A_{k}}\frac{u_{i_1}}{ \sqrt{n}} \cdots \frac{u_{i_{k}}}{ \sqrt{n}} \Big]
	= \frac{1}{ n^ {\frac{p}{2}-1}} \E\Big[ \sum_{k=0}^{p}\binom{p}{k}u_0^{p-k}\sum_{A_{k}} u_{i_1} \cdots u_{i_{k}}\Big],\\
	 \E[\Tr(U+V)^q]&=  \frac{1}{ n^ {\frac{q}{2}-1}}\E\Big[ \sum_{\ell=0}^{q}\binom{q}{\ell} (u_0+v_0)^{q-\ell} \sum_{A_{\ell}} (u_{j_1} + v_{j_1}) \cdots  (u_{j_\ell} + v_{j_\ell})\Big].
	\end{align*}
	Therefore for odd value of $n$, (\ref{eqn:cov_trace_sp}) will be
\begin{align} \label{eqn:cov_o_u_i}
& \Cov\big(\eta_p(t_1),\eta_q(t_2)\big) \nonumber \\
&=  \frac{1}{ n^ {\frac{p+q}{2}-1}} \Bigg[ \E\Big\{ \Big( \sum_{k=0}^{p}\binom{p}{k}u_0^{p-k}\sum_{A_{k}} u_{i_1} \cdots u_{i_{k}} \Big) \Big(  \sum_{\ell=0}^{q}\binom{q}{\ell} (u_0+v_0)^{q-\ell} \sum_{A_{\ell}} (u_{j_1} + v_{j_1}) \cdots  (u_{j_\ell} + v_{j_\ell}) \Big) \Big\} \nonumber \\
 &\qquad  - \E \Big( \sum_{k=0}^{p}\binom{p}{k}u_0^{p-k}\sum_{A_{k}} u_{i_1} \cdots u_{i_{k}} \Big) \Big(\E \sum_{\ell=0}^{q}\binom{q}{\ell} (u_0+v_0)^{q-\ell} \sum_{A_{\ell}} (u_{j_1} + v_{j_1}) \cdots  (u_{j_\ell} + v_{j_\ell}) \Big)     \Bigg] \nonumber \\
   &=\frac{1}{n^{\frac{p+q}{2}-1}}  \sum_{k, \ell =0}^{p, q}\binom{p}{k} \binom{q}{\ell} \sum_ {r=0}^{\ell} \binom{\ell}{r} \sum_{A_{k}, A_{\ell}} \Big\{\E[ u_0^{p-k} (u_0+v_0)^{q-\ell} u_{i_1} \cdots u_{i_k} u_{j_1} \cdots u_{j_r}v_{j_{r+1}} \cdots v_{j_\ell} ] \\
  &\qquad -\E[(u_0)^{p-k} u_{i_1} \cdots u_{i_\ell} ] \E[ (u_0+v_0)^{q-\ell} u_{j_1}\cdots u_{j_r}v_{j_{r+1}} \cdots v_{j_\ell}  ] \Big\} \nonumber.
\end{align}
Now for $0 \leq k \leq p, 0 \leq \ell \leq q$ and $0 \leq r \leq \ell$, we define
 \begin{align*}
 I_k&= (i_1, i_2, \ldots, i_k),\quad J_\ell=(j_1,j_2, \ldots, j_\ell),	\\
 J_r&= (j_1, j_2, \ldots, j_r),\quad J_{r+1,\ell}= (j_{r+1}, j_{r+2}, \ldots, j_\ell),\\
U_{I_k} &= u_{i_1}\cdots u_{i_k},\quad  U_{J_r} = u_{j_1}\cdots u_{j_r} \ \mbox{ and } \ V_{J_{r+1,\ell}} = v_{j_{r+1}}\cdots v_{j_\ell}.
\end{align*} 	
Then \eqref{eqn:cov_o_u_i} can be written as
  \begin{align*}
  \Cov\big(\eta_p(t_1),\eta_q(t_2)\big) &= \frac{1}{n^{\frac{p+q}{2}-1}} \sum_{k, \ell =0}^{p, q}\binom{p}{k} \binom{q}{\ell} \sum_ {r=0}^{\ell} \binom{\ell}{r}  \sum_{A_{k}, A_{\ell}} \Big\{\E[u_0^{p-k} (u_0+v_0)^{q-\ell} U_{I_k}U_{J_r}V_{J_{r+1,\ell}} ]     \\
  &\qquad -\E[ u_0^{p-k} U_{I_k}] \E[(u_0+v_0)^{q-\ell}U_{J_r} V_{J_{r+1,\ell}}  ]        \Big\}, 
  \end{align*}
  where $U_{J_0}=1$ and $V_{J_{q+1,q}}=1$.
  
Now by the similar calculation, for even value of $n$, (\ref{eqn:cov_trace_sp}) will be
 \begin{align*}
 & \Cov\big(\eta_p(t_1),\eta_q(t_2)\big) \\
  &= \frac{1}{4n^{\frac{p+q}{2}-1}} \sum_{k, \ell =0}^{p, q}\binom{p}{k} \binom{q}{\ell}  \Big\{ \E \Big( \big[Y_{k} \sum_{A_{k}}  U_{I_k} +\tilde{Y}_{k} \sum_{\tilde{A}_{k}} U_{I_{k}} \big]  \big[Z_{\ell} \sum_{A_{\ell}}  (U+V)_{J_\ell} +\tilde{Z}_{\ell} \sum_{\tilde{A}_{\ell}} (U+V)_{J_{\ell}} \big] \Big) \nonumber  \\
  &\qquad -\E\big[ Y_{k} \sum_{A_{k}}  U_{I_k} +\tilde{Y}_{k} \sum_{\tilde{A}_{k}} U_{I_{k}} \big] \E \big[Z_{\ell} \sum_{A_{\ell}}  (U+V)_{J_\ell} +\tilde{Z}_{\ell} \sum_{\tilde{A}_{\ell}} (U+V)_{J_{\ell}} \big]  \Big\}, \nonumber
  \end{align*}
  where $\ A_k, \tilde{A}_k$ are as defined in (\ref{def:A_p_sp}) and 
\begin{align} \label{eqn:U,Y_sp}
U_{I_k} &= u_{i_1}\cdots u_{i_k}, \ (U+V)_{J_\ell} = (u_{j_1} + v_{j_1}) \cdots  (u_{j_\ell} + v_{j_\ell}), \\
Y_k & = {(u_0 +  u_{\frac{n}{2}})}^{p-k} + {(u_0 -  u_{\frac{n}{2}})}^{p-k}, \nonumber \ \ \tilde{Y}_k = {(u_0 +  u_{\frac{n}{2}})}^{p-k} - {(u_0 -  u_{\frac{n}{2}})}^{p-k}, \\
Z_\ell & = \{(u_0+v_0) +  (u_{\frac{n}{2}} +v_{\frac{n}{2}})\}^{q-\ell} + \{(u_0+v_0) - (u_{\frac{n}{2}} +v_{\frac{n}{2}})\}^{q-\ell}, \nonumber \\
\tilde{Z}_\ell & = \{(u_0+v_0) +  (u_{\frac{n}{2}} +v_{\frac{n}{2}})\}^{q-\ell} - \{(u_0+v_0) - (u_{\frac{n}{2}} +v_{\frac{n}{2}})\}^{q-\ell}. \nonumber
\end{align}  
This completes the proof.\end{proof}
Now we are ready to prove Theorem \ref{thm:symcovar_sp}. 
\begin{proof}[Proof of Theorem \ref{thm:symcovar_sp}]
First note that, for odd and even values of $n$ we have different trace formulae. 
%Therefore we shall prove Theorem \ref{thm:symcovar_sp} in two steps. 
In Steps $1$ and $2$, we calculate limit of $\Cov\big(\eta_p(t_1),\eta_q(t_2)\big)$ for odd and even $n$, respectively and verify that they are same.
 \\
	
\noindent \textbf{Step 1.} First we recall  from (\ref{eqn:cov_e,o_sp}) that for odd values of $n$,
	\begin{align*} 
 Cov\big(\eta_p(t_1),\eta_q(t_2)\big) &=\frac{1}{n^{\frac{p+q}{2}-1}} \sum_{k, \ell =0}^{p, q}\binom{p}{k} \binom{q}{\ell} \sum_ {r=0}^{\ell} \binom{\ell}{r}  \sum_{A_{k}, A_{\ell}} \Big\{\E[u_0^{p-k} (u_0+v_0)^{q-\ell} U_{I_k}U_{J_r}V_{J_{r+1,\ell}} ]     \\
  &\qquad -\E[ u_0^{p-k} U_{I_k}] \E[(u_0+v_0)^{q-\ell}U_{J_r} V_{J_{r+1,\ell}}  ]        \Big\}  \\
  & = T_1 +T_2, \mbox{ say},
\end{align*}
where 
\begin{align}\label{eqn:T_2_sp}
T_2 & =\frac{1}{n^{\frac{p+q}{2}-1}} \sum_{k, \ell =0}^{p-1, q-1}\binom{p}{k} \binom{q}{\ell} \sum_ {r=0}^{\ell} \binom{\ell}{r}  \sum_{A_{k}, A_{\ell}} \Big\{\E[u_0^{p-k} (u_0+v_0)^{q-\ell} U_{I_k}U_{J_r}V_{J_{r+1,\ell}} ]     \\
  &\qquad -\E[ u_0^{p-k} U_{I_k}] \E[(u_0+v_0)^{q-\ell}U_{J_r} V_{J_{r+1,\ell}}  ]        \Big\} \nonumber
\end{align}
and $T_1=\Cov\big(\eta_p(t_1),\eta_q(t_2)\big)-T_2$. 
Now we calculate limits of $T_1$ and $T_2$.	
	
Note that in $T_1$, either $k=p, \ell \leq q$ or $\ell=q, k \leq p$. Therefore
\begin{align} \label{eqn:S_1+S_2_sp}
T_1&=\Cov\big(\eta_p(t_1),\eta_q(t_2)\big)-T_2 \nonumber \\
&= \frac{1}{n^{\frac{p+q}{2}-1}} \biggl[ \sum_{\ell =0}^{q} \binom{q}{\ell} \sum_ {r=0}^{\ell} \binom{\ell}{r}  \sum_{A_{p}, A_{\ell}} \Big\{\E[(u_0+v_0)^{q-\ell}] \Big( \E[ U_{I_p}U_{J_r}V_{J_{r+1,\ell}} ] -\E[ U_{I_p}] \E[U_{J_r} V_{J_{r+1,\ell}}  ] \Big)  \Big\} \nonumber \\
& \qquad + \sum_{k=0}^{p-1}\binom{p}{k} \sum_ {r=0}^{q} \binom{q}{r}  \sum_{A_{k}, A_{q}} \Big\{\E[u_0^{p-k}] \Big( \E[ U_{I_k}U_{J_r}V_{J_{r+1,q}} ] -\E[ U_{I_k}] \E[U_{J_r} V_{J_{r+1,q}}  ] \Big)  \Big\}  \biggl] \nonumber \\
& = S_1 +S_2, \mbox{ say}.
\end{align}
Now first we calculate $S_1$ of (\ref{eqn:S_1+S_2_sp}).
% Since in this case, we always get $\E[ X_0^{p+q-k-\ell}]= \E[X_0^{p-k}] \E[X_0^{q-\ell}]$. Therefore, if 
Note that if $\{i_1,i_2,\ldots,i_{p}\}\cap \{j_1,j_2,\ldots, j_{r}\}=\emptyset$ or $\ell=0$ or $r=0$ then from independence of $u_i$ and $v_i$, we get
	$$\E[(u_0+v_0)^{q-\ell}] \Big( \E[ U_{I_p}U_{J_r}V_{J_{r+1,\ell}} ] -\E[ U_{I_p}] \E[U_{J_r} V_{J_{r+1,\ell}}  ] \Big)=0.$$
	%the contribution from the right hand side is zero. 
Hence we can get non-zero contribution in $S_1$ only when there is at least one cross-match among $\{i_1,\ldots,i_{p}\}$ and $\{j_1,\ldots,j_{r}\}$, i.e., $\{i_1,i_2,\ldots,i_{p}\}\cap \{j_1,j_2,\ldots,j_{r}\}\neq\emptyset$ for some $r=1, 2, \ldots, \ell$. 
	So from the above observation, limit of $S_1$ as ${n\to\infty}$ can be written as
	\begin{align} \label{eqn:T_p,l_sp}
	\lim_{n\to\infty}  S_1  &=\lim_{n\to\infty}  \frac{1}{n^{\frac{p+q}{2}-1}} \sum_{ \ell =1}^{q} \binom{q}{\ell} \sum_ {r=1}^{\ell} \binom{\ell}{r} \sum_{m=1}^{ \min \{p,r\} } \sum_{C_m}  \Big\{\E[(u_0+v_0)^{q-\ell}] \Big( \E[ U_{I_p}U_{J_r}V_{J_{r+1,\ell}} ] \nonumber\\
	& \qquad -\E[ U_{I_p}] \E[U_{J_r} V_{J_{r+1,\ell}}  ] \Big)  \Big\}  \nonumber \\	
	& =\lim_{n\to\infty}  \frac{1}{n^{\frac{p+q}{2}-1}} \sum_{\ell =1}^{q} \binom{q}{\ell} \sum_ {r=1}^{\ell} \binom{\ell}{r} \sum_{m=1}^{ \min \{p,r\} } T^m_{p, \ell}, \mbox{ say},
\end{align}		 
	where for each $m=1,2, \ldots, \min\{p, r\}$, $C_m$ is defined as
	\begin{equation*}
	C_{m}:=\{((i_1,\ldots,i_{p}),(j_1,\ldots,j_{r}, j_{r+1}, \ldots, j_\ell))\in A_{p}\times A_{\ell}\suchthat |\{i_1,\ldots,i_{p}\}\cap \{j_1,\ldots,j_{r}\}|=m\}.
	\end{equation*}	
	Now we calculate the contribution due to a typical term $T^m_{p, \ell}$ of (\ref{eqn:T_p,l_sp}) for some fixed value of $r= 1, 2, \ldots, \ell$ and $m =1,2, \ldots, \min\{p, r\}$.  
	Since entries are standard Brownian motion, there exist $\gamma >0$, which depends only on $p$ and $r$, such that 
	\begin{align} \label{eqn:gamma_1}
	|T^m_{p, \ell}| & \leq \sum_{C_m} \big| \E[(u_0+v_0)^{q-\ell}] \big( \E[ U_{I_p}U_{J_r}V_{J_{r+1,\ell}} ] -\E[ U_{I_p}] \E[U_{J_r} V_{J_{r+1,\ell}}  ] \big) \big| \nonumber \\   
	& \leq \gamma  |B_{p,\ell}^m|,
	\end{align}
	where  $B_{p,\ell}^m \subseteq A_p \times A_\ell$ with the conditions that  $|\{ i_1, i_2,\ldots,i_p \} \cap \{ j_1,\ldots, j_{r} \}| =m$ and each element of set $\{ i_1, i_2,\ldots,i_{p}\} \cup \{ j_1, j_2, \ldots, j_r, j_{r+1}, \ldots, j_\ell\} $ has multiplicity greater than or equal to two. So, to find $|T^m_{p,\ell}|$, it is enough to calculate the cardinality of $B_{p,\ell}^m$. 
%	Suppose $((i_1, i_2,\ldots,i_{p}), (j_1,j_2,\ldots,j_{r}))\in B_{p,r}^m$ 
%	for some $m=1,2,\ldots, \min\{p,r\}$.  
%	 Therefore  a typical element of $B_{p, r}^m$ can be written as 
%	 $$((d_1, d_2,\ldots, d_m, i_{m+1}, \ldots, i_{p}), (d_1,d_2,\ldots, d_m, j_{m+1}, \ldots, j_{r}, j_{r+1}, \ldots, j_\ell)).$$
%	Observe that, we shall get maximum number of free entries in $B^m_{p, r}$, if following conditions hold:
%	\begin{enumerate}
%		\item [(i)] Each elements of $ \{d_1, d_2, \ldots, d_{m}\}$ are distinct.
%		\item [(ii)] If $(p-m)$ is even, then 
%		$(i_{m+1}, \ldots, i_{p})$ is {\it opposite sign pair matched}
%and $ \{d_1, d_2, \ldots, d_{m}\} \cap \{i_{m+1},\ldots,i_{p}\}=\emptyset$. Similarly, if $(r-m)$ is even, then $(j_{m+1}, \ldots, j_{r})$ is {\it opposite sign pair matched} and $ \{d_1, d_2, \ldots, d_{m}\} \cap \{j_{m+1}, \ldots, j_{r}\}=\emptyset$.
%		\textcolor{red}{\item [(iii)] If $(p-m)$ is odd, then $\{i_{m+1},\ldots,i_{p}\} \setminus \{i^*\}$ is {\it opposite sign pair matched} and
% $ \{d_1, d_2, \ldots, d_{m}\} \setminus \{d_s\} \cap \{i_{m+1},\ldots,i_{p}\} \setminus \{i^*\}=\emptyset$, where $i^{*}$ is {\it opposite sign pair matched} with $d_s$ for some $s=1,2, \ldots, m$.  Similar conditions also hold when $(r-m)$ is odd.}
% \item [(iv)] $\{ j_{r+1}, j_{r+2}, \ldots, j_\ell\}$ is {\it opposite sign pair matched}
%	 \end{enumerate}
%Under the above observations,  the cardinality of $B_{p,r}^m$ will be
 Now by arguments similar to those in \cite{SC_independent_2020}, estimate of the cardinality of $B_{p,\ell}^m$ is 
	\begin{align}\label{card_B_k_l}
	 |B_{p,\ell}^m|  & = 
%	 \left\{\begin{array}{ll} 	 
%		 	 O(n^{m-1+\frac{p-m}{2} +\frac{r -m}{2}}) & \text{if}\ (p-m) \mbox{ and } (r-m) \mbox{ both are even,}\\\\
%		 	  O(n^{m-3+\frac{p-m+1}{2} +\frac{r -m+1}{2}}) & \text{if}\ (p-m) \mbox{ and } (r-m) \mbox{ both are odd,}\\\\
%			O(n^{m-2+\frac{p-m+r -m+1}{2}})& \text{otherwise}, 	 	 
%		 	  \end{array}\right.	 \nonumber \\
	  \left\{\begin{array}{ll} 	 
		 	 O(n^{\frac{p+\ell}{2}-1}) & \text{if}\ (p-m), \ (r-m) \mbox{ and } (\ell-r) \mbox{ are even,}\\\\
			o(n^{\frac{p+\ell}{2}-1}) & \text{otherwise}. 	 	 
		 	  \end{array}\right.		 
	 \end{align}
%	where $[x]$ denotes the greatest integer less than or equal to $x$. 
Therefore from (\ref{eqn:gamma_1}) and (\ref{card_B_k_l}), we get
	\begin{align} \label{card_T_p,l_sp}
	|T^m_{p,\ell}|  & = \left\{\begin{array}{ll} 	 
		 	 O(n^{\frac{p+\ell}{2}-1}) & \text{if, either}\ p, r, m \mbox{ and } \ell \mbox{ all are even or all are odd},\\\\
			o(n^{\frac{p+\ell}{2}-1}) & \text{otherwise}. 	 	 
		 	  \end{array}\right.	
	\end{align}
Now from (\ref{card_T_p,l_sp}), we get that $T^m_{p,\ell}$ has non-zero contribution in (\ref{eqn:T_p,l_sp}) only when $\ell=q$. Moreover $T^m_{p,q}$ has non-zero contribution only when either $p,q, m$ and $r$ are all even or they are all odd. So, using (\ref{card_T_p,l_sp}) in (\ref{eqn:T_p,l_sp}), we get 
\begin{equation} \label{eqn:S_1_sp}
\lim_{n\to\infty} S_1 = \lim_{n\to\infty} \frac{1}{n^{\frac{p+q}{2}-1}} \sum_ {r=1}^{q} \binom{q}{r} \sum_{m=1}^{ \min \{p,r\} } \sum_{C_m} \Big\{\E[U_{I_p}U_{J_r}V_{J_{r+1,q}} ] -\E[ U_{I_p}] \E[U_{J_r} V_{J_{r+1,q}}  ]\Big\}.
\end{equation}	
Now we calculate the second term $S_2$ of (\ref{eqn:S_1+S_2_sp}) in a similar way. 
%The calculation of $S_2$ will be similar to  that of $S_1$. 
%Note that in calculation of $S_2$, similar arguments will go through as we have done for $S_1$. 
So we outline only the main steps. First recall $S_2$ from (\ref{eqn:S_1+S_2_sp}),
	$$
S_2 = \frac{1}{n^{\frac{p+q}{2}-1}} \sum_{k=0}^{p-1}\binom{p}{k} \sum_ {r=0}^{q} \binom{q}{r}  \sum_{A_{k}, A_{q}} \Big\{\E[u_0^{p-k}] \Big( \E[ U_{I_k}U_{J_r}V_{J_{r+1,q}} ] -\E[ U_{I_k}] \E[U_{J_r} V_{J_{r+1,q}}  ] \Big)  \Big\}.
$$
	Here again note that, if $\{i_1,i_2,\ldots,i_{k}\}\cap \{j_1,j_2,\ldots,j_{r}\} = \emptyset$ or $k=0$ or $r=0$ then contribution due to that typical term in $S_2$ is zero. So
	 \begin{align} \label{eqn:S_2,2_sp}
\lim_{n\to\infty} S_2 &= \lim_{n\to\infty} \frac{1}{n^{\frac{p+q}{2}-1}} \sum_{k=1}^{p-1}\binom{p}{k} \sum_ {r=1}^{q} \binom{q}{r}  \sum_{m=1}^{ \min \{k,r\} } \sum_{C'_m}  \Big\{\E[u_0^{p-k}] \Big( \E[ U_{I_k}U_{J_r}V_{J_{r+1,q}} ] \nonumber \\
& \qquad -\E[ U_{I_k}] \E[U_{J_r} V_{J_{r+1,q}}  ] \Big)  \Big\}\nonumber \\
&= \lim_{n\to\infty} \frac{1}{n^{\frac{p+q}{2}-1}} \sum_{k=1}^{p-1}\binom{p}{k} \sum_ {r=1}^{q} \binom{q}{r} \sum_{m=1}^{ \min \{k,r\} } \sum_{C'_m}  T^m_{k,q}, \mbox{ say},
\end{align}
where for each $m=1,2, \ldots, \min\{k,r\}$, $C'_m$ is defined as
	\begin{equation*}
	C'_{m}:=\{((i_1,\ldots,i_{k}),(j_1,\ldots, j_r, j_{r+1}, \ldots, j_{q}))\in A_{k}\times A_{q}\suchthat |\{i_1,\ldots,i_{k}\}\cap \{j_1,\ldots,j_{r}\}|=m\}.
	\end{equation*}
Now following the argument that was used to establish (\ref{card_T_p,l_sp}), we get	
%Now similar to calculation of (\ref{card_T_p,l}), we get
	\begin{align} \label{card_T_k,q_sp}
	|T^m_{k,q}|  & = \left\{\begin{array}{ll} 	 
		 	 O(n^{\frac{k+q}{2}-1}) & \text{if, either}\ k, r, m \mbox{ and } q \mbox{ all are even or all are odd},\\\\
			o(n^{\frac{k+q}{2}-1}) & \text{otherwise}. 	 	 
		 	  \end{array}\right.	
	\end{align}
	On using (\ref{card_T_k,q_sp}) in (\ref{eqn:S_2,2_sp}), we get
	%\begin{equation} \label{eqn:S_2_sp}
	$\displaystyle{\lim_{n\to\infty} S_2  =0}$.
	%\end{equation}
	Finally  using this and (\ref{eqn:S_1_sp})  in (\ref{eqn:S_1+S_2_sp}), we get
	\begin{align} \label{eqn:lim_cov_o,k=p,1_sp}
	\lim_{n\to\infty}  T_1
	&= \displaystyle\lim_{n\to\infty} \frac{1}{n^{\frac{p+q}{2}-1}} \sum_ {r=1}^{q} \binom{q}{r} \sum_{m=1}^{ \min \{p,r\} } \sum_{C_m} \Big\{\E[U_{I_p}U_{J_r}V_{J_{r+1,q}} ] -\E[ U_{I_p}] \E[U_{J_r} V_{J_{r+1,q}}  ]\Big\}		\nonumber \\
	&=  \left\{\begin{array}{ll} 	 
	        \displaystyle\lim_{n\to\infty} \frac{1}{n^{\frac{p+q}{2}-1}} \sum_ {r=2,4,}^{q} \binom{q}{r}  \sum_{m=1}^{ \min\{ \frac{p}{2},\frac{r}{2} \} } \sum_{C_{2m}} \Big\{\E[U_{I_p}U_{J_r}V_{J_{r+1,q}} ] \\
	\quad -\E[ U_{I_p}] \E[U_{J_r} V_{J_{r+1,q}}  ]\Big\}  & \text{if}\ p,  q \mbox{ both are even,}\\\\
		 	 \displaystyle\lim_{n\to\infty}\frac{1}{n^{\frac{p+q}{2}-1}} \sum_ {r=1,3,}^{q} \binom{q}{r}  \sum_{m=0}^{ \min\{ \frac{p-1}{2},\frac{r-1}{2} \} } \sum_{C_{2m+1}} \Big\{\E[U_{I_p}U_{J_r}V_{J_{r+1,q}} ] \\
	\quad -\E[ U_{I_p}] \E[U_{J_r} V_{J_{r+1,q}}  ]\Big\}& \text{if}\ p,  q \mbox{ both are odd,}\\\\
			0 &  \text{otherwise.} \\\
			\end{array}\right. 	 		 	  	 	  
	\end{align}
Now by the arguments similar to those used in the proof of Theorem $1$ of \cite{SC_independent_2020}, we get
%\cite{SC_independent_2020}, to calculate $(17)$, we get
 \begin{align} \label{eqn:lim_cov_o,k=p_sp}
	\lim_{n\to\infty}  T_1
	&=  \left\{\begin{array}{ll} 	 
		 	\displaystyle \sum_ {r=2,4,}^{q} \binom{q}{r}  t_1^{ \frac{p+r}{2}} (t_2-t_1)^{ \frac{q-r}{2}} \Big\{  \frac{2a_1} {2^{\frac{p+q-4}{2}}  } +\sum_{m=2}^{ \min\{ \frac{p}{2},\frac{r}{2} \} } \Big ( \frac{a_m}{2^{\frac{p+r-4m}{2}}  } \\
		 	\qquad \displaystyle \sum_{s=0}^{2m}\binom{2m}{s}^2 s!(2m-s)! \ h_{2m}(s) \Big) \Big\} & \text{if}\ p, \mbox{} q \mbox{ both are even},\\\\
		 	\displaystyle \sum_ {r=1,3,}^{q} \binom{q}{r}  t_1^{ \frac{p+r}{2}} (t_2-t_1)^{ \frac{q-r}{2}}  \sum_{m=0}^{ \min\{ \frac{p-1}{2},\frac{r-1}{2} \} } \Big( \frac{b_m}{2^{\frac{p+r-4m-2}{2}}  } \\
		 \qquad	\displaystyle \sum_{s=0}^{2m+1}\binom{2m+1}{s}^2 s!(2m+1-s)! \ h_{2m+1}(s) \Big) & \text{if}\ p, \mbox{} q \mbox{ both are odd} \\	\\	 	
%		\displaystyle + pq \binom{p-1}{(p-1)/2}\binom{q-1}{(q-1)/2} \l((p-1)/2\r)! \l((q-1)/2\r)!\frac{1}{2^{(\frac{p+q}{2}-1)}} 	 & \text{if}\ p, \mbox{} q \mbox{ both are odd}\\\\
			0 & \text{otherwise}, 	 
		 	  \end{array}\right.			 	  
	\end{align}
	where $h_{2m+1}(s)$ is defined in Result \ref{result:def_h} and
	\begin{align*} 
	 a_m &=	\binom{p}{p-2m}\binom{p-2m}{\frac{p-2m}{2}} (\frac{p-2m}{2})! \binom{r}{r-2m}\binom{r-2m}{\frac{r-2m}{2}} (\frac{r-2m}{2})! \\
	  &\qquad  \binom{q}{q-2r}\binom{q-2r}{\frac{q-2r}{2}} (\frac{q-2r}{2})!, \nonumber \\
	 b_m &=	\binom{p}{p-2m-1}\binom{p-2m-1}{\frac{p-2m-1}{2}} (\frac{p-2m-1}{2})!\binom{r}{r-2m-1} \nonumber\\ 
	&\qquad  \binom{r-2m-1} {\frac{r-2m-1}{2}} (\frac{r-2m-1}{2})! \binom{q}{q-2r-1}\binom{q-2r-1}{\frac{q-2r-1}{2}} (\frac{q-2r-1}{2})!. \nonumber
	\end{align*}
	For more details about $h_{2m+1}(s), a_m, b_m$ and calculation of (\ref{eqn:lim_cov_o,k=p,1_sp}), we refer the reader to the proof of Theorem $1$ (Case I) of \cite{SC_independent_2020}.

Now we calculate the limit of $T_2$ as $n\to\infty$. 
%First recall $T_2$ from \eqref{eqn:T_2_sp},
%		\begin{align*} 
%%		\Cov\big(\eta_p,\eta_q\big) \mathbb I_{\{k<p,\ell<q\}} 
%	T_2	& =\frac{1}{n^{\frac{p+q}{2}-1}} \sum_{k, \ell =0}^{p-1, q-1}\binom{p}{k} \binom{q}{\ell} \sum_ {r=0}^{\ell} \binom{\ell}{r}  \sum_{A_{k}, A_{\ell}} \Big\{\E[u_0^{p-k} (u_0+v_0)^{q-\ell} U_{I_k}U_{J_r}V_{J_{r+1,\ell}} ]     \\
%  &\qquad -\E[ u_0^{p-k} U_{I_k}] \E[(u_0+v_0)^{q-\ell}U_{J_r} V_{J_{r+1,\ell}}  ]        \Big\}. 
%\end{align*}			 
	Similar to the term $T_1$, we get maximum contribution in $T_2$ when $\{i_1,i_2,\ldots,i_{k}\}\cap \{j_1,j_2,\ldots,j_{r}\} \cap \{j_{r+1}, j_{r+2},\ldots,j_{\ell}\}=\emptyset$.
	 Since entries are standard Brownian motion,
	\begin{align} \label{eqn:bound_expect}
	& \sum_{A_{k}, A_{\ell}}  \big| \Big\{\E[u_0^{p-k} (u_0+v_0)^{q-\ell} U_{I_k}U_{J_r}V_{J_{r+1,\ell}} ]  -\E[ u_0^{p-k} U_{I_k}] \E[(u_0+v_0)^{q-\ell}U_{J_r} V_{J_{r+1,\ell}}  ]  \Big\} \big|\\
	& \qquad \leq  O(n^{\lfloor\frac{k}{2}\rfloor+ \lfloor\frac{r}{2}\rfloor +\lfloor\frac{\ell-r}{2} \rfloor}). \nonumber
	\end{align}
	 Now using (\ref{eqn:bound_expect}) and the fact that $\E(u_i)= \E(v_i)=0$ for each $i=1,2, \ldots$, we get
	 \begin{align} \label{eqn:lim_cov_o,k<p_sp}
	& \lim_{n\to\infty}  T_2 \nonumber \\ 
	  &= \left\{\begin{array}{ll} 	 \nonumber
		 	\displaystyle\lim_{n\to\infty} \frac{pq}{n^{\frac{p+q}{2}-1}} \sum_ {r=0,2,}^{q-1} \binom{q-1}{r}  \sum_{A_{p-1}, A_{q-1}} \E[ u_0(u_0 +v_0)] \E[U_{I_{p-1}}U_{J_r}V_{J_{r+1,q-1}} ] & \text{if}\ p, \mbox{} q \mbox{ both are odd},\\\\
			0 & \text{otherwise}, 	 	 
		 	  \end{array}\right.	\\ 
		 &= \left\{\begin{array}{ll} 
		 	 \displaystyle \frac{pq}{2^{(\frac{p+q}{2}-1)}} \sum_ {r=0,2,}^{q-1} \binom{q-1}{r}  t_1^{ \frac{p+1+r}{2}} (t_2-t_1)^{ \frac{q-1-r}{2}} 
		 	 d_r  & \text{if}\ p,  q \mbox{ both are odd},\\\\
		 	 0 & \text{otherwise}, 		 
		 	  \end{array}\right. 
	 \end{align}
	 where $$d_r = \binom{p-1}{(p-1)/2}\binom{r}{r/2} \binom{q-r-1}{(q-r-1)/2} \l((p-1)/2\r)! \l(r/2\r)! \l((q-r-1)/2\r)!.$$
	
Now on combining both the cases, from (\ref{eqn:lim_cov_o,k=p_sp}) and (\ref{eqn:lim_cov_o,k<p_sp}), we get
\begin{align} \label{eqn:lim_cov_o_sp}
	\lim_{ \substack{{n\to\infty} \\ {n \mbox{ odd} } }} & \Cov\big(\eta_p(t_1),\eta_q(t_2)\big) \nonumber \\ 
	& =  \left\{\begin{array}{ll} 	 
		 	\displaystyle \sum_ {r=2,4,}^{q} \binom{q}{r}  t_1^{ \frac{p+r}{2}} (t_2-t_1)^{ \frac{q-r}{2}} \Big\{  \frac{2a_1} {2^{\frac{p+q-4}{2}}  } + \\
		 	\qquad \displaystyle + \sum_{m=2}^{ \min\{ \frac{p}{2},\frac{r}{2} \} } \Big( \frac{a_m}{2^{\frac{p+r-4m}{2}}  } \sum_{s=0}^{2m}\binom{2m}{s}^2 s!(2m-s)! \ h_{2m}(s) \Big) \Big\} & \text{if}\ p, \mbox{} q \mbox{ both are even},\\\\
		 	\displaystyle \sum_ {r=1,3,}^{q} \binom{q}{r}  t_1^{ \frac{p+r}{2}} (t_2-t_1)^{ \frac{q-r}{2}}  \sum_{m=0}^{ \min\{ \frac{p-1}{2},\frac{r-1}{2} \} }  \frac{b_m}{2^{\frac{p+r-4m-2}{2}}  } \\
		 \qquad	\displaystyle \sum_{s=0}^{2m+1} \Big( \binom{2m+1}{s}^2 s!(2m+1-s)! \ h_{2m+1}(s) \Big)  \\
		  \displaystyle \qquad +  \frac{pq}{2^{(\frac{p+q}{2}-1)}} \sum_ {r=0,2,}^{q-1} \binom{q-1}{r}  t_1^{ \frac{p+1+r}{2}} (t_2-t_1)^{ \frac{q-1-r}{2}} d_{r}
		 & \text{if}\ p, \mbox{} q \mbox{ both are odd,} \\\\	 	
			0 & \text{otherwise}.	 
		 	  \end{array}\right.			 	  
	\end{align}

\noindent \textbf{Step 2.} For $n$ even, recall $\Cov\big(\eta_p(t_1),\eta_q(t_2)\big)$ from (\ref{eqn:cov_e,o_sp})
 \begin{align}\label{eqn:cov_e_sp}
 & \lim_{n\to\infty}  \Cov\big(\eta_p(t_1),\eta_q(t_2)\big) \\
 &= \lim_{n\to\infty}  \frac{1}{4n^{\frac{p+q}{2}-1}} \sum_{k, \ell =0}^{p, q}\binom{p}{k} \binom{q}{\ell}  \Big\{ \E \Big( [Y_{k} \sum_{A_{k}}  U_{I_k} +\tilde{Y}_{k} \sum_{\tilde{A}_{k}} U_{I_{k}} ]  [Z_{\ell} \sum_{A_{\ell}}  (U+V)_{J_\ell} +\tilde{Z}_{\ell} \sum_{\tilde{A}_{\ell}} (U+V)_{J_{\ell}} ] \Big)   \nonumber \\
  &\qquad -\E[ Y_{k} \sum_{A_{k}}  U_{I_k} +\tilde{Y}_{k} \sum_{\tilde{A}_{k}} U_{I_{k}} ] \E[Z_{\ell} \sum_{A_{\ell}}  (U+V)_{J_\ell} +\tilde{Z}_{\ell} \sum_{\tilde{A}_{\ell}} (U+V)_{J_{\ell}} ]  \Big\}\nonumber, 
  \end{align}
  where $U_{I_{k}},  (U+V)_{J_\ell},  Y_{k}, \tilde{Y}_{k}, Z_{\ell}, \tilde{Z}_{\ell}$ are as in (\ref{eqn:U,Y_sp}) and $A_k, \tilde{A}_k$ are as defined in (\ref{def:A_p_sp}).
%  $U_{I_k} = u_{i_1}\cdots u_{i_k}, \ (U+V)_{J_\ell} = (u_{j_1} + v_{j_1}) \cdots  (u_{j_\ell} + v_{j_\ell}); \ A_k, \tilde{A}_k$ are as defined in (\ref{def:A_p_sp}) and 
%\begin{align*} 
%Y_k & = {(u_0 +  u_{\frac{n}{2}})}^{p-k} + {(u_0 -  u_{\frac{n}{2}})}^{p-k}, \ \ \ \tilde{Y}_k = {(u_0 +  u_{\frac{n}{2}})}^{p-k} - {(u_0 -  u_{\frac{n}{2}})}^{p-k}, \\
%Z_\ell & = \{(u_0+v_0) +  (u_{\frac{n}{2}} +v_{\frac{n}{2}})\}^{q-\ell} + \{(u_0+v_0) - (u_{\frac{n}{2}} +v_{\frac{n}{2}})\}^{q-\ell},  \\
%\tilde{Z}_\ell & = \{(u_0+v_0) +  (u_{\frac{n}{2}} +v_{\frac{n}{2}})\}^{q-\ell} - \{(u_0+v_0) - (u_{\frac{n}{2}} +v_{\frac{n}{2}})\}^{q-\ell}. 
%\end{align*} 

By arguments similar to those given  in Step 1, we can show that right side of (\ref{eqn:cov_e_sp}) has non-zero contribution only when $k=p, \ell=q$ with  $\{i_1,i_2,\ldots,i_{k}\}\cap \{j_1,j_2,\ldots,j_{r}\} \neq\emptyset$ and $k=p-1, \ell=q-1$ with  $\{i_1,i_2,\ldots,i_{k}\}\cap \{j_1,j_2,\ldots,j_{r}\}=\emptyset$. \\
 
 \noindent \textbf{Case I.} \textbf{$k=p$ and $\ell=q$ :} For  $k =p$ and $\ell=q$;
%\begin{align*}
$Y_p  = 2, \  \tilde{Y}_p = 0, \ Z_q  = 2 \mbox{ and} \ \tilde{Z}_q  = 0$.
%\end{align*} 
So in this case, (\ref{eqn:cov_e_sp}) will be
	\begin{align*} 
	\lim_{n\to\infty}  \Cov\big(\eta_p(t_1),\eta_q(t_2)\big)
	= \displaystyle\lim_{n\to\infty} \frac{1}{n^{\frac{p+q}{2}-1}} \sum_ {r=0}^{q} \binom{q}{r} \sum_{m=1}^{ \min \{p,r\} } \sum_{C_m} \Big\{\E[U_{I_p}U_{J_r}V_{J_{r+1,q}} ] -\E[ U_{I_p}] \E[U_{J_r} V_{J_{r+1,q}}  ]\Big\}.	
	\end{align*}
Note that if $r=0$, then that typical term of right side of the above last equation will be zero due to independence of $u_i$'s and $v_i$'s, and hence
\begin{align*} 
	\lim_{n\to\infty}  \Cov\big(\eta_p(t_1),\eta_q(t_2)\big)
	= \displaystyle\lim_{n\to\infty} \frac{1}{n^{\frac{p+q}{2}-1}} \sum_ {r=1}^{q} \binom{q}{r} \sum_{m=1}^{ \min \{p,r\} } \sum_{C_m} \Big\{\E[U_{I_p}U_{J_r}V_{J_{r+1,q}} ] -\E[ U_{I_p}] \E[U_{J_r} V_{J_{r+1,q}}  ]\Big\}.		\nonumber 
	\end{align*}
Note that the  above expression  is same as (\ref{eqn:lim_cov_o,k=p,1_sp}). Therefore $\lim_{n\to\infty}  \Cov\big(\eta_p(t_1),\eta_q(t_2)\big)$ will be equal to the right side of  \eqref{eqn:lim_cov_o,k=p_sp}.
%Therefore,  the above expression equals 
%\begin{align} \label{eqn:lim_cov_e,k=p_sp}
%	&\lim_{ \substack{{n\to\infty} }}  \Cov\big(\eta_p(t_1),\eta_q(t_2)\big)  \\ 
%	&=  \left\{\begin{array}{ll} 	 
%		 	\displaystyle \sum_ {r=2,4,}^{q} \binom{q}{r}  t_1^{ \frac{p+r}{2}} (t_2-t_1)^{ \frac{q-r}{2}} \Big\{  \frac{2a_1} {2^{\frac{p+q-4}{2}}  } +\sum_{m=2}^{ \min\{ \frac{p}{2},\frac{r}{2} \} } \Big ( \frac{a_m}{2^{\frac{p+r-4m}{2}}  } \\
%		 	\qquad \displaystyle \sum_{s=0}^{2m}\binom{2m}{s}^2 s!(2m-s)! \ h_{2m}(s) \Big) \Big\} & \text{if}\ p, \mbox{} q \mbox{ both are even},\\\\
%		 	\displaystyle \sum_ {r=1,3,}^{q} \binom{q}{r}  t_1^{ \frac{p+r}{2}} (t_2-t_1)^{ \frac{q-r}{2}}  \sum_{m=0}^{ \min\{ \frac{p-1}{2},\frac{r-1}{2} \} } \Big( \frac{b_m}{2^{\frac{p+r-4m-2}{2}}  } \\
%		 \qquad	\displaystyle \sum_{s=0}^{2m+1}\binom{2m+1}{s}^2 s!(2m+1-s)! \ h_{2m+1}(s) \Big) & \text{if}\ p, \mbox{} q \mbox{ both are odd} \nonumber \\	 	
%%		\displaystyle + pq \binom{p-1}{(p-1)/2}\binom{q-1}{(q-1)/2} \l((p-1)/2\r)! \l((q-1)/2\r)!\frac{1}{2^{(\frac{p+q}{2}-1)}} 	 & \text{if}\ p, \mbox{} q \mbox{ both are odd}\\\\
%			0 & \text{otherwise}. 	 
%		 	  \end{array}\right.			 	  
%	\end{align} 
\\
	
\noindent \textbf{Case II.} \textbf{$k=p-1$ and $\ell=q-1$ :} 
 For  $k =p-1$ and $\ell=q-1$,
  $$Y_{p-1}= 2u_0, \tilde{Y}_{p-1}=2u_{\frac{n}{2}}, Z_{q-1}= 2(u_0+v_0) \mbox{ and } \tilde{Z}_{q-1}=2(u_{\frac{n}{2}}+v_{\frac{n}{2}}).$$ 
  So in this case,  (\ref{eqn:cov_e_sp}) will be 
  \begin{align}\label{eqn:cov_e,k<p_sp}
  &\lim_{n\to\infty} \Cov\big(\eta_p(t_1),\eta_q(t_2)\big)  \\
   &= \lim_{n\to\infty} \frac{1}{4n^{\frac{p+q}{2}-1}} p q\Big[ \E  \{ 4u_0 (u_0+v_0) \sum_{A_{p-1},A_{q-1}}  U_{I_{p-1}} (U+V)_{J_{q-1}} \}  + \E \{ 4 u_{\frac{n}{2}}  (u_0+v_0) \sum_{ \tilde{A}_{p-1},A_{q-1}}  U_{I_{p-1}} (U+V)_{J_{q-1}} \}  \nonumber \\
& +\E  \{ 4u_0 (u_{\frac{n}{2}}+v_{\frac{n}{2}}) \sum_{A_{p-1}, \tilde{A}_{q-1}}  U_{I_{p-1}} (U+V)_{J_{q-1}} \} + \E \{ 4 u_{\frac{n}{2}}  (u_{\frac{n}{2}}+v_{\frac{n}{2}}) \sum_{ \tilde{A}_{p-1}, \tilde{A}_{q-1}}  U_{I_{p-1}} (U+V)_{J_{q-1}} \} \nonumber \\
&   - \E[(2u_0\sum_{A_{p-1}} U_{I_{p-1}} + 2u_{\frac{n}{2}} \sum_{ \tilde{A}_{p-1}}  U_{I_{p-1}} ]  \E[ 2(u_0+v_0)\sum_{A_{q-1}} (U+V)_{J_{q-1}} + 2(u_{\frac{n}{2}}+v_{\frac{n}{2}}) \sum_{ \tilde{A}_{q-1}}  (U+V)_{J_{q-1}}] \Big]. \nonumber 
\end{align}
Note that, in this case we get non-zero contribution when $\{i_1,i_2,\ldots,i_{p-1}\}\cap \{j_1,j_2,\ldots,j_{q-1}\}=\emptyset$.
% with $(i_1,i_2,\ldots,i_{p-1})$ and $(j_1,j_2,\ldots,j_{q-1})$ are {\it opposite sign pair matched}. 
 Now similar to Case II of Step 1 and from (\ref{eqn:A,tildeA}), right side of (\ref{eqn:cov_e,k<p_sp}) will be
 \begin{align*} %\label{eqn:lim_cov_e,k<p_sp}
	&\lim_{n\to\infty} \Cov\big(\eta_p(t_1),\eta_q(t_2)\big)  \nonumber \\
  &= \lim_{n\to\infty} \frac{1}{4n^{\frac{p+q}{2}-1}} p q \Big[ \E  \big\{ 4u_0 (u_0+v_0) \sum_{A_{p-1},A_{q-1}}  U_{I_{p-1}} (U+V)_{J_{q-1}} \big\} \Big] \nonumber \\
    &= \left\{\begin{array}{ll} 	 \nonumber
		 	\displaystyle\lim_{n\to\infty} \frac{pq}{n^{\frac{p+q}{2}-1}} \sum_ {r=0,2,}^{q-1} \binom{q-1}{r}  \sum_{A_{p-1}, A_{q-1}} \E[ u_0(u_0 +v_0)] \E[U_{I_{p-1}}U_{J_r}V_{J_{r+1,q-1}} ] & \text{if}\ p, \mbox{} q \mbox{ both are odd},\\\\
			0 & \text{otherwise}, 	 	 
		 	  \end{array}\right.	\\ 
		 &= \left\{\begin{array}{ll} 
		 	 \displaystyle \frac{pq}{2^{(\frac{p+q}{2}-1)}}  \sum_ {r=0,2,}^{q-1} \binom{q-1}{r}  t_1^{ \frac{p+1+r}{2}} (t_2-t_1)^{ \frac{q-1-r}{2}} 
		 	 d_r  & \text{if}\ p,  q \mbox{ both are odd},\\\\ 
		 	 0 & \text{otherwise},
		 	  \end{array}\right.	
	 \end{align*} 
	 where the last equality comes from (\ref{eqn:lim_cov_o,k<p_sp}).
	 
Now, by combining Case I and Case II, 
%from (\ref{eqn:lim_cov_e,k=p_sp}) and (\ref{eqn:lim_cov_e,k<p_sp}), 
we get
\begin{align} \label{eqn:lim_cov_e_sp}
	&\lim_{ \substack{{n\to\infty} \\ {n \mbox{ even} } }}  \Cov\big(\eta_p(t_1),\eta_q(t_2)\big)  \\ 
	& =  \left\{\begin{array}{ll} 	 
		 	\displaystyle \sum_ {r=2,4,}^{q} \binom{q}{r}  t_1^{ \frac{p+r}{2}} (t_2-t_1)^{ \frac{q-r}{2}} \Big\{  \frac{2a_1} {2^{\frac{p+q-4}{2}}  } \\ 
		 	\qquad \displaystyle + \sum_{m=2}^{ \min\{ \frac{p}{2},\frac{r}{2} \} } \Big( \frac{a_m}{2^{\frac{p+r-4m}{2}}  } \sum_{s=0}^{2m}\binom{2m}{s}^2 s!(2m-s)! \ h_{2m}(s) \Big) \Big\} & \text{if}\ p, \mbox{} q \mbox{ both are even},\\\\
		 	\displaystyle \sum_ {r=1,3,}^{q} \binom{q}{r}  t_1^{ \frac{p+r}{2}} (t_2-t_1)^{ \frac{q-r}{2}}  \sum_{m=0}^{ \min\{ \frac{p-1}{2},\frac{r-1}{2} \} }  \frac{b_m}{2^{\frac{p+r-4m-2}{2}}  } \\\\
		 \qquad	\displaystyle \sum_{s=0}^{2m+1} \Big( \binom{2m+1}{s}^2 s!(2m+1-s)! \ h_{2m+1}(s) \Big)  \\\\
		  \displaystyle \qquad +  \frac{pq}{2^{(\frac{p+q}{2}-1)}}  \sum_ {r=0,2,}^{q-1} \binom{q-1}{r}  t_1^{ \frac{p+1+r}{2}} (t_2-t_1)^{ \frac{q-1-r}{2}}  d_r
		 & \text{if}\ p, \mbox{} q \mbox{ both are odd,} \\\\	 	
			0 & \text{otherwise}.
		 	  \end{array}\right.	\nonumber	 	  
	\end{align}
	Finally \eqref{eqn:cov_sp} follows from 	(\ref{eqn:lim_cov_o_sp}) and (\ref{eqn:lim_cov_e_sp}).  
%	 Finally combining Steps 1 and 2, from (\ref{eqn:lim_cov_o_sp}) and (\ref{eqn:lim_cov_e_sp}), we get 
%\begin{align*} 
%	\lim_{n\to\infty} & \Cov\big(\eta_p(t_1),\eta_q(t_2)\big)  \\ 
%& =  \left\{\begin{array}{ll} 	 
%		 	\displaystyle \sum_ {r=2,4,}^{q} \binom{q}{r}  t_1^{ \frac{p+r}{2}} (t_2-t_1)^{ \frac{q-r}{2}} \Big\{  \frac{2a_1} {2^{\frac{p+q-4}{2}}  } + \\
%		 	\qquad \displaystyle + \sum_{m=2}^{ \min\{ \frac{p}{2},\frac{r}{2} \} } \Big( \frac{a_m}{2^{\frac{p+r-4m}{2}}  } \sum_{s=0}^{2m}\binom{2m}{s}^2 s!(2m-s)! \ h_{2m}(s) \Big) \Big\} & \text{if}\ p, \mbox{} q \mbox{ both are even},\\\\
%		 	\displaystyle \sum_ {r=1,3,}^{q} \binom{q}{r}  t_1^{ \frac{p+r}{2}} (t_2-t_1)^{ \frac{q-r}{2}}  \sum_{m=0}^{ \min\{ \frac{p-1}{2},\frac{r-1}{2} \} }  \frac{b_m}{2^{\frac{p+r-4m-2}{2}}  } \\
%		 \qquad	\displaystyle \sum_{s=0}^{2m+1} \Big( \binom{2m+1}{s}^2 s!(2m+1-s)! \ h_{2m+1}(s) \Big)  \\
%		  \displaystyle \qquad +  \frac{pq}{2^{(\frac{p+q}{2}-1)}}  \sum_ {r=0,2,}^{q-1} \binom{q-1}{r}  t_1^{ \frac{p+1+r}{2}} (t_2-t_1)^{ \frac{q-1-r}{2}} d_r
%		 & \text{if}\ p, \mbox{} q \mbox{ both are odd,} \\\\	 	
%			0 & \text{otherwise}, 	 
%		 	  \end{array}\right.			 	  
%	\end{align*}	 
This completes the proof of Theorem \ref{thm:symcovar_sp}.
\end{proof}
The following result is a consequence of  Theorem \ref{thm:symcovar_sp}. 
%It will be used in the proof of Theorem \ref{thm:symmulti_sp}.
\begin{corollary}\label{cor:existence of Gaussian process_sp} 
For $p\geq 2$, there exists a centred Gaussian process $\{N_p(t);t\geq 0\}$  such that
%$$\Cov(N_p(t_1),N_p(t_2))= (\min\{t_1,t_2\})^pp!\ \sum_{s=0}^{p-1}	f_p(s),$$
\begin{align} \label{eqn:cov_gaussian_sp}
	&\lim_{ \substack{{n\to\infty} }}  \Cov(N_p(t_1),N_p(t_2))  \\ 
	& =  \left\{\begin{array}{ll} 	 
		 	\displaystyle \sum_ {r=2,4,}^{p} \binom{p}{r}  t_1^{ \frac{p+r}{2}} (t_2-t_1)^{ \frac{p-r}{2}} \Big\{  \frac{2\tilde{a}_1} {2^{p-2}} +\sum_{m=2}^{\frac{r}{2} } \Big( \frac{\tilde{a}_m}{2^{\frac{p+r-4m}{2}}  } \\
		 	\qquad \displaystyle \sum_{s=0}^{2m}\binom{2m}{s}^2 s!(2m-s)! \ h_{2m}(s) \Big) \Big\} & \text{if}\ p \mbox{ is even},\\\\
		 	\displaystyle \sum_ {r=1,3,}^{p} \binom{p}{r}  t_1^{ \frac{p+r}{2}} (t_2-t_1)^{ \frac{p-r}{2}}  \sum_{m=0}^{\frac{r-1}{2} }  \frac{\tilde{b}_m}{2^{\frac{p+r-4m-2}{2}}  }  \sum_{s=0}^{2m+1} \Big( \binom{2m+1}{s}^2  \\
		 \qquad	\displaystyle s!(2m+1-s)! \ h_{2m+1}(s) \Big) \\
		 \qquad	\displaystyle + \frac{p^2}{2^{p-1}} \sum_ {r=0,2,}^{p-1} \binom{p-1}{r}  t_1^{ \frac{p+1+r}{2}} (t_2-t_1)^{ \frac{p-1-r}{2}}	 \tilde{d}_r
		 & \text{if}\ p \mbox{ is odd,} \\\\	 	
			0 & \text{otherwise}, 	 
		 	  \end{array}\right.	\nonumber		 	  
	\end{align}	 
	where $h_m(s)$ is as in Result \ref{result:def_h} and
	\begin{align*} 
	 \tilde{a}_m &=	\binom{p}{p-2m}\binom{p-2m}{\frac{p-2m}{2}} (\frac{p-2m}{2})! \binom{r}{r-2m}\binom{r-2m}{\frac{r-2m}{2}} (\frac{r-2m}{2})! \\
	  &\qquad  \binom{p}{p-2r}\binom{p-2r}{\frac{p-2r}{2}} (\frac{p-2r}{2})!, \nonumber \\
	 \tilde{b}_m &=	\binom{p}{p-2m-1}\binom{p-2m-1}{\frac{p-2m-1}{2}} (\frac{p-2m-1}{2})!\binom{r}{r-2m-1} \nonumber\\ 
	&\qquad  \binom{r-2m-1} {\frac{r-2m-1}{2}} (\frac{r-2m-1}{2})! \binom{p}{p-2r-1}\binom{p-2r-1}{\frac{p-2r-1}{2}} (\frac{p-2r-1}{2})!, \nonumber \\
	\tilde{d}_{r} &=  \binom{p-1}{(p-1)/2}\binom{r}{r/2} \binom{p-r-1}{(p-r-1)/2} \l((p-1)/2\r)! \l(r/2\r)! \l((p-r-1)/2\r)!.
	\end{align*}
%	\begin{align}
%	d_{p} &= \frac{p^2}{2^{p-1}} \Big\{ \binom{p-1}{(p-1)/2} \l((p-1)/2\r)! \Big\}^2.
%	\end{align}
\end{corollary}
\begin{proof}

The existence of such a Gaussian process is due to Theorem \ref{thm:symcovar_sp} and Kolmogrov existence theorem.
\end{proof}

The following remark gives us a generalization (in the sense of input entries) of Theorem \ref{thm:revmulti_rp} and Theorem \ref{thm:symmulti_sp}. 
\begin{remark}
	  In the proof of Theorem \ref{thm:revmulti_rp} and \ref{thm:symmulti_sp}, we have used the following properties of the standard Brownian motion; $\{b_n(t); t \geq 0\}_{n\geq 0}$ is an  independent sequence of stochastic process, independent increment property of Brownian motion, $\E[b_n(t)]=0$ for all $t$ and $\sup_{1\leq j\leq \ell}\E|b_i(t_j)|^{2m} < \infty$, for all $\ell= 1, 2, \ldots$. Therefore our proof of Theorem \ref{thm:revmulti_rp} and \ref{thm:symmulti_sp} is also true for any other independent stochastic processes $\{X_n(t) ; t\geq 0\}_{n\geq 0}$, which have the following properties: \\
	For fixed $n$
	\begin{enumerate}
		\item [(i)] $\E[X_n(t)]=0$ for all $t$.
		\item [(ii)] $\{X_n(t) ; t\geq 0\}$ has independent increment property.
		\item [(iii)] $\E(|X_n(t)|^2)=\alpha_2<\infty \ \mbox{for all}\ n\geq 0$ and 
		 $\sup_{n \geq 0}\E(|X_n(t)|^k)<\infty \ \mbox{for}\ k\geq 3.$
	\end{enumerate}
\end{remark}

\section{Proof of Theorem \ref{thm:symmulti_sp}}\label{sec:joint convergence_sp}
We first recall the notion of connected vectors and clusters from Definitions \ref{def:connected_rev}
and \ref{def:cluster_rev} in Section \ref{sec:joint convergence_sp}. Also recall $A_{p}$ and $\tilde{A}_p$ from \eqref{def:A_p_sp} in Section \ref{sec:joint convergence_sp},
\begin{align*}
A_{p} & =\Big\{(j_1,\ldots,j_{p})\suchthat \sum_{i=1}^{p}\epsilon_i j_i=0\; \mbox{(mod n)}, \epsilon_i\in\{+1,-1\}, 1\le j_1,\ldots,j_{p}< \frac{n}{2}\Big\}, \\
\tilde{A}_{p} &=\Big\{(j_1,\ldots,j_{p})\suchthat \sum_{i=1}^{p}\epsilon_i j_i=0\; \mbox{(mod } \frac{n}{2}) \mbox{ and } \sum_{i=1}^{p}\epsilon_i j_i \neq 0\; \mbox{(mod }n),  \epsilon_i\in\{+1,-1\}, 1\le j_1,\ldots,j_{p}< \frac{n}{2}\Big\}.
\end{align*}  
We define a subset $B_{P_\ell}$ of the Cartesian product  $ A_{p_1} \times A_{p_2} \times \cdots \times A_{p_\ell}$, where $A_{p_i}$ is as  in \eqref{def:A_p_sp} for all $i=1,2,\ldots,\ell$.  
\begin{definition}\label{def:B_{P_l}_sp}
	Let $\ell \geq 2$ and  $P_\ell = (p_1,p_2, \ldots, p_\ell ) $. $B_{P_\ell}$ is a subset of $ A_{p_1} \times A_{p_2} \times \cdots \times A_{p_\ell}$ such that  $ (J_1, J_2, \ldots, J_\ell) \in B_{P_\ell} $ if 
	\begin{enumerate} 
		\item[(i)] $\{J_1, J_2, \ldots, J_\ell\} $ form a cluster, 
		\item[(ii)] each element in  $\displaystyle{\cup_{i=1}^{\ell} S_{J_i} }$ has  multiplicity greater than or equal to two. 
	\end{enumerate}
\end{definition} 
 The next Result gives us the cardinality of $B_{P_\ell}$.
\begin{result} \label{res:|B_{P_l}|_sp} (Lemma 13 \cite{SC_independent_2020})
	For  $\ell \geq 3 $,  
	\begin{equation*} 
	|B_{P_\ell }| = o \big(n^{\frac{p_1+p_2 + \cdots + p_\ell -\ell}{2} }\big).
	\end{equation*}
	%where $B_{P_\ell}$ is as defined in Definition $\ref{def:B_{P_l}}$.
\end{result}

The following lemma is an easy consequence of Result \ref{res:|B_{P_l}|_sp}.
\begin{lemma}\label{lem:maincluster_sp}
Suppose $\{J_1, J_2, \ldots, J_\ell \} $ form a cluster where $J_i\in A_{p_i}$ with $p_i\geq 2$ for $1\leq i\leq \ell$. Then for $\ell \geq 3,$
\begin{equation}\label{equation:maincluster_sp}
	\frac{1}{ n^{\frac{p_1+p_2+ \cdots + p_\ell - \ell}{2}} } \sum_{A_{p_1}, A_{p_2}, \ldots, A_{p_\ell}} \E\Big[\prod_{k=1}^{\ell}\Big(b_{J_k}(t_k) - \E(b_{J_k}(t_k))\Big)\Big] = o(1),
\end{equation}
where $0<t_1 \leq t_2 \leq \cdots \leq t_\ell $ and  for $k \in \{1,2, \ldots, \ell \} $, 
$$J_k = (j_{k,1}, j_{k,2}, \ldots, j_{k,p_k} ) \ \mbox{and} \ b_{J_k}(t_k) = b_{j_{k,1}}(t_k) b_{j_{k,2}}(t_k) \cdots b_{j_{k,p_k}}(t_k).$$	
\end{lemma}
\begin{proof} The proof of this lemma is similar to the proof of Lemma \ref{lem:maincluster_rp}. So, we skip it.
\end{proof}

%\begin{remark} \label{rem:main_cluster}
% Suppose $d_i < p_i$, for some $i=1,2, \ldots, \ell$, then from Lemma \ref{lem:maincluster}, we observe that
%%Suppose $\{J'_1, J'_2, \ldots, J'_\ell \} $ form a cluster where $J'_i\in A_{d_i}$ with $d_i\geq 2$ for $1\leq i\leq \ell$, $\{X_i\}_{i \geq 1}$ is independent and it satisfies (\ref{eqn:condition_rev}). Then for $\ell \geq 3,$
%	\begin{equation*}\label{eqn:maincluster_remark}
%	\frac{1}{ n^{\frac{p_1+p_2+ \cdots + p_\ell - \ell}{2}}} \sum_{A_{d_1}, \ldots, A_{d_\ell}} \E\Big[\prod_{k=1}^{\ell}\Big(X_0^{p_k-d_k}X_{J'_k} - \E(X_0^{p_k-d_k} X_{J'_k})\Big)\Big] = o(1),
%	\end{equation*}
%	where 
%	$$J'_k = (j^{k}_{1}, j^{k}_{2}, \ldots, j^{k}_{d_k} ) \ \mbox{and} \ X_{J'_k} = X_{j^{k}_{1}} X_{j^{k}_{2}} \cdots X_{j^{k}_{d_k}}.$$	
%\end{remark}

\begin{lemma}\label{lem:cluster,decompose_A_sp}
	Suppose $J_i\in A_{d_i}$ and $d_i\geq 2$ for $1\leq i\leq \ell$. Then for $\ell \geq 3$,
	\begin{align*}
	\sum_{A_{d_1}, A_{d_2}, \ldots, A_{d_\ell}} \E\Big[\prod_{k=1}^{\ell}\Big(b_{J_k}(t_k) - \E(b_{J_k}(t_k))\Big)\Big]
	 &= \left\{\begin{array}{ll} 
		 	O( n^{\frac{d_1+d_2+ \cdots + d_\ell - \ell}{2}}) & \text{if} \ \{ J_1, J_2, \ldots, J_\ell\} \mbox{ decomposes } \\ 
		 	& \mbox{ into clusters of length } 2  \\\\
		 	 o( n^{\frac{d_1+d_2+ \cdots + d_\ell - \ell}{2}}) & \text{otherwise}, 	 
		 	  \end{array}\right.
	\end{align*}
where $0<t_1 \leq t_2 \leq \cdots \leq t_\ell $ and  for $k \in \{1,2, \ldots, \ell \} $, 
$$J_k = (j_{k,1}, j_{k,2}, \ldots, j_{k,d_k} ) \ \mbox{and} \ b_{J_k}(t_k) = b_{j_{k,1}}(t_k) b_{j_{k,2}}(t_k) \cdots b_{j_{k,d_k}}(t_k).$$	
\end{lemma}
The  proof of this lemma is similar to the proof of Lemma $12$ of \cite{SC_independent_2020}, where only the input entries are different but the idea is same. So, we skip it.

\begin{remark} \label{rem:cluster,decompose_F_sp}
 	Suppose $J_i\in F_{d_i}$ and $d_i\geq 2$ for $1\leq i\leq \ell$, where $F_{d_i}$ is $A_{d_i}$ or $\tilde{A}_{d_i}$. Then from (\ref{eqn:A,tildeA}) and Lemma \ref{lem:cluster,decompose_A_sp}, we get
	\begin{align*}
	\sum_{F_{d_1}, \ldots, F_{d_\ell}} \E\Big[\prod_{k=1}^{\ell}\Big(b_{J_k}(t_k) - \E(b_{J_k}(t_k))\Big)\Big]
	 &= \left\{\begin{array}{ll} 
		 	O( n^{\frac{d_1+d_2+ \cdots + d_\ell - \ell}{2}}) & \text{if} \ \{ J_1, J_2, \ldots, J_\ell\} \mbox{ decomposes into clusters } \\ 
		 	& \mbox{ of length }2 \mbox{ and } F_{d_i} =A_{d_i} \forall \ i=1, 2, \ldots, \ell  \\\\
		 	 o( n^{\frac{d_1+d_2+ \cdots + d_\ell - \ell}{2}}) & \text{otherwise.}	 	 
		 	  \end{array}\right.
	\end{align*}
\end{remark}
%We shall use the above lemmata, Remarks and Theorem \ref{thm:symcovar_sp} to prove Theorem \ref{thm:symmulti_sp}.  

\begin{proof}[Proof of Theorem \ref{thm:symmulti_sp}] 
The proof of Theorem \ref{thm:symmulti_sp}  is similar to the proof of Theorem \ref{thm:revmulti_rp}.
%$2$ of \cite{SC_independent_2020}, where the authors have proved a similar result for symmetric circulant matrices. The only difference is that  the entries, but standard Brownian motion entries have all those properties that entries of \cite{SC_independent_2020} have. 
So, we point out only the main steps of the proof and skip the details.

%We use method of moments and  Wick's formula to prove Theorem \ref{thm:symmulti_sp}. 
%Recall that from the method of moments, to prove $w_Q \stackrel{d}{\longrightarrow} N(0,\sigma_{Q}^2)$, it is sufficient to show that
%	\begin{align*} 
%	\lim_{n\to\infty} \E[ (w_Q)^\ell] = \E[ (N(0, \sigma^2_{Q}))^\ell ] \ \ \forall \  \ell=1,2, \ldots.
%	\end{align*}
		So, it is enough to show that, for $p_1, p_2, \ldots , p_\ell \geq 2$,
	\begin{align} \label{eqn:moment,product_w_p}
	\lim_{n\tends \infty}  \E[\eta_{p_1} (t_1)\eta_{p_2}(t_2) \cdots \eta_{p_\ell}(t_\ell)]=\E[N_{p_1}(t_1)N_{p_2}(t_2) \cdots N_{p_\ell}(t_\ell)].
	\end{align}
%	where $\{N_{p}\}_{p \geq 1}$ is a centered Gaussian family with covariance $\sigma_{p,q}$, that is, $\E[N_{p},N_{q}]= \sigma_{p,p}$, where $\sigma_{p,q}$ as in (\ref{eqn:sigma_p,q}). 
%Since we are using trace formula to prove the theorem and for odd and even value of $n$, we have different trace formula. Therefore, we show 
As before, we show (\ref{eqn:moment,product_w_p}) for odd and even value of $n$ separately. 

First suppose $n$ is odd. Then by using trace formula (\ref{trace,formula_sp}), we get 
	\begin{align}\label{eqn:expectation_thm2_sp} 
	\lim_{n\tends \infty}  \E[\eta_{p_1} (t_1)\eta_{p_2}(t_2) \cdots \eta_{p_\ell}(t_\ell)]
	&=\lim_{n\tends \infty}  \frac{1}{n^{\frac{p_1 + p_2 + \cdots +p_\ell -\ell}{2}}} \sum_{d_1=0}^{p_1} \cdots \sum_{d_\ell=0}^{p_\ell} \binom{p_1}{d_1} \cdots \binom{p_\ell}{d_\ell} \sum_{A_{d_1}, \ldots, A_{d_\ell}} \nonumber\\
	&\qquad \E \Big[ \prod_{k=1}^{\ell}  \Big((b_0(t_k))^{p_k-d_k} b_{J_{d_k}}(t_k) - \E [(b_0(t_k))^{p_k-d_k} b_{J_{d_k}}(t_k)] \Big) \Big] \nonumber \\ 
	& = \lim_{n\tends \infty}  \frac{1}{n^{\frac{p_1 + p_2 + \cdots +p_\ell -\ell}{2}}} \sum_{A_{p_1}, \ldots, A_{p_\ell}} \E\Big[\prod_{k=1}^{\ell}\Big( b_{J_{d_k}}(t_k) - \E [ b_{J_{d_k}}(t_k)] \Big)\Big],
	\end{align}
	where the last equality comes due to Lemma \ref{lem:cluster,decompose_A_sp}. Because %when $d_i <p_i$ for some $i=1, 2, \ldots,\ell$, then
	$$\sum_{A_{d_1}, \ldots, A_{d_\ell}} \E \Big[ \prod_{k=1}^{\ell}  \Big((b_0(t_k))^{p_k-d_k} b_{J_{d_k}}(t_k) - \E [(b_0(t_k))^{p_k-d_k} b_{J_{d_k}}(t_k)] \Big) \Big]  \leq O(n^{\frac{d_1 + d_2 + \cdots +d_\ell -\ell}{2}}).$$
Now on combining Lemma \ref{lem:cluster,decompose_A_sp} for $d_i=p_i$ and \eqref{eqn:expectation_thm2_sp}, we get
	\begin{align*} \label{eq:multisplit_sp}
	& \quad \lim_{n\tends \infty}  \E[\eta_{p_1}(t_1) \eta_{p_2}(t_2)\cdots \eta_{p_\ell}(t_\ell)]\\
	& \quad  =\lim_{n\to\infty} \frac{1}{n^{\frac{p_1 + p_2 + \cdots +p_\ell -\ell}{2}}} \sum_{A_{p_1}, \ldots, A_{p_\ell}} \E\Big[\prod_{k=1}^{\ell}\Big( b_{J_{p_k}}(t_k) - \E [ b_{J_{p_k}}(t_k)] \Big)\Big]\\
	& \quad  =\lim_{n\to\infty} \frac{1}{n^{\frac{p_1 + p_2 + \cdots +p_\ell -\ell}{2}}} \sum_{\pi \in \mathcal P_2(\ell)} \prod_{i=1}^{\frac{\ell}{2}}  \sum_{A_{p_{y(i)}},\ A_{p_{z(i)}}} \E\big[ \{ b_{J_{y(i)}} (t_{y(i)}) - \E b_{J_{y(i)}}(t_{y(i)})\} \{b_{J_{z(i)}}(t_{z(i)}) - \E b_{J_{z(i)}} (t_{z(i)}) \}\big],
	\end{align*}
	where  $\pi = \big\{ \{y(1), z(1) \}, \ldots , \{y(\frac{\ell}{2}), z(\frac{\ell}{2})  \} \big\}\in \mathcal P_2(\ell)$ and $\mathcal P_2(\ell)$ is the set of all pair partition of $ \{1, 2, \ldots, \ell\} $. Using Theorem \ref{thm:symcovar_sp}, from the last equation, we get
	\begin{equation}\label{eqn:product of expectation_sp}
	\lim_{n\tends \infty}  \E[\eta_{p_1}(t_1)\eta_{p_2}(t_2) \cdots \eta_{p_\ell}(t_\ell)]
	=\sum_{\pi \in P_2(\ell)} \prod_{i=1}^{\frac{\ell}{2}} \lim_{n\tends \infty} \E[\eta_{p_{y(i)}} (t_{y(i)})  \eta_{p_{z(i)}} (t_{z(i)})].
	\end{equation}
	%&= \lim_{n\tends \infty} \sum_{\pi \in P_2(\ell)} \prod_{i=1}^{\frac{\ell}{2}} \E[w_{p_{y(i)}} (t_{y(i)}) w_{p_{z(i)}} (t_{z(i)})] \nonumber \\
%    Since from Theorem \ref{thm:symcovar}, we have 
%	$$\lim_{n\to\infty}\E(w_p(t_1) w_q(t_2)) = \sigma_{p,q} (= \E(N_p(t_1) N_q(t_2))).$$
	Therefore using Wick's formula and Theorem \ref{thm:symcovar_sp},  from \eqref{eqn:product of expectation_sp} we get 
	\begin{align*}
	\lim_{ \substack{{n\to\infty}  \\ {n \mbox{ odd} } }} \E[\eta_{p_1}(t_1)\eta_{p_2}(t_2) \cdots \eta_{p_\ell}(t_\ell)]
%	&=\sum_{\pi \in P_2(\ell)} \prod_{i=1}^{\frac{\ell}{2}} \lim_{n\tends \infty} \E[w_{p_{y(i)}} (t_{y(i)}) w_{p_{z(i)}} (_{z(i)})] \\
%	& =\sum_{\pi \in \mathcal P_2(\ell)} \prod_{i=1}^{\frac{\ell}{2}} \E[N_{p_{y(i)}}(t_{y(i)}) N_{p_{z(i)}} (t_{z(i)})] \\
	&=\E[ N_{p_1}(t_1)N_{p_2}(t_2) \cdots N_{p_\ell}(t_\ell) ].
	\end{align*} 
Now suppose $n$ is even. Then by using trace formula (\ref{trace,formula_sp}), we get
  \begin{align*}
	\eta_{p_k}(t_k) & = \frac{1}{n^{\frac{p_k -1}{2}}} \sum_{d_k=0}^{p_k}\binom{p_k}{d_k}  \Big[ y_{d_k}  \sum_{A_{d_k}} b_{J_{d_k}}(t_k) + \tilde{y}_{d_k} \sum_{ \tilde{A}_{d_k}}  b_{J_{d_k}}(t_k)  - \E [y_{d_k}  \sum_{A_{d_k}} b_{J_{d_k}}(t_k) + \tilde{y}_{d_k} \sum_{ \tilde{A}_{d_k}}  b_{J_{d_k}}(t_k) ] \Big] \\
	 & = \frac{1}{n^{\frac{p_k -1}{2}}} \sum_{d_k=0}^{p_k}\binom{p_k}{d_k}  \Big[ \sum_{A_{d_k}} y_{d_k} b_{J_{d_k}}(t_k) - \E[y_{d_k} b_{J_{d_k}}(t_k) ] +  \sum_{ \tilde{A}_{d_k}} \tilde{y}_{d_k} b_{J_{d_k}}(t_k)  - \E [\tilde{y}_{d_k} b_{J_{d_k}}(t_k)] \Big],\\
	\end{align*}
  where 
\begin{align*}
y_{d_k} &=  {[b_0(t_k) +  b_{\frac{n}{2}} (t_k)]}^{p_k-d_k} + {[b_0(t_k) -  b_{\frac{n}{2}} (t_k)]}^{p_k-d_k},\\
 \tilde{y}_{d_k} &= {[b_0(t_k) +  b_{\frac{n}{2}} (t_k)]}^{p_k-d_k} - {[b_0(t_k) -  b_{\frac{n}{2}} (t_k)]}^{p_k-d_k}.
\end{align*}  
  Now from the above expression of $\eta_{p_k}(t_k)$, we have
	\begin{align}\label{eqn:expectation_thm2,even_sp} 
	 \lim_{n\tends \infty}   \E[\eta_{p_1}(t_1)\eta_{p_2}(t_2) \cdots \eta_{p_\ell}(t_\ell)] 
	&= \lim_{n\tends \infty}  \frac{1}{n^{\frac{p_1 + p_2 + \cdots +p_\ell -\ell}{2}}} \sum_{d_1=0}^{p_1} \cdots \sum_{d_\ell=0}^{p_\ell} \binom{p_1}{d_1} \cdots \binom{p_\ell}{d_\ell} \E\Big[ \prod_{k=1}^{\ell} \Big(\nonumber\\ 
	& \qquad  \sum_{A_{d_k}} y_{d_k} b_{J_{d_k}}(t_k) - \E[ y_{d_k} b_{J_{d_k}}(t_k) ]  +  \sum_{ \tilde{A}_{d_k}} \tilde{y}_{d_k} b_{J_{d_k}}(t_k)  - \E [\tilde{y}_{d_k} b_{J_{d_k}} (t_k)] \Big) \Big]\nonumber \\
	& = \lim_{n\tends \infty}  \frac{1}{n^{\frac{p_1 + p_2 + \cdots +p_\ell -\ell}{2}}} \sum_{A_{p_1}, \ldots, A_{p_\ell}} \E\Big[\prod_{k=1}^{\ell}\Big(b_{J_k}(t_k) - \E(b_{J_k}) (t_k)\Big)\Big],
	\end{align}
	where the last equality follows from Lemma \ref{lem:cluster,decompose_A_sp} and Remark \ref{rem:cluster,decompose_F_sp}. Since (\ref{eqn:expectation_thm2,even_sp}) is same as (\ref{eqn:expectation_thm2_sp}), therefore by the the similar calculation as we have done for $n$ odd, we get
%	\begin{align*}
%	\lim_{ \substack{{n\to\infty}  \\ {n \mbox{ odd} } }} \E[w_{p_1}w_{p_2} \cdots w_{p_\ell}]
%	&=\E[ N_{p_1}N_{p_2} \cdots N_{p_\ell} ].
%	\end{align*} 
	$$\lim_{ \substack{{n\to\infty}  \\ {n \mbox{ even} } }} \E[\eta_{p_1}(t_1)\eta_{p_2}(t_2) \cdots \eta_{p_\ell}(t_\ell)]
	=\E[ N_{p_1}(t_1)N_{p_2}(t_2) \cdots N_{p_\ell}(t_\ell) ].$$
This completes the proof of Theorem \ref{thm:symmulti_sp}.
\end{proof}

\section{Proof of Theorem \ref{thm:symprocess}}\label{sec:process convergence_sp} 
%In this section, we prove Theorem \ref{thm:symprocess}, that is, for fixed $p \geq 2$, $ \{ \eta_{p}(t) ; t \geq 0\} \stackrel{\mathcal D}{\rightarrow} \{N_{p}(t) ;t \geq 0\} $, where the existence of $\{N_{p}(t) ; t \geq 0\}$ is guaranteed by Result \ref{cor:existence of Gaussian process_sp}.

The idea of the proof 
%of Theorem \ref{thm:symprocess} 
is similar to the proof of Theorem \ref{thm:revprocess}. So it is enough to prove Proposition \ref{pro:finite convergence_rp} and Proposition \ref{pro:tight_rp} for the process $\{\eta_p(t);t\ge 0\}$.

\begin{proof}[Proof of Theorem \ref{thm:symprocess}] The finite dimensional joint fluctuation of $\{\eta_p(t);t \geq 0\}$ follows from Theorem \ref{thm:symmulti_sp} once we take $p_i=p$. So, to finish the proof of Theorem \ref{thm:symprocess}, it is sufficient to prove Proposition \ref{pro:tight_rp} for the process $\{\eta_p(t);t \geq 0\}$. 
For that, we first note $\eta_p(0)=0$ and we shall show \eqref{eq:tight4_rp} for $\{\eta_p(t);t\ge 0\}$. In particular, we show 
\begin{equation*}
\E{|\eta_{p}(t)- \eta_{p}(s)|^ 4 } \leq M_T |t-s|^2 \ \ \ \forall \ n \in \mathbb{N} \ \mbox{and } t,s \in[0,\ T],
\end{equation*}
where $M_T$ is a positive constant which depends only on $p$ and $T$.

First note from the definition of $\eta_{p}(t)$ in (\ref{eqn:w_p(t)_sp}) 
\begin{align} \label{eqn:w_p(t-s)_sp}
 \eta_{p}(t) - \eta_{p}(s) &= \frac{1}{\sqrt{n}} \big[ \Tr(SC_n(t))^{p} - \Tr(SC_n(s))^{p} - \E[\Tr(SC_n(t))^{p} - \Tr(SC_n(s))^{p}]\big].
\end{align}
 Using binomial expansion for  $0< s <t $, we get
\begin{equation} \label{eqn:trace t-s_sp}
 \Tr (SC_n(t))^{p} - \Tr (SC_n(s))^{p}   
 = \Tr[ (SC'_n(t-s) )^{p}] + \sum_{d=1}^{p-1} \binom{p}{d} \Tr [(SC'_n(t-s))^d (SC_n(s))^{p-d}],
\end{equation}
where $SC'_n(t-s)$ is a symmetric circulant matrix with entries $\{ b_n(t) -b_n(s)\}_{n \geq 0}$. 

%Recall the trace formula from (\ref{trace formula RC_n}) for $(RC'_n(t-s) )^{2p}$ and $(RC'_n(t-s))^d (RC_n(s))^{2p-d}$
Now for odd $n$, from the trace formula (\ref{trace,formula_sp}) and (\ref{trace,product_sp}), we get
\begin{align} \label{eqn:tracemultiply_odd_sp}
	\Tr[ (SC'_n(t-s) )^{p}] & =\frac{1}{n^{\frac{p}{2}-1}} \sum_{k=0}^p \binom{p}{k} (b'_{0}(t-s))^{p-k} \sum_{A_{k}} b'_{I_{k}}(t-s),  \\
	\Tr [(SC'_n(t-s))^d (SC_n(s))^{p-d}] & = \frac{1}{n^{\frac{p}{2}-1}} \sum_{m=0}^d \sum_{r=0}^{p-d} \binom{d}{m} \binom{p-d}{r} (b'_{0}(t-s))^{d-m} (b_{0}(s))^{(p-d)-r} \nonumber \\
	 & \qquad \sum_{A_{m+r}} b'_{J_{m}}(t-s) b_{J_{r}}(s), \nonumber
\end{align}
where 
\begin{align*}
b'_{J_{k}}(t-s) & = (b_{j_1}(t)-b_{j_1}(s)) (b_{j_2}(t) - b_{j_2}(s)) \cdots (b_{j_{k}}(t) - b_{j_{k}}(s)), \\
b'_{J_{m}}(t-s) & = (b_{j_1}(t)-b_{j_1}(s)) (b_{j_2}(t) - b_{j_2}(s)) \cdots (b_{j_{m}}(t) - b_{j_{m}}(s)), \\
b_{J_{r}}(s) & = b_{j_{m+1}}(s) b_{j_{m+2}}(s)\cdots b_{j_{m+r}}(s).
\end{align*}
Note that, $b_n(t) -b_n(s)$ has same distribution as $b_n(t-s)$ and  $b_n(t)$ has same distribution as $N(0, \sqrt{t})$. Therefore
\begin{align}  \label{eqn:same distrib_odd}
b'_{J_{k}}(t-s) & \distas{D} x_{j_1} x_{j_2} \cdots x_{j_{k}} \distas{D} (t-s)^{\frac{k}{2}} \frac{x_{j_1}}{\sqrt{(t-s)}} \frac{x_{j_2}}{\sqrt{(t-s)}} \cdots  \frac{x_{j_{k}}} {\sqrt{(t-s)}} = (t-s)^{\frac{k}{2}} X_{J_{k}}, \mbox{ say, } \\
b'_{I_{m}}(t-s)  & \distas{D} (t-s)^{\frac{m}{2}} \frac{x_{j_1}}{\sqrt{(t-s)}} \frac{x_{j_2}}{\sqrt{(t-s)}} \cdots  \frac{x_{j_{m}}} {\sqrt{(t-s)}} = (t-s)^{\frac{m}{2}} X_{J_{m}}, \mbox{ say,} \nonumber \\
b_{J_{r}}(s) & \distas{D} y_{j_{m+1}} y_{j_{m+2}} \cdots y_{j_{m+r}} \distas{D} {s}^{\frac{r}{2}} \frac{y_{j_{m+1}}}{\sqrt{s}} \frac{y_{j_{m+2}}}{\sqrt{s}} \cdots  \frac{y_{j_{m+r}}} {\sqrt{s}} = {s}^{\frac{r}{2}} Y_{J_r}, \mbox{ say, } \nonumber
\end{align}
where $x_{i_r}$'s and $y_{i_r}$'s are independent normal random variable with mean zero, variance $(t-s)$ and $s$, respectively.  
% By $x_1 \distas{D} x_2$, we mean  $x_1$ and $x_2$ have same distribution. 

Now, by using (\ref{eqn:tracemultiply_odd_sp}) and (\ref{eqn:same distrib_odd}) in (\ref{eqn:trace t-s_sp}), we get
\begin{align} \label{eqn:Tr_t-s_odd,1}
 \Tr (SC_n(t))^{p}- \Tr (SC_n(s))^{p}  & \distas{D}  \frac{1}{n^{\frac{p}{2}-1}} \Biggl\{ \sum_{k=0}^p  (t-s)^{\frac{p-k}{2} + \frac{k}{2}} \binom{p}{k} x_{0}^{p-k} \sum_{A_{k}} X_{J_{k}} + \sum_{d=1}^{p-1} \binom{p}{d}  \Big[ \sum_{m=0}^d \sum_{r=0}^{p-d} \binom{d}{m} \binom{p-d}{r} \nonumber \\
 & \qquad (t-s)^{\frac{d-m}{2} + \frac{m}{2}} {s}^{\frac{p-d-r}{2} + \frac{r}{2}}  x_0^{d-m} y_0^{p-d-r} \sum_{A_{m+r}} X_{J_{m}}  Y_{J_{r}} \Big] \Biggr\} \nonumber \\
 & = \frac{\sqrt{t-s}}{n^{\frac{p}{2}-1} } \Biggl\{ \sum_{k=0}^p  (t-s)^{\frac{p-1}{2}} \binom{p}{k} x_{0}^{p-k} \sum_{A_{k}} X_{J_{k}} + \sum_{d=1}^{p-1} \binom{p}{d}  \Big[ \sum_{m=0}^d \sum_{r=0}^{p-d} \binom{d}{m} \binom{p-d}{r}  \\
 & \qquad (t-s)^{\frac{d-1}{2}} {s}^{\frac{p-d}{2}}  x_0^{d-m} y_0^{p-d-r} \sum_{A_{m+r}} X_{J_{m}}  Y_{J_{r}} \Big] \Biggr\}. \nonumber 
% & = \frac{\sqrt{t-s}}{n^{p-1}}\sum_{A_{2p}} Z_{I_{2p}}, \mbox{ say}.
\end{align} 
Since $x_{i_r}$'s and $y_{i_r}$'s are independent normal random variable and $A_k$ has a constraint, (\ref{def:A_p_sp}), therefore
\begin{align} \label{eqn:E,o(1)_odd}
\E \Big( \frac{1}{n^{\frac{p}{2}-1}} x_{0}^{p-k} \sum_{A_{k}} X_{J_{k}} \Big)
&= \left\{\begin{array}{ll} 
		 	O(1) & \text{if} \   k=p  \\\\
		 	 o(1) & \text{otherwise}, 	 
		 	  \end{array}\right. \\\
 \E \Big( \frac{1}{n^{\frac{p}{2}-1}} x_0^{d-m} y_0^{p-d-r} \sum_{A_{m+r}} X_{J_{m}}  Y_{J_{r}}  \Big)
&= \left\{\begin{array}{ll} 
		 	O(1) & \text{if} \   m+r=p  \\\\
		 	 o(1) & \text{otherwise}. 
		 	  \end{array}\right. \nonumber
\end{align}
Hence on using (\ref{eqn:E,o(1)_odd}) in (\ref{eqn:Tr_t-s_odd,1}), we get
\begin{align} \label{eqn:W_p,odd_sp}
 [\Tr (SC_n(t))^{p} - \Tr (SC_n(s))^{p}] \mathbb I_{\{ n \ odd\}} & \distas{D}  
  \sqrt{t-s} \Big[  \frac{ (t-s)^{\frac{p-1}{2}}} {n^{\frac{p}{2}-1 }}  \sum_{A_{p}} X_{J_{p}} + \frac{1}{n^{\frac{p}{2}-1}} \sum_{d=1}^{p-1} \binom{p}{d} \\
 & \qquad (t-s)^{\frac{d-1}{2}} {s}^{\frac{p-d}{2}} \sum_{A_{p}} X_{J_{d}}  Y_{J_{p-d}} + \theta_p \Big], \nonumber 
% & = \sqrt{t-s} \Big[ \frac{1}{n^{\frac{p}{2}-1 }} \sum_{A_{p}} W_{I_{p}} +\theta_p\Big], \mbox{ say},
\end{align} 
where $\E[ \theta_p]=o(1).$ 

Now for even value of $n$, from the trace formula (\ref{trace,formula_sp}) and (\ref{trace,product_sp}), we get
\begin{align} \label{eqn:tracemultiply_even_sp}
	\Tr[ (SC'_n(t-s) )^{p}] & = \frac{1}{2n^{\frac{p}{2}-1}} \sum_{k=0}^{p}\binom{p}{k} \Big[ R_k  \sum_{ A_k} b'_{I_{k}}(t-s) + \tilde{R}_k \sum_{ \tilde{A}_k}  b'_{I_{k}}(t-s) \Big], \\
%	&= \frac{1}{n^{\frac{p}{2}-1}} \sum_{k=0}^p \binom{p}{k} (b'_{0}(t-s))^{p-k} \sum_{A_{k}} b'_{I_{k}}(t-s),  \\
%&= \sum_{m=0}^p \sum_{r=0}^{q} \binom{p}{m} \binom{q}{r} \Big[ W_k  \sum_{J_{m+r} \in A_{m+r}} X_{J_{m}}X'_{J_{r}} + \tilde{W}_k \sum_{J_{m+r} \in \tilde{A}_{m+r}}  X_{J_{m}}X'_{J_{r}} \Big] \\
	\Tr [(SC'_n(t-s))^d (SC_n(s))^{p-d}] & = \frac{1}{2n^{\frac{p}{2}-1}} \sum_{m=0}^d \sum_{r=0}^{p-d} \binom{d}{m} \binom{p-d}{r} \Big[ Q_{m,r} \sum_{ A_{m+r}} b'_{J_{m}}(t-s)b_{J_{r}}(s) \nonumber\\
	& \qquad + \tilde{Q}_{m,r} \sum_{ \tilde{A}_{m+r}}  b'_{J_{m}}(t-s)b_{J_{r}}(t-s)  \Big],	\nonumber
%	& (b'_{0}(t-s))^{d-m} (b_{0}(s))^{(p-d)-r} \sum_{A_{m+r}} b'_{J_{m}}(t-s) b_{J_{r}}(s), \nonumber
\end{align}
where  
\begin{align*}
R_{k} &=  {[b_0(t-s) +  b_{\frac{n}{2}} (t-s)]}^{p-k} + {[b_0(t-s) -  b_{\frac{n}{2}} (t-s)]}^{p-k}, \\
 \tilde{R}_{k} &= {[b_0(t-s) +  b_{\frac{n}{2}} (t-s)]}^{p-k} - {[b_0(t-s) -  b_{\frac{n}{2}}) (t-s]}^{p-k}, \nonumber \\
 Q_{m,r} & = {[b_0(t-s) +  b_{\frac{n}{2}} (t-s)]}^{d-m} {[b_0(s) +  b_{\frac{n}{2}} (s)]}^{(p-d)-r}+  {[b_0(t-s) -  b_{\frac{n}{2}} (t-s)]}^{d-m} {[b_0(s) -  b_{\frac{n}{2}} (s)]}^{(p-d)-r}, \\
\tilde{Q}_{m,r} &={[b_0(t-s) +  b_{\frac{n}{2}} (t-s)]}^{d-m} {[b_0(s) +  b_{\frac{n}{2}} (s)]}^{(p-d)-r}-  {[b_0(t-s) -  b_{\frac{n}{2}} (t-s)]}^{d-m} {[b_0(s) -  b_{\frac{n}{2}} (s)]}^{(p-d)-r},\\
b'_{J_{k}}(t-s) & = (b_{j_1}(t)-b_{j_1}(s)) (b_{j_2}(t) - b_{j_2}(s)) \cdots (b_{j_{k}}(t) - b_{j_{k}}(s)), \\
b'_{J_{m}}(t-s) & = (b_{j_1}(t)-b_{j_1}(s)) (b_{j_2}(t) - b_{j_2}(s)) \cdots (b_{j_{m}}(t) - b_{j_{m}}(s)), \\
b_{J_{r}}(s) & = b_{j_{m+1}}(s) b_{j_{m+2}}(s)\cdots b_{j_{m+r}}(s).
\end{align*}
Using (\ref{eqn:tracemultiply_even_sp}) and (\ref{eqn:same distrib_odd}) in (\ref{eqn:trace t-s_sp}), we get
\begin{align} \label{eqn:Tr,t-s_even,1}
 \Tr (SC_n(t))^{p}- \Tr (SC_n(s))^{p}  & \distas{D}  
% \frac{1}{2n^{\frac{p}{2}-1}} \Biggl\{ \sum_{k=0}^p  (t-s)^{\frac{p-k}{2} + \frac{k}{2}} \binom{p}{k} \Big[ R'_k \sum_{ A_k}  X_{J_{k}} + \tilde{R'}_k \sum_{ \tilde{A}_k}   X_{J_{k}} \Big]  \\
% & \qquad + \sum_{d=1}^{p-1} \binom{p}{d}  \Big[ \sum_{m=0}^d \sum_{r=0}^{p-d} \binom{d}{m} \binom{p-d}{r}  (t-s)^{\frac{d-m}{2} + \frac{m}{2}} {s}^{\frac{p-d-r}{2} + \frac{r}{2}}  \nonumber \\
% & \qquad \Big( Q'_{m,r} \sum_{ A_{m+r}} X_{J_{m}}Y_{J_{r}}  + \tilde{Q}_{m,r} \sum_{ \tilde{A}_{m+r}}  X_{J_{m}} Y_{J_{r}}  \Big) \Big] \Biggr\} \nonumber \\
  \frac{\sqrt{t-s}}{2n^{\frac{p}{2}-1} } \Biggl\{ \sum_{k=0}^p  (t-s)^{\frac{p-1}{2}} \binom{p}{k} \Big[ R'_k \sum_{ A_k}  X_{J_{k}} + \tilde{R'}_k \sum_{ \tilde{A}_k}   X_{J_{k}} \Big]  \\
 &\qquad + \sum_{d=1}^{p-1} \binom{p}{d}  \Big[ \sum_{m=0}^d \sum_{r=0}^{p-d} \binom{d}{m} \binom{p-d}{r}  
  (t-s)^{\frac{d-1}{2}} {s}^{\frac{p-d}{2}}  \nonumber \\
 & \qquad \Big( Q'_{m,r} \sum_{ A_{m+r}} X_{J_{m}}Y_{J_{r}}  + \tilde{Q'}_{m,r} \sum_{ \tilde{A}_{m+r}}  X_{J_{m}} Y_{J_{r}}  \Big)  \Big] \Biggr\}, \nonumber 
% & = \frac{\sqrt{t-s}}{n^{p-1}}\sum_{A_{2p}} Z_{I_{2p}}, \mbox{ say}.
\end{align} 
where 
\begin{align*}
R_k & \distas{D}  (t-s)^{\frac{p-k}{2}} R'_k , \ \ \tilde{R}_k \distas{D} (t-s)^{\frac{p-k}{2}} \tilde{R'}_k, \\
 Q_{m,r} & \distas{D}  (t-s)^{\frac{d-m}{2}} (s)^{\frac{p-d-r}{2}} Q'_{m,r},
\ \ \tilde{Q}_{m,r} \distas{D} (t-s)^{\frac{d-m}{2}} (s)^{\frac{p-d-r}{2}} \tilde{Q'}_{m,r}. 
\end{align*}
Observe from the definition of $A_k$ and $\tilde{A}_k$, (\ref{def:A_p_sp}) and (\ref{eqn:A,tildeA}), respectively that if $F_i=A_i$ or $\tilde{A}_i$, then
\begin{align} \label{eqn:E,o(1)_even}
\E \Big( \frac{1}{n^{\frac{p}{2}-1}} \sum_{F_{k}} X_{J_{k}} \Big)
&= \left\{\begin{array}{ll} 
		 	O(1) & \text{if} \   k=p  \mbox{ and } F_k=A_k\\\\
		 	 o(1) & \text{otherwise}, 	 
		 	  \end{array}\right. \\\
 \E \Big( \frac{1}{n^{\frac{p}{2}-1}} \sum_{F_{m+r}} X_{J_{m}}  Y_{J_{r}}  \Big)
&= \left\{\begin{array}{ll} 
		 	O(1) & \text{if} \   m+r=p  \mbox{ and } F_{m+r}=A_{m+r}  \\\\
		 	 o(1) & \text{otherwise}. 	 
		 	  \end{array}\right. \nonumber
\end{align}
So, by using (\ref{eqn:E,o(1)_even}) in (\ref{eqn:Tr,t-s_even,1}), we get
\begin{align} \label{eqn:W_p,even_sp}
 [\Tr (SC_n(t))^{p}- \Tr (SC_n(s))^{p}] \mathbb I_{\{ n \ even\}} & \distas{D}  
  \sqrt{t-s} \Big[  \frac{ (t-s)^{\frac{p-1}{2}}} {n^{\frac{p}{2}-1 }}  \sum_{A_{p}} X_{J_{p}} + \frac{1}{n^{\frac{p}{2}-1}} \sum_{d=1}^{p-1} \binom{p}{d} \\
& \qquad  (t-s)^{\frac{d-1}{2}} {s}^{\frac{p-d}{2}} \sum_{A_{p}} X_{J_{d}}  Y_{J_{p-d}} + \theta'_p \Big], \nonumber
% & = \sqrt{t-s} \Big[ \frac{1}{n^{\frac{p}{2}-1 }} \sum_{A_{p}} W_{I_{p}} +\theta_p\Big], \mbox{ say},
\end{align} 
where $\E[ \theta'_p]=o(1).$ 
Now on combining (\ref{eqn:W_p,odd_sp}) and (\ref{eqn:W_p,even_sp}), we get
\begin{align} \label{eqn:W_p_sp}
 \Tr (SC_n(t))^{p}- \Tr (SC_n(s))^{p}  & \distas{D}  
  \sqrt{t-s} \Big[  \frac{ (t-s)^{\frac{p-1}{2}}} {n^{\frac{p}{2}-1 }}  \sum_{A_{p}} X_{J_{p}} + \frac{1}{n^{\frac{p}{2}-1}} \sum_{d=1}^{p-1} \binom{p}{d} 
  (t-s)^{\frac{d-1}{2}} {s}^{\frac{p-d}{2}} \sum_{A_{p}} X_{J_{d}}  Y_{J_{p-d}} + \phi_p \Big], \nonumber \\
 & = \sqrt{t-s} \Big[ \frac{1}{n^{\frac{p}{2}-1 }} \sum_{A_{p}} W_{I_{p}} +\phi_p\Big], \mbox{ say},
\end{align} 
where $\E[ \phi_p]=o(1).$

Finally, by using (\ref{eqn:W_p_sp}) in (\ref{eqn:w_p(t-s)_sp}), we get
\begin{align*}
 \eta_{p}(t) - \eta_{p}(s) \distas{D} \sqrt{t-s} \Big[ \frac{1}{n^{\frac{p-1}{2} }} \sum_{A_{p}} ( W_{I_{p}} -\E[W_{I_{p}}])+\phi_p- \E[\phi_p]\Big],
\end{align*}
and therefore 
\begin{align} \label{eqn:w_p^4_sp}
\E[\eta_{p}(t) - \eta_{p}(s)]^4 & = (t-s)^2\Bigl( \frac{1}{n^{2p-2 }} \sum_{A_{p}, A_{p}, A_{p}, A_{p}} \E \Big[ \prod_{k=1}^4 (W_{J^k_{p}} - \E[ W_{J^k_{p}}] ) \Big] +o(1) \Bigr),
\end{align}
where $J^k_{p}$ are vectors from $A_{p}$, for each $k=1, 2, 3, 4$. Depending on connectedness between $J^k_{p}$'s, 
%we divide the right side of \eqref{eqn:w_p^4} into three parts, $I_1,I_2, I_3$. 
the following three cases arise:
\vskip5pt
\noindent \textbf{Case I.} \textbf{At least of $J^k_{p}$ for $k=1,2,3,4$, is not connected with the remaining ones:}  By arguments similar to those in Case I of proof of Theorem \ref{thm:revprocess}, we get
$$\E \Big[ \prod_{k=1}^4 (Z_{J^k_{p}} - \E Z_{J^k_{p}}) \Big]= 0,$$ 
 and hence there exists $M_1 >0$, depending only on $p, T$ such that 
\begin{align} \label{case I,(t-s)_sp}
%\E[w_{p}(t) -w_{p}(s)]^4 
\E[\eta_{p}(t) -\eta_{p}(s)]^4 \leq M_1 (t-s)^2 \ \mbox{ for all }n\geq 1.
\end{align} 
\vskip5pt
 \noindent \textbf{Case II.} \textbf{$J^1_{p}$ is connected with one of $J^2_{p}, J^3_{p}, J^4_{p}$ only and the remaining two of $J^2_{p}, J^3_{p}, J^4_{p}$  are also connected with themselves only:} Without loss of generality, we assume $J^1_{p}$ is connected with $J^2_{p}$ only and $J^3_{p}$ is connected with $J^4_{p}$ only.
 Therefore from independence, the right side of (\ref{eqn:w_p^4_sp}) will be
\begin{align} \label{eqn:case II_sp} 
%\E[w_{2p}(t) - w_{2p}(s)]^4 & = 
\frac{(t-s)^2}{n^{2p-2}} &\sum_{ A_{p}, A_{p}} \E \big[ (W_{J^1_{p}} - \E W_{J^1_{p}}) (W_{J^2_{p}} - \E W_{J^2_{p}}) \big]  \\ 
& \qquad \qquad \qquad \qquad \qquad \sum_{A_{p}, A_{p}} \E \big[ (W_{J^3_{p}} - \E W_{J^3_{p}}) (W_{J^4_{p}} - \E W_{J^4_{p}}) \big] + (t-s)^2 o(1) \nonumber.
\end{align}
We denote the above expression by $S_2$. Now by the arguments as given in Case II of proof of Theorem \ref{thm:revprocess}, we get
\begin{align} \label{eqn:case2, B_2}
\sum_{ A_{p}, A_{p}} \Big|\E \big[ (W_{J^1_{p}} - \E W_{J^1_{p}}) (W_{J^2_{p}} - \E W_{J^2_{p}}) \big] \Big| \leq \sum_{( J^1_{p}, J^2_{p}) \in B_{P_2}} \tilde{\alpha} =  \tilde{\alpha} |B_{P_2}|,
\end{align} 
where $B_{P_2}$ is as defined in Definition \ref{def:B_{P_l}_sp}. Note that
\begin{align*}B_{P_2}=\{( J^1_{p}, J^2_{p}) \in  A_{p} \times A_{p} &: S_{J^1_{p}} \cap S_{J^2_{p}} \neq \emptyset \mbox{ and each entries of } S_{J^1_{p}} \cup S_{J^2_{p}}\\
&\quad \mbox{ has multiplicity greater than or equal to two}\}.
\end{align*}
%Now we calculate cardinality of $B_{P_2}$. First recall
Therefore
\begin{align*}
 |B_{P_2}|  & = \left\{\begin{array}{ccc} 	 
		 	 O({n}^{p-1}) & \text{if}&  p \mbox{ and } |S_{J^1_{p}} \cap S_{J^2_{p}}| \mbox{ both are even or both are odd}, \\
			o({n}^{p-1}) & \text{otherwise},& 	 	 
		 	  \end{array}\right.
\end{align*}
and hence $|B_{P_2}|= O({n}^{p-1})$. So, if we use the above observation in (\ref{eqn:case2, B_2}), we get
\begin{align} \label{eqn:W_I,J_sp}
\sum_{ A_{p}, A_{p}} \Big|\E \big[(W_{J^1_{p}} - \E W_{J^1_{p}}) (W_{J^2_{p}} - \E W_{J^2_{p}}) \big] \Big| \leq \tilde{\alpha} O({n}^{p-1}).
\end{align} 
%Similarly, we have 
%\begin{align} \label{eqn:Z_K,L}
%\sum_{ A_{p}, A_{p}} \Big|\E \big[ (Z_{J^3_{p}} - \E Z_{J^3_{p}}) (Z_{J^4_{p}} - \E Z_{J^4_{p}}) \big] I(i_3, J^3_p) I(i_4, J^4_p)\Big| \leq \alpha O({b_n}^{p-1}).
%\end{align}
Finally from (\ref{eqn:W_I,J_sp}) and (\ref{eqn:case II_sp}), we get
\begin{align*}
%\E[w_{p}(t) -w_{p}(s)]^4 
|S_2|\leq (t-s)^2 \tilde{\alpha}^2 O(1) +(t-s)^2 o(1) = (t-s)^2 [ \tilde{\alpha}^2 O(1)+ o(1)].
\end{align*}
Since $t, s \in [0,T]$,  there exist $M_2 >0$, depending only on $p, T$ such that 
\begin{align} \label{case II,(t-s)_sp}
%\E[w_{p}(t) -w_{p}(s)]^4 
|S_2|\leq M_2 (t-s)^2 \ \mbox{ for all }n\geq 1.
\end{align} 
\vskip5pt
 \noindent \textbf{Case III.} \textbf{$J^1_{p}, J^2_{p}, J^3_{p}, J^4_{p}$ forms a cluster:}  
In this case we get
\begin{align*}
\sum_{ A_{p}, A_{p}, A_{p}, A_{p}} & \E \big[ \prod_{k=1}^4 (W_{J^k_{p}} - \E W_{J^k_{p}}) \big] 
 = \sum_{( J^1_{p}, J^2_{p}, J^3_{p}, J^4_{p}) \in B_{P_4}} \E \big[ \prod_{k=1}^4 (W_{J^k_{p}} - \E W_{J^k_{p}}) \big],
\end{align*}
 where $B_{P_4}$ as in Definition \ref{def:B_{P_l}_sp} for $\ell=4$ and $p_i=p$ for $i=1,2,3,4$. By the similar arguments as given  in Case II, there exists $\beta>0$ such that 
\begin{align*}
\sum_{ A_{p}, A_{p}, A_{p}, A_{p}} \Big|\E\big[\prod_{k=1}^4 (W_{J^k_{p}} - \E W_{J^k_{p}})\big]  \Big| & \leq \sum_{( J^1_{p}, J^2_{p}, J^3_{p}, J^4_{p}) \in B_{P_4}} \tilde{\beta}
 = \tilde{\beta} |B_{P_4}|.
\end{align*} 
Since from Result \ref{res:|B_{P_l}|_sp}, we have $|B_{P_4}| = o({b_n}^{2p-2})$, and therefore  
 \begin{align} \label{eqn:case III,W_sp}
\frac{(t-s)^2}{n^{2p-2}}  \sum_{ A_{p}, A_{p}, A_{p}, A_{p}} \Big|\E\big[ \prod_{k=1}^4 (W_{J^k_{p}} - \E W_{J^k_{p}})\big] \Big| &\leq \frac{(t-s)^2}{n^{2p-2}}  \tilde{\beta} o(n^{2p-2})\nonumber\\
& = \tilde{\beta} (t-s)^2  o(1).
 \end{align}
Now using (\ref{eqn:case III,W_sp}) in (\ref{eqn:w_p^4_sp}), we get
\begin{align*}
\E[\eta_{p}(t) -\eta_{p}(s)]^4 \leq \tilde{\beta} (t-s)^2 o(1)+ (t-s)^2 o(1)= (t-s)^2[\tilde{\beta} o(1)+ o(1)],
\end{align*}
and hence there exist $M_3 >0$, depending only on $p, T$ such that 
\begin{align} \label{case III,(t-s)_sp}
\E[\eta_{p}(t) -\eta_{p}(s)]^4 \leq M_3 (t-s)^2\  \mbox{ for all }n\ge 1.
\end{align}
Now on combining \eqref{case I,(t-s)_sp}, \eqref{case II,(t-s)_sp} and \eqref{case III,(t-s)_sp}, we get that there exist a positive constant $M_T$, depending only on $p, T$ such that
\begin{align} \label{eqn:M_T_sp}
%\E[w_p(t) - w_p(s)]^4 & \leq M^T_7 (t-s)^2 + M^T_{10} (t-s)^2 \nonumber \\
\E[\eta_{p}(t) -\eta_{p}(s)]^4 & \leq M_T (t-s)^2 \ \ \ \forall \ n \in \mathbb{N} \ \mbox{and } t,s \in[0,\ T]. 
\end{align}
This complete the proof of Theorem \ref{thm:symprocess}.
\end{proof}

\begin{remark} 
	Note that the constant $M_T$ of (\ref{eqn:M_T_sp}) depends on $p$. So there is a possibility that it tends to infinity as $p \tends \infty$. Therefore, from Theorem \ref{thm:symprocess} we cannot conclude anything about the process convergence of $\{ \eta_{p}(t) ; t \geq 0, p \geq 2\}$. But the proof of Theorem \ref{thm:symprocess} yields
		\begin{equation*}
	\{ \eta_p(t) ; t \geq 0, 2 \leq p \leq N\} \stackrel{\mathcal D}{\rightarrow} \{N_p(t) ; t \geq 0 , 2 \leq p \leq N\}, 
	\end{equation*}  
	for any fixed $N \in \mathbb{N}$.
\end{remark}

\begin{remark} \label{rem:poly_test_function_sp}
For a given real polynomial $Q(x)=\sum_{k=1}^da_kx^{k}$
with degree $d$, $d\geq 2$. If we define
$$\eta_Q(t) := \frac{1}{\sqrt{n}} \bigl\{ \Tr(Q(SC_n(t))) - \E[\Tr(Q(SC_n(t)))]\bigr\}.$$
Then we can extend our results, Theorem \ref{thm:symcovar_sp}, \ref{thm:symmulti_sp} and \ref{thm:symprocess} for $\eta_Q(t)$. The proofs will be similar to the proofs of Theorem \ref{thm:symcovar_sp}, \ref{thm:symmulti_sp} and \ref{thm:symprocess}, respectively.
\end{remark}

%\bibliography{shambhubib}
\bibliographystyle{amsplain}

\providecommand{\bysame}{\leavevmode\hbox to3em{\hrulefill}\thinspace}
\providecommand{\MR}{\relax\ifhmode\unskip\space\fi MR }
\providecommand{\MRhref}[2]{%
  \href{http://www.ams.org/mathscinet-getitem?mr=#1}{#2}
}
\providecommand{\href}[2]{#2}

\end{document}